\documentclass[11pt]{article}
\usepackage{amsfonts,amssymb,amsmath}
\usepackage{amsthm, verbatim}
\usepackage{fullpage}
\usepackage{algorithm}
\usepackage[noend]{algpseudocode}
\usepackage[utf8]{inputenc}
\usepackage[firstinits=true,backend=bibtex, doi=false,isbn=false,url=false,eprint=false,maxbibnames=10]{biblatex} 
\renewbibmacro{in:}{\ifentrytype{article}{}{\printtext{\bibstring{in}\intitlepunct}}} 

\DeclareFieldFormat{sentencecase}{\MakeSentenceCase{#1}}

\renewbibmacro*{title}{%
  \ifthenelse{\iffieldundef{title}\AND\iffieldundef{subtitle}}
    {}
    {\ifthenelse{\ifentrytype{article}\OR\ifentrytype{inbook}%
      \OR\ifentrytype{incollection}\OR\ifentrytype{inproceedings}%
      \OR\ifentrytype{inreference}}
      {\printtext[title]{%
        \printfield[sentencecase]{title}%
        \setunit{\subtitlepunct}%
        \printfield[sentencecase]{subtitle}}}%
      {\printtext[title]{%
        \printfield[titlecase]{title}%
        \setunit{\subtitlepunct}%
        \printfield[titlecase]{subtitle}}}%
     \newunit}%
  \printfield{titleaddon}}

\newcommand{\Ref}[1]{#1}
\bibliography{\Ref{jabref_library},\Ref{Peter}}

\usepackage{tikz}
\usetikzlibrary{decorations.pathreplacing,calc,positioning,arrows,shapes}
\providecommand{\nodepoint}[3]{\node[inner sep =0mm, outer sep =0mm] (#1) at (#2,#3) {\huge\bf$\cdot$}}
\usepackage{graphicx}
\usepackage{xcolor}
\usepackage{caption}
\usepackage{subcaption}
\usepackage[normalem]{ulem}
\usepackage[makeroom]{cancel}
\graphicspath{{./figures/}}  
\newcommand{\R}{\mathbb{R}}

 \newcommand{\eqdef}{\overset{\text{def}}{=}} 
\providecommand{\Null}[1]{\mathbf{Null}\left( #1\right)}
\providecommand{\Rank}[1]{\mathbf{Rank}\left( #1\right)}

\providecommand{\Tr}[1]{\mathbf{Tr}( #1)}

\providecommand{\E}[1]{\mathbf{E}\left[ #1\right]}
\providecommand{\Esmall}[1]{\mathbf{E}[ #1]}

 \providecommand{\dotprod}[1]{\langle #1\rangle} 
 
\providecommand{\proj}[2]{Z}
\providecommand{\norm}[1]{\left\lVert#1\right\rVert}

\providecommand{\myRange}[1]{\mathbf{Range}\left( #1\right)}
\usepackage[colorinlistoftodos,bordercolor=orange,backgroundcolor=orange!20,linecolor=orange,textsize=scriptsize]{todonotes}

\newtheorem{remark}{Remark}[section]
\newtheorem{proposition}{Proposition}[section]

\newtheorem{lemma}{Lemma}[section]

\newtheorem{definition}{Definition}[section]
\newtheorem{theorem}{Theorem}[section]

\title{Randomized  Quasi-Newton Updates  are  Linearly Convergent  Matrix Inversion Algorithms}

\author{Robert M. Gower and Peter Richt\'{a}rik\thanks{This author would like to acknowledge support from the EPSRC Grant EP/K02325X/1, {\em Accelerated Coordinate Descent Methods for Big Data Optimization} and the EPSRC Fellowship EP/N005538/1, {\em Randomized Algorithms for Extreme Convex Optimization.} }\\\\
\it School of Mathematics\\ \it University of Edinburgh\\\it United Kingdom}

\date{February 5, 2016 \footnote{Slightly updated on March 21, 2016. } }

\begin{document}

\maketitle

\begin{abstract}
We develop and analyze a broad family of stochastic/randomized algorithms for inverting a matrix. We also develop specialized variants  maintaining symmetry or positive definiteness of the iterates. All methods in the family  converge globally and linearly (i.e., the error decays exponentially), with explicit rates. In special cases, we obtain stochastic block variants of several quasi-Newton updates, including   bad Broyden (BB), good Broyden (GB),  Powell-symmetric-Broyden (PSB), Davidon-Fletcher-Powell (DFP) and Broyden-Fletcher-Goldfarb-Shanno (BFGS). Ours are the first stochastic versions of these updates shown to converge to an inverse of a fixed matrix. Through a dual viewpoint we uncover a fundamental link between quasi-Newton updates and approximate inverse preconditioning.  Further, we develop an adaptive variant  of randomized block BFGS, where we modify the distribution underlying the stochasticity of the method throughout the iterative process to achieve faster convergence. By inverting several matrices from varied applications, we demonstrate that AdaRBFGS is highly competitive when compared to the well established Newton-Schulz and minimal residual methods.  In particular, on large-scale problems our method outperforms the standard methods by orders of magnitude. Development of efficient methods for estimating the inverse of very large matrices is a  much needed tool for preconditioning and variable metric optimization methods in the advent of the big data era.
\end{abstract}

\newpage
\tableofcontents
\newpage

\section{Introduction}

Matrix inversion is a standard tool in numerics, needed, for instance, in computing a projection matrix or a Schur complement, which are common place calculations.  When only an approximate inverse is required, then iterative methods  are the methods of choice, for they can terminate the iterative process when the desired accuracy is reached. This can be far more efficient than using a direct method. Calculating an approximate inverse is a much needed tool in preconditioning~\cite{Saad2003} and, if the output is guaranteed to be positive definite, then the it can be  used to design variable metric optimization methods. Furthermore, iterative methods can make use of an initial estimate of the inverse when available.

The driving motivation of this work is the need to develop algorithms capable of computing the inverse of very large matrices, where standard techniques take an exacerbating amount of time or simply fail. In particular, we develop a family of randomized/stochastic methods for inverting a matrix, with  specialized variants maintaining symmetry or positive definiteness of the iterates. All methods in the family  converge globally (i.e., from any starting point) and linearly (i.e., the error decays exponentially).  We give an explicit expression for the convergence rate.

  As special cases, we obtain stochastic block variants of several quasi-Newton (qN) updates, including  bad Broyden (BB), good Broyden (GB),  Powell-symmetric-Broyden (PSB), Davidon-Fletcher-Powell (DFP) and Broyden-Fletcher-Goldfarb-Shanno (BFGS). To the best of our knowledge, these are the first stochastic versions of qN updates. Moreover, this is the first time that  qN updates are shown to be iterative methods for inverting a matrix. We offer a new interpretation of the qN methods through a Lagrangian dual viewpoint, uncovering a fundamental link between qN updates and approximate inverse preconditioning.  
  
  We develop an adaptive variant  of randomized block BFGS, in which we modify the distribution underlying the stochasticity of the method throughout the iterative process to achieve faster convergence. Through extensive numerical experiments with large matrices arising  from several applications, we show that AdaRBFGS can significantly outperform  the well established Newton-Schulz and minimal residual methods.

 \subsection{Outline} The rest of the paper is organized as follows. 
In Section~\ref{sec:contributions} we summarize the main contributions of this paper. In Section~\ref{sec:QN} we describe the qN methods, which is the main inspiration of our methods. Subsequently, Section~\ref{sec:SIMI-nonsym}  describes two  algorithms, each corresponding to a variant of the inverse equation,  for inverting general square matrices. We also provide insightful dual viewpoints for both methods. In Section~\ref{sec:SIMI-sym} we describe a method specialized to inverting symmetric matrices. Convergence in expectation is examined in Section~\ref{sec:conv}, were we consider two types of convergence: the convergence of i) the expected norm of the error, and the convergence of ii) the norm of the expected error. In Section~\ref{sec:discrete} we specialize our methods to discrete distributions, and comment on how one may construct a probability distribution leading to better complexity rates (i.e., importance sampling). We then describe a convenient probability distribution which leads to  convergence rates which can be described in terms of spectral properties of the original matrix to be inverted. In Section~\ref{sec:discretemethods} we detail several instantiations of our family of methods, and their resulting convergence rates. We show how via the choice of the parameters of the method,  we obtain {\em stochastic block variants} of several well known quasi Newton methods.
We also describe the simultaneous randomized Kaczmarz method here.
Section~\ref{sec:AdaRBFGS} is dedicated to the development of an adaptive variant of our randomized BFGS method, AdaRBFS, for inverting positive definite matrices.  This method adapts stochasticity throughout the iterative process to obtain faster practical convergence. Finally, in Section~\ref{sec:numerical} we show through numerical tests that AdaRBFGS significantly outperforms state-of-the-art  iterative matrix inversion methods on large-scale matrices.




\subsection{Notation} \label{subsec:notation}

By $I $ we denote the $n\times n$ identity matrix. Let
$\dotprod{X,Y}_{F(W^{-1})} \eqdef \Tr{X^T W^{-1}YW^{-1}}$
denote the weighted Frobenius inner product,
where $X,Y \in \R^{n\times n}$ and $W \in \R^{n\times n}$ is a symmetric positive definite ``weight'' matrix. Further, let
\begin{equation}\label{eq:98y988ff}
\norm{X}_{F(W^{-1})}^2 \eqdef \Tr{X^TW^{-1}X W^{-1}} 
= \|W^{-1/2}XW^{-1/2}\|_{F}^2,\end{equation}
where we have used the convention $F=F(I)$, since  $\|\cdot\|_{F(I)}$ is the standard Frobenius norm. Let $\|\cdot\|_{2}$ denote the induced operator norm for square matrices defined via $\|Y\|_{2} \eqdef \max_{\|v\|_2=1} \|Yv\|_2.$ Finally, for positive definite $W \in \R^{n\times n}$, we define the weighted induced norm via 
\[\|Y\|_{W^{-1}} \eqdef  \|W^{-1/2}YW^{-1/2}\|_2.\]

\subsection{Previous work}\label{sec:previous}

A widely used iterative method for inverting matrices is the Newton-Schulz method~\cite{Schulz1933} introduced in 1933, and its variants which is still subject of ongoing research~\cite{Li2010}. The drawback of the Newton-Schulz methods is that they do not converge for any initial estimate. Instead,  an initial estimate that is close to $A^{-1}$ (in some norm) is required. In contrast, the methods we present converge globally for any initial estimate. Bingham~\cite{Bingham1941} describes a method that uses the characteristic polynomial to recursively calculate the inverse, though it requires the calculating the coefficients of the polynomial when initiated, which is costly, and the method has fallen into disuse.  Goldfarb~\cite{Goldfarb1972} uses Broyden's method~\cite{Broyden1965} for iteratively inverting matrices. Our methods include a stochastic variant of Broyden's method.
 
 The approximate inverse preconditioning (\emph{AIP}) methods~\cite{Chow1998,Saad2003,Gould1998,Benzi1999} calculate an approximate inverse by minimizing the residual $\norm{XA-I}_F$ in $X$. s by applying a number of iterations of the steepest descent or minimal residual method.     A considerable drawback of the AIP methods is that the iterates are not guaranteed to be positive definite nor symmetric, even when $A$ is both. A solution to the lack of symmetry is to ``symmetrize'' the estimate between iterations. However,  then it is difficult to guarantee the quality of the new symmetric estimate.
Another solution is to calculate directly a factored form $LL^T=X$ and minimize in $L$ the residual $\norm{L^TAL-I}_F$. But now this residual is a non-convex function, and is thus difficult to minimize. A variant of our method naturally maintains symmetry of the iterates.

\section{Contributions} \label{sec:contributions}

We  now describe the main contributions of this work.

\subsection{New algorithms} \label{subsec:primal}
We develop a novel and surprisingly simple  family of stochastic algorithms for inverting matrices. The problem of finding the inverse of an $n\times n$ invertible matrix $A$ can be characterized as finding the solution to either one of the two {\em inverse equations}\footnote{One may use other equations uniquely defining the inverse, such as $AXA = A$, but we do not explore these in this paper.} $AX=I$ or $XA=I.$
Our methods make use of randomized sketching~\cite{Pilanci2014,Gower2015,Pilanci2015,Pilanci2015a} to reduce the dimension of the inverse equations in an iterative fashion. To the best of our knowledge, these are the first stochastic algorithms for inverting a matrix with global complexity rates.

In particular, our nonsymmetric  method (Algorithm~\ref{alg:asym-row}) is based on the inverse equation $AX=I$,  and performs the  {\em sketch-and-project} iteration
\begin{equation}\label{eq:98hs8s}X_{k+1} = \arg \min_{X\in \R^{n\times n} } \frac{1}{2}\norm{X - X_{k}}_{F(W^{-1})}^2 \quad \mbox{subject to } \quad  S^TAX =S^T,\end{equation}
where $S\in \R^{n\times q}$ is a random matrix drawn in an i.i.d.\ fashion from a fixed distribution $\cal{D}$, and  $W\in \R^{n\times n}$ is the positive definite ``weight'' matrix.
 The distribution $\cal D$ and matrix $W$ are the parameters of the method. Note that if we choose $q\ll n$, the constraint in the projection problem  \eqref{eq:98hs8s} will be of a much smaller dimension than the original inverse equation, and hence the iteration \eqref{eq:98hs8s} will become cheap.

In an analogous way, we   design a method based on the inverse equation $XA=I$ (Algorithm~\ref{alg:asym-col}). Adding the symmetry constraint $X=X^T$ leads to Algorithm~\ref{alg:sym}---a specialized method for  symmetric $A$ capable of maintaining symmetric iterates.

\subsection{Dual formulation}

  Besides the {\em primal formulation} described in Section~\ref{subsec:primal}---\emph{sketch-and-project}---we also provide  {\em dual formulations} of all three  methods (Algorithms~\ref{alg:asym-row}, \ref{alg:asym-col} and \ref{alg:sym}).  For instance, the dual formulation of \eqref{eq:98hs8s} is \begin{equation}\label{eq:98hs8sssds}X_{k+1} = \arg_X \min_{X,Y} \frac{1}{2}\norm{X_{k}-A^{-1}}_{F(W^{-1})}^2 \quad \mbox{subject to} \quad  X = X_k + WA^TSY^T,\end{equation}
where the minimization is performed over $X\in \R^{n\times n}$ and $ Y\in \R^{n\times q}$.  We call the dual formulation  \emph{constrain-and-approximate} as one seeks to perform the best approximation of the inverse (with respect to the weighted Frobenius distance) while constraining the search to  a random affine space of matrices passing through $X_k$. While the projection \eqref{eq:98hs8sssds} cannot be performed directly since $A^{-1}$ is not known, it can be performed indirectly via the equivalent primal formulation \eqref{eq:98hs8s}.

\subsection{Quasi-Newton updates and approximate inverse preconditioning}

As we will discuss in Section~\ref{sec:QN}, through the lens of the sketch-and-project formulation, Algorithm~\ref{alg:sym} can be seen as {\em randomized block extension of the quasi-Newton (qN) updates}~\cite{Broyden1965,Fletcher1960,Goldfarb1970,Shanno1971}. 
We distinguish here between qN methods, which are algorithms  used in optimization, and qN updates, which are  {\em matrix-update} rules used in qN methods.   Standard qN updates work with $q=1$ (``block'' refers to the choice $q>1$) and $S$ chosen in a deterministic way, depending on the sequence of iterates of the underlying optimization problem.  To the best of our knowledge, this is the first time stochastic versions of qN updates were designed and analyzed. On the other hand,  through the lens of the constrain-and-approximate formulation, our methods can be seen as {\em new variants of the approximate inverse preconditioning (\emph{AIP}) methods}~\cite{Chow1998,Saad2003,Gould1998,Benzi1999}. Moreover, the equivalence between these two formulations reveals deep connections  between what were before seen as distinct fields: the qN and AIP literature.  Our work also provides several new insights for {\em deterministic} qN updates. For instance, the {\em bad Broyden update}~\cite{Broyden1965,Griewank2012} is a particular best rank-1 update that minimizes the distance to the inverse of $A$ under the Frobenius norm. 
The {\em BFGS update}~\cite{Broyden1965,Fletcher1960,Goldfarb1970,Shanno1971} can be seen as a projection of $A^{-1}$ onto a space of rank-2 symmetric matrices. It seems this has not been observed before.

\subsection{Complexity: general results} Our framework leads to global linear convergence (i.e., exponential decay) under very weak assumptions on $\cal D$.  In particular, we provide an explicit convergence rate $\rho$ for the exponential decay of the norm of the expected error of the iterates (line~2 of Table~\ref{tab:complexity}) and the expected norm of the error (line~3 of Table~\ref{tab:complexity}), where the rate is given by
\begin{equation}\label{eq:rho}
\rho = 1- \lambda_{\min}(W^{1/2}\E{Z}W^{1/2}),
\end{equation}
where  $Z \eqdef A^TS(S^T A WA^T S)^{-1}S A^T$. We show that  $\rho$ is always bounded between $0$ and $1$. Furthermore, we provide a lower bound on $\rho$ that shows that the rate can potentially improve as the number of columns in $S$ increases. This sets our method apart from current methods for inverting matrices that lack global guarantees, such as Newton-Schulz, or the self-conditioning variants of the minimal residual method.
\begin{table}
\centering
\begin{tabular}{|c|c|}
\hline
& \\
\footnotesize $\E {X_{k+1} -A^{-1}} = \left(I  - W\E{Z}\right) \E{X_{k+1} - A^{-1}}  $ & \footnotesize Theorem 4.1\\
 & \\
\footnotesize $\norm{\E {X_{k+1} -A^{-1}}}_{W^{-1}}^2 \leq \rho^2 \; \cdot \; \norm{\E{X_{k+1} - A^{-1}}}_{W^{-1}}^2
 $ & \footnotesize Theorem~\ref{theo:normEconv}\\
 & \\
\footnotesize $ \E {\norm{X_{k+1} -A^{-1} }_{F(W^{-1})}^2 } \leq \rho \;\cdot\; \E{ \norm{X_{k+1} - A^{-1}}_{F(W^{-1})}^2}$  & \footnotesize Theorem~\ref{theo:Enormconv}\\
& \\
 \hline
 \end{tabular}
 \caption{\footnotesize Our main complexity results.}
 \label{tab:complexity}
 \end{table}
 
\subsection{Complexity: discrete distributions}
 
We detail a convenient choice of probability for discrete distributions $\cal D$ that gives easy-to-interpret convergence results depending on a scaled condition number of $A$. This way we obtain methods for inverting matrices with the same convergence rate as the randomized Kaczmarz method~\cite{Strohmer2009} and randomized coordinate descent~\cite{Leventhal2010} for solving linear systems. We also obtain importance sampling results by optimizing an upper bound on the convergence rate.

\subsection{Adaptive randomized BFGS}
  
We develop an additional highly efficient method---adaptive randomized BFGS (AdaRBFGS)---for calculating an approximate inverse of {\em positive definite matrices}. Not only does the method greatly outperform the state-of-the-art methods such as Newton-Schulz and approximate inverse preconditioning methods, but it also preserves positive definiteness, a quality not present in previous methods. Therefore, AdaRBFGS can be used to precondition positive definite systems and to design new variable-metric optimization methods.  Since the inspiration behind this method comes from the desire to design an {\em optimal adaptive} distribution for $S$ by examining the complexity rate $\rho$, this work  also highlights the importance of developing algorithms with explicit convergence rates.

%
%

\subsection{Extensions}  This work opens up many possible avenues for extensions. For instance, new efficient methods could be designed by experimenting and analyzing through our framework with different sophisticated sketching matrices $S$, such as the Walsh-Hadamard matrix~\cite{Lu2013,Pilanci2014}. Furthermore, our methods produce low rank estimates of the inverse and can be adapted to calculate low rank estimates of any matrix. They can be applied to singular matrices, in which case they converge to a particular pseudo-inverse.   Our results can be used to push forward work into stochastic variable metric optimization methods, such as the work by Leventhal and Lewis~\cite{Leventhal2011}, where they present a randomized iterative method for estimating Hessian matrices that converge in expectation with known convergence rates for any initial estimate. Stich et al.\ \cite{Stich2015} use Leventhal and Lewis' method to design a stochastic variable metric method for black-box minimization, with explicit convergence rates, and promising numeric results. We leave these and other extensions to future work.

\section{Randomization of Quasi-Newton Updates}\label{sec:QN}

Our methods are inspired by, and in some cases can be considered to be, randomized block variants of the quasi-Newton (qN) updates. Here we explain how our algorithms arise naturally from the qN setting. Readers familiar with  qN methods may  jump ahead to Section~\ref{subsec:iuh9898}.

\subsection{Quasi-Newton methods}

A problem of fundamental interest in  optimization is the unconstrained minimization problem
\begin{equation}\label{eq:opt}\min_{x\in \R^n} f(x),\end{equation}
where $f:\R^n\to \R$ is a sufficiently smooth  function. Quasi-Newton (QN) methods, first proposed by Davidon in 1959~\cite{Davidon1959}, are an extremely powerful and popular class of algorithms for solving this problem, especially in the regime of moderately large $n$. In each iteration of a QN method, one approximates the function locally around the current iterate $x_k$ by a quadratic of the form
\begin{equation}\label{eq:98h98hff}f(x_k+s)\approx f(x_k) + (\nabla f(x_k))^T s + \frac{1}{2}s^T B_k s,\end{equation}
where $B_k$ is a suitably chosen approximation of the Hessian ($B_k\approx \nabla^2 f(x_k)$). After this, a direction $s_k$ is computed by minimizing the quadratic approximation in $s$:
\begin{equation}\label{eq:QNdirection}s_k = - B_k^{-1}\nabla f(x_k),\end{equation}
assuming $B_k$ is invertible. The next iterate is then set to
\[x_{k+1} = x_k+h_k, \qquad h_k =\alpha_k s_k,\]
for a suitable choice of stepsize $\alpha_k$, often chosen by a line-search procedure (i.e., by approximately minimizing $f(x_k+ \alpha s_k)$ in $\alpha$). 

Gradient descent arises as a special case of this process by choosing $B_k$ to be constant throughout the iterations. A popular choice is $B_k=L I$, where $I$ is the identity matrix and $L\in \R_+$ is the Lipschitz constant of the gradient of $f$. In such a case, the quadratic approximation \eqref{eq:98h98hff} is a global upper bound on $f(x_k+s)$, which means that $f(x_k+s_k)$ is guaranteed to be at least as good (i.e., smaller or equal) as $f(x_k)$, leading to guaranteed descent. Newton's method also arises as a special case: by choosing $B_k = \nabla^2 f(x_k)$. These two algorithms are extreme cases on the opposite end of  a spectrum. Gradient descent benefits from a trivial  update rule for $B_k$ and from cheap iterations due to the fact that no linear systems need to be solved. However, curvature information is largely ignored, which slows down the practical convergence of the method. Newton's method utilizes the full curvature information contained in the Hessian, but requires the computation of the Hessian in each step, which is expensive for large $n$.  QN methods aim to find a sweet spot on the continuum between these two extremes. In particular, the QN methods choose  $B_{k+1}$ to be a matrix for which  the {\em secant equation} is satisfied:
\begin{equation}\label{eq:secant}B_{k+1}(x_{k+1}-x_k) = \nabla f(x_{k+1})-\nabla f(x_k).\end{equation}

The basic reasoning behind this requirement is the following: if $f$ is a convex quadratic,  then the Hessian  satisfies the secant equation for all pairs of vectors $x_{k+1}$ and $x_k$. If $f$ is not a quadratic, the reasoning is as follows. Using the fundamental theorem of calculus, we have 
\[ \left(\int_{0}^1 \nabla^2 f(x_k + t h_k) \; dt\right) (x_{k+1}-x_k)= \nabla f(x_{k+1}) -\nabla f(x_k) \eqdef y_k. \]
By selecting $B_{k+1}$ that satisfies the secant equation, we are enforcing  $B_{k+1}$ to mimic the action of the integrated Hessian along the line segment joining $x_k$ and $x_{k+1}$. Unless $n=1$, the secant equation \eqref{eq:secant} does not have a unique solution in $B_{k+1}$. All qN methods differ only in which particular solution is used. The formulas transforming $B_k$ to $B_{k+1}$ are called {\em qN updates}.

Since these matrices are used to compute the direction $s_k$ via \eqref{eq:QNdirection}, it is often more reasonable to instead maintain a sequence of inverses $X_k = B_k^{-1}$. By multiplying both sides of \eqref{eq:secant} by $X_{k+1}$, we arrive at the {\em secant equation for the inverse:} \begin{equation}\label{eq:secant2}X_{k+1}(\nabla f(x_{k+1})-\nabla f(x_k)) = x_{k+1}-x_k.\end{equation} 

 The most popular classes of qN updates choose $X_{k+1}$ as the closest matrix to $X_k$, in a suitable norm (usually a weighted Frobenius norm with various weight matrices), subject to the secant equation, often with an  explicit symmetry constraint:
\begin{equation}\label{eq:QN:project_form}X_{k+1} = \arg \min_{X \in \R^{n\times n}} \left\{ \|X-X_k\| \;:\;  X y_k = h_k, \; X = X^T\right\}.\end{equation}

\subsection{Quasi-Newton updates}

Consider now problem \eqref{eq:opt} with 
\begin{equation} \label{eq:QN_quadratic}f(x) = \frac{1}{2}x^T A x - b^T x + c,\end{equation}
where $A$ is an $n\times n$ symmetric positive definite matrix, $b\in \R^n$ and $c\in \R$. Granted, this is not a typical problem for which qN methods would be used by a practitioner. Indeed, the Hessian of $f$ does not change, and hence one {\em does not have to} track it. The problem  can simply be solved by setting the gradient to zero, which leads to the system $Ax = b$,  the  solution being $x_*=A^{-1}b$. As solving a linear system is much simpler than computing the inverse $A^{-1}$, approximately tracking the (inverse) Hessian of $f$ along the path of the iterates $\{x_k\}$---the basic strategy of all  qN methods---seems like too much effort for what is ultimately a much simpler problem.

However, and this is one of the main insights of this work, instead of viewing qN methods as optimization algorithms, we can alternatively interpret them as iterative algorithms producing a sequence of matrices, $\{B_k\}$ or $\{X_k\}$, hopefully converging to some matrix of interest. In particular, one would hope that $X_k \to A^{-1}$ if  a qN method is applied to  \eqref{eq:QN_quadratic}, with any symmetric positive definite initial guess $X_0$. In this case, the qN updates of the minimum distance variety given by \eqref{eq:QN:project_form} take the form
\begin{equation}\label{eq:QN:project_form2}X_{k+1} = \arg \min_{X \in \R^{n\times n}} \left\{ \|X-X_k\| \;:\;  X A h_k = h_k , \; X = X^T\right\}.\end{equation}

\subsection{Randomized quasi-Newton updates}\label{subsec:iuh9898}

While the motivation for our work comes from optimization, having arrived at the update \eqref{eq:QN:project_form2}, we can dispense of some of the implicit assumptions and propose and analyze a wider class of methods. In particular, in this paper we analyze a large class of {\em randomized algorithms} of the type \eqref{eq:QN:project_form2}, where the vector $h_k$ is replaced by a random matrix $S$ and $A$ is \emph{any} invertible\footnote{In fact, one can apply the method to an arbitrary real matrix $A$, in which case the iterates $\{X_k\}$ converge to the Moore-Penrose pseudo-inverse of $A$. However, this development is outside the scope of this paper, and is left for future work. }, and not necessarily symmetric or positive definite matrix. This constitutes a randomized block extension of the qN updates.

\section{Inverting Nonsymmetric Matrices} \label{sec:SIMI-nonsym}

In this paper we are concerned with the development of a family of stochastic algorithms for computing the inverse of a nonsingular matrix $A\in \R^{n\times n}$. The starting point in the development of our methods is the simple observation that the inverse $A^{-1}$ is the (unique) solution of a linear matrix equation, which we shall refer to as the {\em inverse equation}: 
\begin{equation} \label{eq:inverse_equations} AX = I.\end{equation}
Alternatively, one can use the inverse equation $XA = I$ instead. Since \eqref{eq:inverse_equations} is difficult to solve directly, our approach is to iteratively solve a small  randomly relaxed version of \eqref{eq:inverse_equations}. That is, we choose a random matrix $S\in \R^{n \times q}$, with $q\ll n$, and instead solve  the following {\em sketched inverse equation}: \begin{equation}\label{eq:inverse_eq_sketched}S^T AX = S^T.\end{equation}
If we base the method on the second inverse equation, the  sketched inverse equation  $XA S = S$ should be used instead. Note that $A^{-1}$ satisfies \eqref{eq:inverse_eq_sketched}. If $q\ll n$, the sketched inverse equation is of a much smaller dimension than the original one, and hence easier to solve. However, the equation will no longer have a unique solution and  in order to design an algorithm, we need a way of picking a particular solution. Our algorithm defines $X_{k+1}$ to be the solution that is closest to the current iterate $X_k$ in a weighted Frobenius norm. This is repeated in an iterative fashion, each time drawing $S$ independently from a fixed distribution $\cal D$. The distribution $\cal D$ and the matrix $W$ can be seen as parameters of our method. The flexibility of being able to adjust $\cal D$ and $W$ is important: by varying these parameters we obtain various specific instantiations of the generic method, with varying properties and convergence rates. This gives the practitioner the flexibility to adjust the method to the structure of $A$, to the computing environment and so on.


\subsection{Projection viewpoint: sketch-and-project}
 
The next iterate $X_{k+1}$ is the nearest point to $X_k$ that satisfies a \emph{sketched} version of  the inverse equation:
\begin{align}
\boxed{ X_{k+1} = \arg \min_{X \in \R^{n\times n}} \frac{1}{2} \norm{X - X_{k}}_{F(W^{-1})}^2 \quad \mbox{subject to } \quad  S^TAX =S^T} \label{eq:NF}
\end{align}
  In the special case when $S=I$, the only such matrix is the inverse itself, and 
\eqref{eq:NF} is not helpful. However, if  $S$ is ``simple'', \eqref{eq:NF} will be easy to compute and the hope is that through a sequence of such steps, where the matrices $S$ are sampled in an i.i.d. fashion from some distribution, $X_k$ will converge to $A^{-1}$. 

 
 Alternatively, we can sketch the equation $XA=I$ and project onto $XAS=S$:
\begin{equation}\label{eq:NFcols}
\boxed{X_{k+1} = \arg \min_{X \in \R^{n\times n}} \frac{1}{2} \norm{X - X_{k}}_{F(W^{-1})}^2 \quad \mbox{subject to } \quad  XAS =S}\end{equation}

While method~\eqref{eq:NF} sketches the rows of $A$, method~\eqref{eq:NF} sketches the columns of $A.$ Thus, we refer to~\eqref{eq:NF} as the row variant and to~\eqref{eq:NFcols} as the column variant. Both variants converge to the inverse of $A$, as will be established in Section~\ref{sec:conv}.  If $A$ is singular, then it can be shown that the iterates of~\eqref{eq:NFcols} converge to the left inverse, while the iterates of~\eqref{eq:NF} converge to the right inverse.

\subsection{Optimization viewpoint: constrain-and-approximate} 

The row sketch-and-project method can be cast in an apparently different yet equivalent way:
\begin{align}
\boxed{X_{k+1} = \arg_X \min_{X,Y} \frac{1}{2}\norm{X - A^{-1}}_{F(W^{-1})}^2 \quad \mbox{subject to} \quad  X=X_{k} +WA^TSY^T}\label{eq:RF}
\end{align}
Minimization is done over $X\in \R^{n\times n}$ and $Y\in \R^{n\times q}$. In this viewpoint, in each iteration~\eqref{eq:RF}, we select a random affine space that passes through $X_k$ and then select the point in this space that is as close as possible to the inverse. This random search space is special in that, independently of the input pair $(W,S)$, we can efficiently compute
the projection of $A^{-1}$ onto this space, without knowing $A^{-1}$ explicitly. 

Method~\eqref{eq:NFcols} also has an equivalent constrain-and-approximate formulation:
\begin{align}
\boxed{X_{k+1} = \arg_X \min_{X,Y} \frac{1}{2}\norm{X - A^{-1}}_{F(W^{-1})}^2 \quad \mbox{subject to} \quad  X=X_{k} +YS^TA^TW}\label{eq:RFcols}
\end{align}
Methods~\eqref{eq:RF} and~\eqref{eq:RFcols} can be viewed as new variants of approximate inverse preconditioning (AIP)~\cite{Benzi1999,Gould1998,Kolotilina1993,Huckle2007},
which are a class of methods for computing an approximate inverse of $A$ by minimizing  $\norm{XA-I}_F$ via iterative optimization algorithms, such as steepest descent or the minimal residual method. Our methods use a different iterative procedure (projection onto a randomly generated affine space), and work with a more general norm (weighted Frobenius norm).

\subsection{Equivalence}

We now prove that~\eqref{eq:NF} and~\eqref{eq:NFcols} are equivalent to~\eqref{eq:RF} and~\eqref{eq:RFcols}, respectively, and give their explicit solution.
 \begin{theorem}\label{theo:NFRF}
Viewpoints~\eqref{eq:NF} and~\eqref{eq:RF} are equivalent to~\eqref{eq:NFcols} and~\eqref{eq:RFcols}, respectively. 
Further, if $S$ has full column rank, then the explicit solution to~\eqref{eq:NF}  is
   \begin{equation}
  \boxed{ X_{k+1} = X_{k} +WA^TS(S^T A WA^T S)^{-1}S^T(I-AX_{k})} \label{eq:Xupdate}
   \end{equation}
   and the explicit solution to~\eqref{eq:NFcols} is
   \begin{equation}
   \boxed{X_{k+1} = X_{k} +(I-X_{k}A^T)S(S^T A^T WA S)^{-1}S^T A^TW} \label{eq:Xupdatecols}
   \end{equation}
\end{theorem}
\begin{proof}
We will prove all the claims for the row variant, that is, we prove that~\eqref{eq:NF} are~\eqref{eq:RF} equivalent and that their solution is given by~\eqref{eq:Xupdate}. The remaining claims, that~\eqref{eq:NFcols} are~\eqref{eq:RFcols} are equivalent and that their solution is given by~\eqref{eq:Xupdatecols}, follow with analogous arguments.  It suffices to consider the case when $W=I$, as we can perform a change of variables to recover the solution for any $W$. Indeed, in view of \eqref{eq:98y988ff}, with the change of variables
\begin{equation}\label{eq:varchangeW}\hat{X}\eqdef W^{-1/2}X W^{-1/2}, \quad  \hat{A} \eqdef W^{1/2}AW^{1/2}, \quad \hat{S} \eqdef W^{-1/2}S,
\end{equation}
Equation \eqref{eq:NF} becomes
\begin{align}
\min_{\hat{X} \in \R^{n\times n}} \frac{1}{2} \|\hat{X} - \hat{X}_{k}\|_{F}^2 \quad \mbox{subject to } \quad  \hat{S}^T\hat{A}\hat{X} =\hat{S}^T. \label{eq:NFbar}
\end{align}
Moreover, if we let $\hat{Y} = W^{-1/2}Y$, then~\eqref{eq:RF} becomes
\begin{align}
\min_{\hat{X} \in \R^{n\times n}, \hat{Y}\in \R^{n\times q}} \frac{1}{2} \|\hat{X} - \hat{A}^{-1}\|_{F}^2 \quad \mbox{subject to } \quad  \hat{X} =\hat{X}_k + \hat{A}^T\hat{S} \hat{Y}^T. \label{eq:RFbar}
\end{align}

By substituting the constraint in~\eqref{eq:RFbar} into the objective function, then differentiating to find the stationary point, we obtain that
\begin{equation} \label{eq:Xupdatebar}
\hat{X} = \hat{X}_k +\hat{A}^T\hat{S}(\hat{S}^T\hat{A}\hat{A}^T\hat{S})^{-1}\hat{S
}^T(I-\hat{A}\hat{X}_k),\end{equation}
is the solution to~\eqref{eq:RFbar}. Changing the variables back using~\eqref{eq:varchangeW}, \eqref{eq:Xupdatebar} becomes~\eqref{eq:XZupdate}.

Now we prove the equivalence of~\eqref{eq:NFbar} and~\eqref{eq:RFbar} using Lagrangian duality. The sketch-and-project viewpoint~\eqref{eq:NFbar} has a convex quadratic objective function with linear constraints, thus strong duality holds. 
Introducing Lagrangian multiplier $\hat{Y} \in \R^{n\times q}$, the Langrangian dual of~\eqref{eq:NFbar} is given by
\begin{equation}\label{eq:lagdual}
L(\hat{X},\hat{Y})=   \frac{1}{2}\|\hat{X}-\hat{X}_k\|_{F}^2 - \dotprod{{\hat{Y}}^T,\hat{S}^T\hat{A}(\hat{X}-\hat{A}^{-1})}_{F}.
\end{equation}
Clearly,
\[\eqref{eq:NFbar} =\min_{X\in\R^{n\times n}} \max_{\hat{Y} \in\R^{n\times q}} L(\hat{X}, \hat{Y}).\]
  We will now prove that
\[\eqref{eq:RFbar} =\max_{\hat{Y} \in\R^{n\times q}} \min_{X\in\R^{n\times n}}  L(\hat{X}, \hat{Y}),\]
thus proving that~\eqref{eq:NFbar} and~\eqref{eq:RFbar} are equivalent by strong duality. 
Differentiating the Lagrangian in $\hat{X}$ and setting to zero gives
\begin{equation}\label{eq:primopt1}
\hat{X} = \hat{X}_k +\hat{A}^T\hat{S}{\hat{Y}}^T.
\end{equation}
Substituting  into~\eqref{eq:lagdual} gives
\[
L(\hat{X},\hat{Y}) =  \frac{1}{2}\|\hat{A}^T\hat{S}\hat{Y}^T\|_{F}^2 - \dotprod{\hat{A}^T\hat{S}{\hat{Y}}^T, \hat{X}_k +\hat{A}^T\hat{S}{\hat{Y}}^T -\hat{A}^{-1}}_{F} 
= -  \frac{1}{2} \| \hat{A}^T\hat{S}{\hat{Y}}^T \|_{F}^2 - \dotprod{\hat{A}^T\hat{S}{\hat{Y}}^T, \hat{X}-\hat{A}^{-1} }_{F}.
\] 
Adding $\pm \frac{1}{2}\| \hat{X}_k -\hat{A}^{-1} \|_{F}^2$ to the above, we get 
$L(\hat{X},\hat{Y}) =   -\frac{1}{2}\| \hat{A}^T\hat{S}{\hat{Y}}^T+\hat{X}_k-\hat{A}^{-1} \|_{F}^2+\frac{1}{2}\| \hat{X}_k -\hat{A}^{-1} \|_{F}^2$. Finally,  substituting~\eqref{eq:primopt1} into the last equation,  minimizing in $\hat{X}$, then maximizing in $\hat{Y}$, and dispensing of the term $\frac{1}{2}\|\hat{X}_k -\hat{A}^{-1}\|_{F}^2$ as it does not depend on $\hat{Y}$ nor $\hat{X}$, we obtain the dual problem:
\[ \max_{\hat{Y}}\min_{\hat{X}} L(\hat{X},\hat{Y}) =\min_{\hat{X},\hat{Y}} \frac{1}{2}\|\hat{X}-\hat{A}^{-1}\|_{F}^2 \quad \mbox{subject to} \quad \hat{X} =\hat{X}_k+ \hat{A}^T\hat{S}{\hat{Y}}^T.\]
It now remains to change variables using~\eqref{eq:varchangeW} and set $Y = W^{1/2}\hat{Y}$ to obtain~\eqref{eq:RF}.
\end{proof}

\begin{algorithm}[!t]
\begin{algorithmic}[1]
\State \textbf{input:} invertible matrix $A \in \R^{n\times n}$
\State \textbf{parameters:} ${\cal D}$ = distribution over random matrices; positive definite matrix $W\in \R^{n\times n}$
\State \textbf{initialize:} arbitrary square matrix $X_0\in \R^{n\times n}$
\For {$k = 0, 1, 2, \dots$}
	\State Sample an independent copy $S\sim {\cal D}$
	\State Compute $\Lambda = S(S^T A WA^T S)^{-1}S^T$
    \State $X_{k+1} = X_{k} + WA^T\Lambda (I-AX_{k})$
    \Comment This is equivalent to \eqref{eq:NF} and \eqref{eq:RF}
\EndFor
\State \textbf{output:} last iterate $X_k$
\end{algorithmic}
\caption{Stochastic Iterative Matrix Inversion (SIMI) -- nonsym.\ row variant}
\label{alg:asym-row}
\end{algorithm}

\begin{algorithm}[!h]
\begin{algorithmic}[1]
\State \textbf{input:} invertible matrix $A \in \R^{n\times n}$
\State \textbf{parameters:} ${\cal D}$ = distribution over random matrices; pos.\ def.\ matrix $W\in \R^{n\times n}$
\State \textbf{initialize:} arbitrary square matrix $X_0\in \R^{n\times n}$
\For {$k = 0, 1, 2, \dots$}
	\State Sample an independent copy $S\sim {\cal D}$
	\State Compute $\Lambda = S(S^T A^T WA S)^{-1}S^T$
    \State $X_{k+1} = X_{k} + (I-X_{k} A^T) \Lambda  A^T W$
    \Comment This is equivalent to \eqref{eq:NFcols} and \eqref{eq:RFcols}
\EndFor
\State \textbf{output:} last iterate $X_k$
\end{algorithmic}
\caption{Stochastic Iterative Matrix Inversion (SIMI) -- nonsymmetric column variant}
\label{alg:asym-col}
\end{algorithm}

Based on Theorem~\ref{theo:NFRF}, we can summarize the methods described in this section as Algorithm~\ref{alg:asym-row} and Algorithm~\ref{alg:asym-col}. The explicit formulas~\eqref{eq:Xupdate} and~\eqref{eq:Xupdatecols} for \eqref{eq:NF} and~\eqref{eq:NFcols} allow us to efficiently implement these methods, and  facilitate  convergence analysis. In particular, we now see that the convergence analysis of~\eqref{eq:Xupdatecols} follows trivially from analyzing~\eqref{eq:Xupdate}. 
 This is because~\eqref{eq:Xupdate} and~\eqref{eq:Xupdatecols} differ only in terms of a transposition.That is, transposing~\eqref{eq:Xupdatecols} gives $X_{k+1}^T = X_{k}^T +WAS(S^T A^T WA S)^{-1}S^T(I-A^TX_{k}^T)$,
which is the solution to the row variant of the sketch-and-project viewpoint but where the equation $A^TX^T = I$ is sketched instead of $AX=I.$ Thus, since the weighted Frobenius norm is invariant under transposition, it suffices to study the convergence of~\eqref{eq:Xupdate};  convergence of~\eqref{eq:Xupdatecols} follows by simply swapping the role of $A$ for $A^T.$ We collect this observation is the following remark.
\begin{remark}\label{rem:alg2conv}
The expression for the rate of convergence of Algorithm~\ref{alg:asym-col}  is the same as the expression for the rate of convergence of Algorithm~\ref{alg:asym-row}, but with every occurrence of $A$ swapped for $A^T.$ 
\end{remark}

%
%
%
%

\subsection{Relation to multiple linear systems}
\label{sec:relationlinear}



Any iterative method for solving linear systems can be applied to the $n$ linear systems that define the inverse through $AX=I$ to obtain an approximate inverse. However,
not all methods for solving linear systems can be applied to solve these $n$ linear systems simultaneously,  which is necessary for efficient matrix inversion.

The recently proposed methods in~\cite{Gower2015} for solving linear systems can be easily and efficiently generalized to inverting a matrix, and the resulting method is equivalent to our row variant method~\eqref{eq:NF} and~\eqref{eq:RF}.  To show this, we perform the change of variables  $\hat{X}_k= X_kW^{-1/2},$ 
$\hat{A} =W^{1/2} A$ and $\hat{S} = W^{-1/2}S$ then~\eqref{eq:NF} becomes
\[\hat{X}_{k+1} \eqdef X_{k+1}W^{-1/2} = \arg \min_{\hat{X} \in \R^{n\times n}} \frac{1}{2} \|W^{-1/2}(\hat{X} - \hat{X}_{k})\|_{F}^2 \quad \mbox{subject to} \quad   \hat{S}^T\hat{A}\hat{X}=\hat{S}^T.\]
The above is a separable problem and each column of $\hat{X}_{k+1}$ can be calculated separately. 
Let $\hat{x}_{k+1}^i$ be the $i$th column of $\hat{X}_{k+1}$ which can be calculated through
\[\hat{x}_{k+1}^i = \arg \min_{\hat{x} \in \R^n}\frac{1}{2} \|W^{-1/2}(\hat{x} - \hat{x}_{k}^i)\|_{2}^2 \quad \mbox{subject to} \quad  \hat{S}^T \hat{A}\hat{x} =\hat{S}^Te_i.
\]
The above was proposed as a method for solving linear systems in~\cite{Gower2015} applied to the system $\hat{A}\hat{x} = e_i.$ Thus, the convergence results established in~\cite{Gower2015} 
carry over to our row variant~\eqref{eq:NF} and~\eqref{eq:RF}. In particular, the theory in~\cite{Gower2015} proves that the expected norm difference of each column of $W^{-1/2}X_k$ converges to $W^{-1/2}A^{-1}$ with rate $\rho$ as defined in~\eqref{eq:rho}.
This equivalence breaks down when we impose additional matrix properties through constraints, such as symmetry.

\section{Inverting Symmetric Matrices} \label{sec:SIMI-sym}

When $A$ is symmetric, it may be useful to maintain symmetry in the iterates, in which case the nonsymmetric methods---Algorithms~\ref{alg:asym-row} and \ref{alg:asym-col}---have an issue, as they do not guarantee that the iterates are symmetric. However, we can modify~\eqref{eq:NF} by adding a symmetry constraint. The resulting \emph{symmetric} method naturally maintains symmetry in the iterates. 
 
\subsection{Projection viewpoint: sketch-and-project}

The new iterate $X_{k+1}$ is the result of projecting $X_k$ onto the space of matrices that satisfy a sketched inverse equation and that are also symmetric, that is
\begin{align}
\boxed{ X_{k+1} = \arg \min_{X \in \R^{n\times n}}\frac{1}{2} \norm{X - X_{k}}_{F(W^{-1})}^2 \quad \mbox{subject to}\quad   S^TAX =S^T,  \quad X = X^T} \label{eq:NFsym}
\end{align}
 See Figure~\ref{fig:proj} for an illustration of the symmetric update~\eqref{eq:NFsym}. 

 This viewpoint can be seen as a randomized block version of the quasi-Newton (qN) methods~\cite{Goldfarb1970,Greenstadt1969}, as detailed in Section~\ref{sec:QN}.
 The flexibility in using a weighted norm is important for choosing a norm that better reflects the geometry of the problem. For instance, when $A$ is symmetric positive definite, it turns out that $W^{-1} =A$ results in a good method.  This added freedom of choosing an appropriate weighting matrix has proven very useful in the qN literature, in particular, the highly successful BFGS method~\cite{Broyden1965,Fletcher1960,Goldfarb1970,Shanno1971} selects $W^{-1}$ as an estimate of the Hessian matrix.

\begin{figure}
\centering
\resizebox{0.4\textwidth}{!}{
\begin{tikzpicture}[>=triangle 45,font=\sffamily, ] 
   \node (planes)  at (0,0) {\includegraphics[width =10cm]{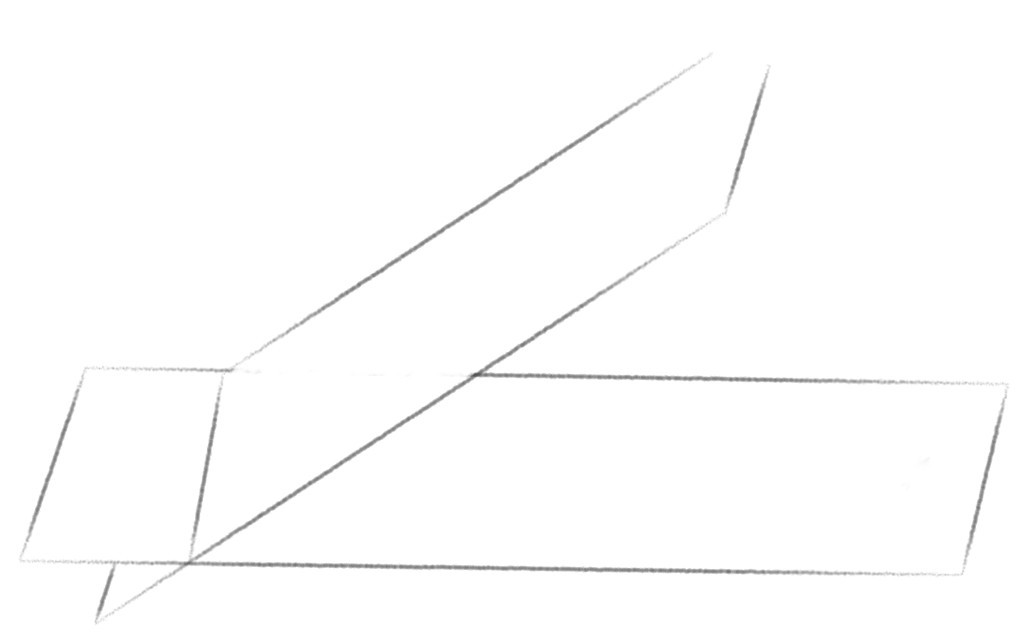}};
   \node (sym) at (3.5cm,-1cm) {$\{X \; | \; X= X^T\}$};   
   \nodepoint{Prevp}{3.7cm}{-2.1cm};
   \node[inner sep =0mm, outer sep =0mm, right = 0pt of  Prevp] (Prev) {$X_{k}$};
   \node (action) at (2.5cm,2.5cm) {$\left\{ X \; | \; S^TAX= S^T\right\}$};
   \nodepoint{Nextp}{-3.1cm}{-2cm};
   \node[inner sep =0mm, outer sep =0mm, left = 0pt of  Nextp] (Next) {$X_{k+1}$}; 
   \nodepoint{Invp}{-2.9cm}{-1cm};
   \node[inner sep =0mm, outer sep =0mm, left = 0pt of  Invp] (Inv) {$A^{-1}$};  
    \draw [semithick,->] (Prevp) -- (Nextp)  node [midway,above=0.0cm] {Projection};
    \draw (Nextp)++(0.05,0.3cm) --  ++ (0.3cm,0.0cm) -- ++ (-0.05cm, -0.3cm);
    \node[inner sep =0mm, outer sep =0mm] (orthodot) at (-3.1cm+0.17cm,-2cm+0.15cm) {$\cdot$};
\end{tikzpicture}}
\caption{\footnotesize The new estimate $X_{k+1}$ is obtained by projecting $X_{k}$ onto the affine space formed by intersecting $\{X \; | \; X= X^T\}$ and $\left\{ X \; |\; S^TAX= S^T\right\}$.}
\label{fig:proj}
\end{figure}
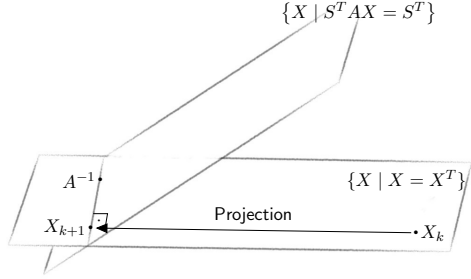

\subsection{Optimization viewpoint: constrain-and-approximate} 
The viewpoint~\eqref{eq:NFsym} also has an interesting dual viewpoint:
\begin{equation}
\boxed{X_{k+1} = \arg_{X}  \min_{X, Y} \frac{1}{2}\norm{X -A^{-1}}_{F(W^{-1})}^2 \quad \mbox{subject to} \quad
X = X_k + \frac{1}{2}(YS^TAW + WA^TSY^T)} \label{eq:RFsym}
\end{equation}
The minimum is taken over $X\in \R^{n\times n}$ and $Y\in \R^{n\times q}$. The next iterate, $X_{k+1}$, is the best approximation to $A^{-1}$ restricted to a random affine space of symmetric matrices.
Furthermore, \eqref{eq:RFsym} is a symmetric equivalent of~\eqref{eq:RF}; that is, the constraint in~\eqref{eq:RFsym} is the result of intersecting the constraint in~\eqref{eq:RF} with the space of symmetric matrices. 
%

\subsection{Equivalence}

We now prove that the two viewpoints~\eqref{eq:NFsym} and~\eqref{eq:RFsym} are equivalent, and show their explicit solution.
 \begin{theorem}
If $A$ and $X_k$ are symmetric, then the viewpoints~\eqref{eq:NFsym} and~\eqref{eq:RFsym} are equivalent. That is, they define the same $X_{k+1}$. Furthermore, if $S$ has full column rank, and we let $\Lambda \eqdef (S^T A WA S)^{-1}$,  then the explicit solution to ~\eqref{eq:NFsym} and~\eqref{eq:RFsym} is
\begin{align}&\boxed{X_{k+1} =X_{k}-(X_{k}A-I)S\Lambda S^T A W +WAS\Lambda S^T (AX_{k}-I)(AS\Lambda S^TAW-I)} \label{eq:Xupdatesym}
 \end{align}
\end{theorem}
\begin{proof}
It was recently shown in~\cite[Section~2]{Gower2014c} and~\cite[Section~4]{Hennig2015}\footnote{To re-interpret methods for solving linear systems through Bayesian inference, Hennig constructs estimates of the inverse system matrix using the sampled action of a matrix taken during a linear solve~\cite{Hennig2015}.}  that~\eqref{eq:Xupdatesym} is the solution to~\eqref{eq:NFsym}. We now prove the equivalence of~\eqref{eq:NFsym} and~\eqref{eq:RFsym} using Lagrangian duality.  It suffices to prove the claim for  $W=I$ as we did in the proof of Theorem~\ref{theo:NFRF}, since the change of variables~\eqref{eq:varchangeW} applied to~\eqref{eq:NFsym} shows that~\eqref{eq:NFsym} is equivalent to
\begin{equation}
\min_{\hat{X} \in \R^{n\times n}} \frac{1}{2} \|\hat{X} - \hat{X}_{k}\|_{F}^2 \quad \mbox{subject to} \quad  \hat{S}^T\hat{A}\hat{X} =\hat{S}^T, \quad \hat{X}= \hat{X}^T. \label{eq:NFsymbar}
\end{equation}
Since~\eqref{eq:NFsym} has a convex quadratic objective with linear constraints, strong duality holds. Thus we will derive a dual formulation for~\eqref{eq:NFsymbar} then use the change of coordinates~\eqref{eq:varchangeW} to recover the solution to~\eqref{eq:NFsym}. 
Let $\hat{Y} \in \R^{n\times q}$ and $\Gamma \in \R^{n\times n}$ and consider the Lagrangian of~\eqref{eq:NFsymbar} which is
\begin{equation}\label{eq:lagsym}
L(\hat{X},\hat{Y}, \Gamma)= \frac{1}{2}\|\hat{X}-\hat{X}_k\|_F^2 - \dotprod{\hat{Y}^T,\hat{S}^T\hat{A}(\hat{X}-\hat{A}^{-1})}_{F} - \dotprod{\Gamma,\hat{X}-\hat{X}^T}_F.
\end{equation}
 Differentiating in $\hat{X}$ and setting to zero gives
\begin{equation}\label{eq:optsym}
\hat{X} = \hat{X}_k+ \hat{A}^T\hat{S}\hat{Y}^T +\Gamma-\Gamma^T.
\end{equation}
Applying the symmetry constraint $X=X^T$ gives
$\Gamma-\Gamma^T = \frac{1}{2}(\hat{Y}\hat{S}^T\hat{A} - \hat{A}^T\hat{S}\hat{Y}^T)$.
 Substituting the above into~\eqref{eq:optsym} gives
\begin{equation}\label{eq:optsym2}
\hat{X} = \hat{X}_k +\frac{1}{2}(\hat{Y}\hat{S}^T\hat{A} + \hat{A}^T\hat{S}\hat{Y}^T ).
\end{equation}
Now let $\Theta = \frac{1}{2}(\hat{Y}\hat{S}^T\hat{A} + \hat{A}^T\hat{S}\hat{Y}^T)$ and note that, since the matrix $\Theta+\hat{X}_k -\hat{A}^{-1}$ is symmetric, we get
\begin{equation}\label{eq:symmob}
 \dotprod{\hat{A}^T\hat{S}\hat{Y}^T,\Theta+\hat{X}_k -\hat{A}^{-1}}_F =  \dotprod{\Theta ,\Theta+\hat{X}_k -\hat{A}^{-1}}_F.  \end{equation}
Substituting~\eqref{eq:optsym2}  into~\eqref{eq:lagsym} gives
$
L(\hat{X},\hat{Y},\Gamma) = \frac{1}{2}\norm{\Theta}_F^2 - \dotprod{\hat{A}^T\hat{S}\hat{Y}^T,\Theta+\hat{X}_k -\hat{A}^{-1}}_F 
 \overset{\eqref{eq:symmob}}{=} \frac{1}{2}\norm{\Theta}_F^2 - \dotprod{\Theta,\Theta+\hat{X}_k -\hat{A}^{-1}}_F
 = -\frac{1}{2}\norm{\Theta}_F^2 -\dotprod{\Theta,\hat{X}_k -\hat{A}^{-1}}_F$. Adding $\pm\frac{1}{2}\|\hat{X}_k-\hat{A}^{-1}\|_F^2$ to the above, we obtain 
$L(\hat{X},\hat{Y},\Gamma) = -\frac{1}{2}\|\Theta +\hat{X}_k -\hat{A}^{-1}\|_F^2+\frac{1}{2}\|\hat{X}_k-\hat{A}^{-1}\|_F^2. $
Finally, using~\eqref{eq:optsym2} and maximizing over $\hat{Y}$ then minimizing over $X$ gives the dual:
\[\min_{\hat{X},\hat{Y}} \frac{1}{2}\|\hat{X} -\hat{A}^{-1}\|_F^2 \quad \mbox{subject to} \quad \hat{X} = \hat{X}_k + \frac{1}{2}( \hat{Y}	\hat{S}^T\hat{A} + \hat{A}^T\hat{S}\hat{Y}^T ). \]
It now remains to change variables according to~\eqref{eq:varchangeW} and set $Y = W^{1/2}\hat{Y}.$
\end{proof}


\begin{algorithm}[!h]
\begin{algorithmic}[1]
\State \textbf{input:} symmetric invertible matrix $A \in \R^{n\times n}$
\State \textbf{parameters:} ${\cal D}$ = distribution over random matrices; symmetric pos.\ def.\  $W\in \R^{n\times n}$
\State \textbf{initialize:} symmetric matrix $X_0\in \R^{n\times n}$
\For {$k = 0, 1, 2, \dots$}
	\State Sample an independent copy $S\sim {\cal D}$
	\State Compute $\Lambda\leftarrow  S(S^T A WA S)^{-1}S^T$,\quad  $\Theta \leftarrow \Lambda A W$, \quad 	 $M_k \leftarrow X_k A - I$
    \State $X_{k+1} =X_{k}- M_k \Theta  - (M_k \Theta)^T + \Theta^T (AX_{k}A-A) \Theta$
    \Comment Equiv. to \eqref{eq:NFsym} \& \eqref{eq:RFsym}    
\EndFor
\State \textbf{output:} last iterate $X_k$
\end{algorithmic}
\caption{Stochastic Iterative Matrix Inversion (SIMI) -- symmetric variant}
\label{alg:sym}
\end{algorithm}

\section{Convergence} \label{sec:conv}
We now analyze the convergence of the \emph{error}, $X_{k}-A^{-1}$, for iterates of Algorithms~\ref{alg:asym-row}, \ref{alg:asym-col}   and~\ref{alg:sym}. For the sake of economy of space, we only analyze Algorithms~\ref{alg:asym-row} and \ref{alg:sym}. Convergence of  Algorithm~\ref{alg:asym-col}  follows from convergence of Algorithm~\ref{alg:asym-row} by observing Remark~\ref{rem:alg2conv}.

 The first analysis we present in Section~\ref{sec:normEconv} is concerned with the convergence of $\norm{\E{X_{k}-A^{-1}}}^2,$ that is, the {\em norm of the expected error}.   We then analyze the convergence of $\E{\norm{X_{k}-A^{-1}}}^2,$ the {\em expected norm of the error}. The latter is a stronger type of convergence,  as explained in the following proposition.

\begin{proposition} \label{lem:conv} Let $X\in \R^{n\times n}$ be a random matrix, $\norm{\cdot}$ a matrix norm induced by an inner product, and fix $A^{-1}\in \R^{n\times n}$. Then
\[ \|\E{X-A^{-1}}\|^2   = \E{\| X-A^{-1} \|^2}  - \E{\| X - \E{X}\|^2}.\]
\end{proposition}
\begin{proof} Note that $\Esmall{\norm{ X - \Esmall{X}}^2}= \Esmall{\norm{X}^2} - \norm{\Esmall{X}}^2.$
Adding and subtracting $\norm{A^{-1}}^2-2\dotprod{\E{X},A^{-1}}$ from the right hand side, then grouping the appropriate terms, yields the desired result. 
\end{proof} 

This shows that if  $\Esmall{\|X_{k}-A^{-1}\|^2}$ converges to zero, then  $\norm{\Esmall{X_{k}-A^{-1}}}^2$ converges to zero. But the converse is not necessarily true. Rather, the variance $\Esmall{\left\| X_{k} - \E{X_{k}}\right\|^2}$ must converge to zero for the converse to be true\footnote{The convergence of  $\|\Esmall{X_{k}-A^{-1}}\|^2$ is also known in the probability literature as $L_2$--norm convergence. It also follows trivially from the Markov's inequality that convergence in $L_2$--norm implies convergence in probability.}. The convergence of  Algorithms~\ref{alg:asym-row} and \ref{alg:sym} can  be  characterized by studying the  random matrix
   \begin{equation}\label{eq:Z}
   Z \eqdef A^TS(S^T A WA^T S)^{-1}S^T A.
   \end{equation}
The update step of  Algorithm~\ref{alg:asym-row} can be re-written as a simple fixed point formula
   \begin{align}
   X_{k+1} -A^{-1}&=   \left(I-WZ\right)(X_{k}-A^{-1}). \label{eq:XZupdate}
   \end{align}
We can also simplify the iterates of Algorithm~\ref{alg:sym} to 
   \begin{align}
X_{k+1}-A^{-1} &=  \left(I-WZ\right) (X_{k}-A^{-1})\left(I -ZW \right).
 \label{eq:XZupdatesym}
\end{align}


    The only stochastic component in both methods is contained in the matrix $Z$, and ultimately, the convergence of the iterates will depend on $\E{Z}$, the expected value of this matrix. Thus we start with two lemmas concerning the $Z$ and $\E{Z}$ matrices.

\begin{lemma}\label{lem:ZW} If $Z$ is defined as in~\eqref{eq:Z}, then
\begin{enumerate}
\item the eigenvalues of $W^{1/2}ZW^{1/2}$ are either $0$ or $1$,
\item  $W^{1/2}ZW^{1/2}$ projects onto the $q$--dimensional subspace $\myRange{W^{1/2}A^TS}$. 
\end{enumerate}
\end{lemma}
\begin{proof} Using~\eqref{eq:Z}, simply verifying that $(W^{1/2}ZW^{1/2})^2 = W^{1/2}ZW^{1/2}$ proves that it is a projection matrix, and thus has eigenvalues $0$ or $1$. Furthermore, the matrix $W^{1/2}ZW^{1/2}$ projects onto $\myRange{W^{1/2}A^TS}$, which follows from 
$W^{1/2}ZW^{1/2} (W^{1/2}A^TS) = W^{1/2}A^TS$ and the fact that $W^{1/2}ZW^{1/2} y =0$ for all  $y \in \Null{W^{1/2}A^TS}$.
Finally, \[\dim\left(\myRange{W^{1/2}A^TS}\right) =\Rank{W^{1/2}A^TS} = \Rank{S} =q.\]
\end{proof}

\begin{lemma}\label{lem:EZ}
The spectrum of $W^{1/2}\E{Z}W^{1/2}$ is contained in $[0,1].$
\end{lemma}
\begin{proof}
Let $\hat{Z} = W^{1/2}ZW^{1/2},$ thus $W^{1/2}\E{Z}W^{1/2} =\Esmall{\hat{Z}}.$
Since the mapping $A \mapsto \lambda_{\max}(A)$ is convex,  by Jensen's inequality we get
$\lambda_{\max} (\Esmall{\hat{Z}}) \leq\Esmall{\lambda_{\max}(\hat{Z})}$.
Applying Lemma~\ref{lem:ZW}, we conclude that $\lambda_{\max} ( \Esmall{\hat{Z}} )  \leq 1$.
The inequality $\lambda_{\min} (\Esmall{\hat{Z}})  \geq 0$ can be shown analogously using convexity of the mapping $A\mapsto -\lambda_{\min}(A)$.
\end{proof}

\subsection{Norm of the expected error} \label{sec:normEconv}

We start by proving that the norm of the expected error of the iterates of Algorithm~\ref{alg:asym-row} and Algorithm~\ref{alg:sym} converges to zero. The following theorem is remarkable in that we do not need to make any assumptions on the distribution $S$, except that $S$ has full column rank. Rather, the theorem pinpoints that convergence  depends solely on the spectrum of $I-W^{-1/2}\E{Z}W^{-1/2}.$  

\begin{theorem} \label{theo:normEconv}
Let $S$ be a random matrix which has full column rank with probability~$1$ (so that $Z$ is well defined). Then the iterates $X_{k+1}$ of Algorithm~\ref{alg:asym-row} satisfy
\begin{equation} \label{eq:XXXX}
\E{ X_{k+1} -A^{-1} } =(I-W\E{Z}) \E{X_{k} - A^{-1}}.
 \end{equation}
Let $X_0 \in \R^{n\times n}$. If $X_k$ is calculated in either one of these two ways
 \begin{enumerate}
 \item Applying $k$ iterations of  Algorithm~\ref{alg:asym-row}, 
 \item   Applying $k$ iterations of Algorithm~\ref{alg:sym} (assuming $A$ and $X_0$ are  symmetric),
\end{enumerate}
then $X_k$ converges to the inverse exponentially fast, according to
\begin{equation}\label{eq:normexpconv}
\norm{\E{X_{k}-A^{-1}}}_{W^{-1}} \leq \rho^k \norm{X_{0}-A^{-1}}_{W^{-1}}, 
\end{equation}
where 
\begin{equation} \label{eq:rhoequiv}\rho \eqdef 1-\lambda_{\min}(W^{1/2}\E{Z}W^{1/2}).\end{equation}
Moreover, we have the following lower and upper bounds on the convergence rate:
\begin{equation}\label{eq:rholower}
0 \leq 1- \E{q}/ n  \leq \rho \leq 1.
\end{equation}
\end{theorem}
\begin{proof}
Let
\begin{equation}\label{eq:oihsoi8dhyY8J}R_k \eqdef W^{-1/2}(X_k - A^{-1})W^{-1/2} \quad \text{and} \quad \hat{Z} \eqdef W^{1/2}ZW^{1/2},\end{equation}
for all $k$. Left and right multiplying~\eqref{eq:XZupdate} by $W^{-1/2}$ gives
\begin{equation}\label{eq:barRknext}
R_{k+1} = (I -\hat{Z})R_k.
\end{equation}
Taking expectation with respect to $S$ in~\eqref{eq:barRknext} gives
\begin{equation}\label{eq:EXinXk} \Esmall{R_{k+1} \;| \; R_k} = (I - \Esmall{\hat{Z}} ) R_k. 
\end{equation}
Taking full expectation in~\eqref{eq:barRknext} and using the tower rule gives
\begin{eqnarray}
 \E{ R _{k+1}} &=& \E{\E{ R _{k+1} \;|\; R_{k}}} 
 \overset{\eqref{eq:EXinXk}}{=} \Esmall{(I -\Esmall{\hat{Z}})R_{k}}   =
 (I -\Esmall{\hat{Z}})\E{R_{k}}. \label{eq:Efullasym}
\end{eqnarray}
Applying the norm in~\eqref{eq:Efullasym} gives
\begin{align} \label{eq:normErecur}
\|\E{X_{k+1}-A^{-1}}\|_{W^{-1}} = \|\Esmall{ R_{k+1}}\|_{2} &\leq \|I -\Esmall{\hat{Z}}\|_2  \|\Esmall{ R_{k}}\|_{2} \nonumber \\
&= \| I -\Esmall{\hat{Z}}\|_2  \|\Esmall{X_k-A^{-1}}\|_{W^{-1}}. 
\end{align}
Furthermore,
\begin{align}
\|I -\Esmall{\hat{Z}}\|_2 = \lambda_{\max}\left(I -\Esmall{\hat{Z}}\right)
 =1-\lambda_{\min}\left(\Esmall{\hat{Z}}\right) \overset{\eqref{eq:rhoequiv}}{=} \rho, \label{eq:rhonorm}
\end{align}
where we used to symmetry of $(I-\Esmall{\hat{Z}})$ when passing from the operator norm to the spectral radius. Note that the symmetry of $\Esmall{\hat{Z}}$ derives from the symmetry of $\hat{Z}$.
It now remains to unroll the recurrence in~\eqref{eq:normErecur} to get~\eqref{eq:normexpconv}. Now we analyze the iterates of  Algorithm~\ref{alg:sym}. Left and right multiplying~\eqref{eq:XZupdatesym} by $W^{-1/2}$  we have 
\begin{equation}\label{eq:barRevol}
R_{k+1}= P(R_k)\eqdef (I-\hat{Z}) R_k (I -\hat{Z}). 
\end{equation}
Defining 
$\bar{P}: R \mapsto \E{P(R) \, | \, R_k}$,  
taking expectation in \eqref{eq:barRevol} conditioned on $R_{k}$, gives
$\E{R_{k+1}\; | \; R_k} =  \bar{P}(R_{k})$. As $\bar{P}$ is a linear operator, taking expectation again yields
\begin{equation} \label{eq:barPZERk}
 \E{R_{k+1}} = \E{\bar{P}(R_{k} )} =  \bar{P}(\E{R_{k} }). 
 \end{equation}
Letting $||| \bar{P}|||_2 \eqdef \max_{\norm{R}_2=1} \norm{\bar{P}(R)}_2$, applying norm in~\eqref{eq:barPZERk} gives
\begin{eqnarray} 
\norm{\E{X_{k+1}-A^{-1}}}_{W^{-1}} &=& \norm{\E{R_{k+1}}}_2 
\leq  ||| \bar{P}|||_2 \norm{\E{R_{k}}}_2 \nonumber\\
& = &  ||| \bar{P}|||_2 \norm{\E{X_{k}-A^{-1}}}_{W^{-1}}.\label{eq:normsplitbarP}
\end{eqnarray}

Clearly, $P$ is a \emph{positive linear map}, that is, it is linear and maps positive semi-definite matrices to positive semi-definite matrices. Thus, by Jensen's inequality, the map $\bar{P}$ is also a positive linear map.  As every positive linear map attains its norm at the identity matrix (see Corollary 2.3.8 in~\cite{bhatia07}), we have  
\[
||| \bar{P}|||_2 =  \norm{\bar{P}(I)}_2 
\overset{\eqref{eq:barRevol}}{=} \|\Esmall{ (I-\hat{Z}) I(I -\hat{Z} )}\|_2
\overset{(\text{Lemma}~\ref{lem:ZW})}{=} \|\Esmall{ I-\hat{Z}}\|_2
 \overset{\eqref{eq:rhonorm}}{=} \rho.
\]
Inserting the above equivalence in~\eqref{eq:normsplitbarP} and unrolling the recurrence gives~\eqref{eq:normexpconv}.

Finally, to prove~\eqref{eq:rholower}, as proven in Lemma~\ref{lem:EZ}, the spectrum of $W^{1/2}\E{Z}W^{1/2}$ is contained in $[0,\,1]$ consequently $0\leq \rho \leq 1.$ Furthermore,
 as the trace of a matrix is equal to the sum of its eigenvalues, we have
\begin{eqnarray}
\Esmall{q} &\overset{(\text{Lemma}~\ref{lem:ZW})}{=}& \Esmall{\Tr{W^{1/2}ZW^{1/2}}}\notag 
= \Tr{\Esmall{W^{1/2}ZW^{1/2}}} \notag \\
&\geq & n \lambda_{\min}(\Esmall{W^{1/2}ZW^{1/2}}),\label{eq:traceb1} \end{eqnarray}
where we used that $W^{1/2}ZW^{1/2}$ projects onto a $q$--dimensional subspace (Lemma~\ref{lem:ZW}), and thus $\Tr{W^{1/2}ZW^{1/2}} = q.$ Rearranging~\eqref{eq:traceb1} gives~\eqref{eq:rholower}.
\end{proof}

If $\rho =1$, this theorem does not guarantee convergence.
However, $\rho<1$ when $\E{Z}$ is positive definite, which is the case in all practical variants of our method, some of which we describe in Section~\ref{sec:discretemethods}.

\subsection{Expectation of the norm of the error}

Now we consider the convergence of the expected norm of the error. This form of convergence is preferred, as it also proves that the variance of the iterates converges to zero (see Proposition~\ref{lem:conv}). 

\begin{theorem} \label{theo:Enormconv}
Let $S$ be a random matrix that has full column rank with probability~$1$ and such that
$\E{Z}$ is positive definite, where $Z$ is defined in~\eqref{eq:Z}.
Let $X_0 \in \R^{n\times n}$. If $X_k$ is calculated in either one of these two ways
 \begin{enumerate}
 \item Applying $k$ iterations of  Algorithm~\ref{alg:asym-row}, 
 \item   Applying $k$ iterations of Algorithm~\ref{alg:sym} (assuming  $A$ and $X_0$ are symmetric),
\end{enumerate}
then $X_k$ converges to the inverse according to
\begin{equation} \label{eq:Enormconv}
\E{\norm{X_{k} -A^{-1} }_{F(W^{-1})}^2} \leq \rho^k \norm{X_{0} - A^{-1}}_{F(W^{-1})}^2.
 \end{equation}
\end{theorem}

\begin{proof}
First consider Algorithm~\ref{alg:asym-row}, where $X_{k+1}$ is calculated by iteratively applying~\eqref{eq:XZupdate}. Using the substitution~\eqref{eq:oihsoi8dhyY8J} again, then from~\eqref{eq:XZupdate} we have 
$ R_{k+1} = (I-\hat{Z})R_k$, from 
which we obtain
\begin{eqnarray}
\norm{R_{k+1}}_{F}^2  & = & \|(I-\hat{Z})R_{k}\|_{F}^2 
  = \Tr{ (I-\hat{Z} )(I-\hat{Z})R_kR_{k}^T} \nonumber\\
&\overset{(\text{Lemma}~\ref{lem:ZW})}{=}& \Tr{(I-\hat{Z})R_kR_{k}^T} \quad = \quad \norm{R_k}_{F}^2 - \Tr{\hat{Z}R_{k}R_k^T}. \label{eq:asymmapplylem} 
\end{eqnarray}
Taking expectations,  we get
$\Esmall{\norm{R_{k+1}}_{F}^2  \, | \, R_k} =\norm{R_k}_{F}^2 - \Tr{\Esmall{\hat{Z}}R_{k}R_k^T}.$  Using the inequality $\Tr{\Esmall{\hat{Z}}R_{k}R_k^T} \geq  \lambda_{\min}(\Esmall{\hat{Z}})\Tr{R_kR_{k}^T} $, which relies on the symmetry of $\Esmall{\hat{Z}},$ we get
\[\Esmall{\norm{R_{k+1}}_{F}^2 \,| \, R_k} 
\leq  (1-\lambda_{\min}(\Esmall{\hat{Z}}) )
\norm{R_k}_{F}^2 =\rho\cdot  \norm{R_k}_{F}^2.\]
In order to arrive at \eqref{eq:Enormconv},
it now remains to take full expectation,  unroll the recurrence and use the substitution \eqref{eq:oihsoi8dhyY8J}
  
Now we assume that $A$ and $X_0$ are symmetric and $\{X_k\}$ are the iterates computed by Algorithm~\ref{alg:sym}. 
Left and right multiplying~\eqref{eq:XZupdatesym} by $W^{-1/2}$  we have 
\begin{equation}\label{eq:barRevol2}
R_{k+1}= (I-\hat{Z}) R_k (I -\hat{Z} ). 
\end{equation}
Taking norm we have
\begin{eqnarray*}
\norm{R_{k+1}}_{F}^2 & \overset{(\text{Lemma}~\ref{lem:ZW})}{=} &
\Tr{R_k (I -\hat{Z} ) R_k (I -\hat{Z} )} 
\\
&=& \Tr{R_k R_k (I -\hat{Z} )} -\Tr{R_k \hat{Z}R_k (I -\hat{Z} )}\\
& \leq&  \Tr{R_k R_k (I -\hat{Z}) }, 
\end{eqnarray*}
where in the last inequality we used that $I -\hat{Z}$ is an orthogonal projection and thus it is symmetric positive semi-definite, whence
$\Tr{R_k \hat{Z}R_k (I -\hat{Z} )} = 
\Tr{\hat{Z}^{1/2}R_k (I -\hat{Z} )R_k \hat{Z}^{1/2}} \geq 0.$
The remainder of the proof follows similar steps as those we used in the first part of the proof from~\eqref{eq:asymmapplylem} onwards.
\end{proof}

Theorem~\ref{theo:Enormconv} establishes that  the expected norm of the error converges exponentially fast to zero. Moreover, the convergence rate $\rho$ is the same that appeared in Theorem~\ref{theo:normEconv}, where we established the convergence of the norm of the expected error.  Both results can be recast as iteration complexity bounds. For instance, using standard arguments, from Theorem~\ref{theo:normEconv} we observe that for a given $0<\epsilon <1$ we get
\begin{equation} \label{eq:itercomplex}k \geq \left(\frac{1}{2}\right)\frac{1}{1-\rho} \log\left(\frac{1}{\epsilon}\right) \quad \Rightarrow \quad \norm{\E{X_k-A^{-1}}}_{W^{-1}}^2 \leq \epsilon \norm{X_0-A^{-1}}_{W^{-1}}^2.
\end{equation}
On the other hand, from Theorem~\ref{theo:Enormconv} we have
\begin{equation} \label{eq:itercomplex2}k \geq \frac{1}{1-\rho} \log\left(\frac{1}{\epsilon}\right) \quad \Rightarrow \quad \E{\norm{X_k-A^{-1}}_{F(W^{-1})}^2} \leq \epsilon \norm{X_0-A^{-1}}_{F(W^{-1})}^2.
\end{equation}
To push the expected norm of the error below $\epsilon$ (see\eqref{eq:itercomplex2}), we require double the  iterates compared to bringing the norm of expected error below $\epsilon$ (see\eqref{eq:itercomplex}).  This is because in Theorem~\ref{theo:Enormconv} we determined that $\rho$ is the rate at which the expectation of the \emph{squared} norm error converges, while in Theorem~\ref{theo:normEconv} we determined that $\rho$ is the rate at which the norm, without the square, of the expected error converges. However, as proven in Proposition~\ref{lem:conv}, the former is a stronger form of convergence. Thus, Theorem~\ref{theo:normEconv} does not give a stronger result than Theorem~\ref{theo:Enormconv}, but rather, these theorems give qualitatively different results.

\section{Discrete Random Matrices} \label{sec:discrete}

We now consider the case of a discrete random matrix $S$. We show that when $S$ is a \emph{complete discrete sampling}, then $\E{Z}$ is 
positive definite, and thus from Theorems~\ref{theo:normEconv} and~\ref{theo:Enormconv}, Algorithms~\ref{alg:asym-row}--\ref{alg:sym} converge.

\begin{definition}[Complete Discrete Sampling]\label{def:complete}
The random matrix $S$ has a finite discrete distribution with $r$ outcomes. In particular,  $S= S_i \in \R^{n \times q_i}$ with probability  $p_i>0$ for $i=1,\ldots, r$, where $S_i$ is of full column rank.   We say that $S$ is a complete discrete sampling when
 $\mathbf{S} \eqdef [S_1, \ldots, S_r] \in \R^{n\times n}$ has full row rank.
\end{definition}

As an  example of a complete discrete sampling, let $S =e_i$ (the $i$th unit coordinate vector in $\R^n$) with probability $p_i =1/n$, for $i =1,\ldots, n.$ Then $\mathbf{S}$, as defined in Definition~\ref{def:complete},  is equal to the identity matrix: $\mathbf{S} =I$. Consequently, $S$ is a complete discrete sampling. In fact, from any basis of $\R^n$ we could construct a complete discrete sampling in an analogous way.

Next we establish that when $S$ is discrete random matrix, that $S$ having a complete discrete distribution is a necessary and sufficient condition for $\E{Z}$ to be positive definite.  This will allow us to determine an optimized distribution for $S$ in Section~\ref{sec:discreteopt}.

\begin{proposition}\label{prop:Ediscrete} 
Let $S$ be a discrete random matrix with $r$ outcomes $S_r$ all of which have full column rank. The matrix $\E{Z}$ is positive definite if and only if $S$ is a complete discrete sampling. Furthermore
\begin{equation}\E{Z} =  A^T \mathbf{S} D^2 \mathbf{S}^T  A \label{eq:EZdiscrete}, \qquad \text{where} \end{equation}
\begin{equation} \label{eq:D}
D~\eqdef~\mbox{Diag}\left( \sqrt{p_1}(S_1^T A WA^T S_1)^{-1/2}, \ldots, \sqrt{p_r}(S_r^T A WA^T S_r)^{-1/2}\right).\end{equation}
\end{proposition}
\begin{proof} 
 Taking the expectation of $Z$ as defined in~\eqref{eq:Z} gives
\begin{align}
\E{Z} &= \sum_{i=1}^r A^T S_i (S_i^T A WA^T S_i)^{-1}S_i^T A p_i \nonumber \\
& =   A^T\left(\sum_{i=1}^r   S_i \sqrt{p_i}(S_i^T A WA^T S_i)^{-1/2} (S_i^T A WA^T S_i)^{-1/2}  \sqrt{p_i}S_i^T \right) A \nonumber \\
&= \left(  A^T \mathbf{S} D\right)  \left( D \mathbf{S}^T  A\right), \nonumber
\end{align} and $\E{Z}$ is clearly positive semi-definite.  
Note that, since we assume that $S$ has full column rank with probability $1$, the matrix $D$ is well defined and nonsingular. Given that $\E{Z}$ is positive semi-definite, we need only show that $\Null{\E{Z}}$ contains only the zero vector if and only if $S$ is a complete discrete sampling.
 Let $v \in \Null{\E{Z}}$ and $v \neq 0,$ thus 
$0=v^\top A^\top  \mathbf{S}  D^2 \mathbf{S}^\top   Av = \|D\mathbf{S}^\top   Av\|_2^2, $
which shows that $\mathbf{S}^\top   Av =0$ and thus $v \in \Null{\mathbf{S}^\top A}.$  As $A$ is nonsingular, it follows that $v=0$ if and only if  $\mathbf{S}^\top $ has full column rank. 
%
%
%
\end{proof}

With a closed form expression for $\E{Z}$ we can optimize $\rho$ over the possible distributions of $S$ to yield a better convergence rate.

\subsection{Optimizing an upper bound on the convergence rate} \label{sec:discreteopt}

So far we have proven two different types of convergence for Algorithms~\ref{alg:asym-row}, \ref{alg:asym-col} and  \ref{alg:sym} in Theorems~\ref{theo:normEconv} and~\ref{theo:Enormconv}. Furthermore, both forms of convergence depend on the same convergence rate $\rho$ for which we have a closed form expression~\eqref{eq:rhoequiv}.

The availability of a closed form expression for the convergence rate opens up the possibility of designing particular distributions for $S$ optimizing the rate. In~\cite{Gower2015} it was shown that (in the context of solving linear systems) for a complete discrete sampling, computing the optimal probability distribution, assuming that the matrices $\{S_i\}_{i=1}^r$ are fixed, leads to  a semi-definite program (SDP). In some cases,  the gain in performance from the optimal probabilities is much larger than the loss incurred by having to solve the SDP. However, this is not always the case.  Here we propose an alternative: to optimize the following upper bound on the convergence rate:
\begin{equation}\label{eq:98h9s8h8s} \rho = 1- \lambda_{\min}(W^{1/2}\E{Z} W^{1/2}) \leq  1 - \frac{1}{\Tr{W^{-1/2} (\E{Z_p})^{-1} W^{-1/2}}} \eqdef \gamma(p).\end{equation}
By writing $Z_p$ instead of $Z$, we emphasized the dependence of $Z$  on $p = (p_1,\ldots, p_r)\in \R^r$, belonging to the probability simplex
\[ \Delta_r \eqdef \left\{p = (p_1,\dots,p_r) \in \R^r \;:\; \sum_{i=1}^r p_i =1, \; p\geq 0\right\}.\]
Our goal is to minimize $\gamma(p)$ over $\Delta_r$.
\begin{theorem}
Let $S$ be a complete discrete sampling and let $\overline{S}_i\in \R^{n\times q_i}$, for $i=1,2,\dots,r$, be such that 
$\mathbf{S}^{-T} = [\overline{S}_1, \ldots, \overline{S}_r]$.
 Then
\begin{equation}\label{eq:discoptrate}
\min_{p\in \Delta_r}\gamma(p) \quad=\quad 1 -\left(\sum_{i=1}^r \|W^{1/2}A^T S_i \overline{S}^T_i  A^{-T} W^{-1/2}\|_F\right)^{-2} .
\end{equation}
\end{theorem} 
 \begin{proof} In view of \eqref{eq:98h9s8h8s}, minimizing $\gamma$ in $p$ is equivalent to minimizing the expression \[\Tr{W^{-1/2} (\E{Z_p})^{-1} W^{-1/2}}\] in $p$.
Further, we have
\begin{eqnarray}
 \Tr{W^{-1/2} (\E{Z_p})^{-1}W^{-1/2}} &\overset{\eqref{eq:EZdiscrete}}{=} & \Tr{W^{-1/2} ( A^T \mathbf{S} D^2 \mathbf{S}^T  A)^{-1}W^{-1/2}} \nonumber \\
 &=& \Tr{W^{-1/2} A^{-1} \mathbf{S}^{-T} D^{-2} \mathbf{S}^{-1}  A^{-T} W^{-1/2}} \nonumber \\
&  \overset{\eqref{eq:D}}{=}&  \sum_{i=1}^r \frac{1}{p_i}\Tr{W^{-1/2} A^{-1} \overline{S}_i (S_i^T A WA^T S_i) \overline{S}^T_i  A^{-T} W^{-1/2}} \nonumber\\
&=& \sum_{i=1}^r \frac{1}{p_i}\norm{W^{1/2}A^{-1} \overline{S}_i S_i^T  A W^{-1/2}}_F^2.
\label{eq:pifracsum}
\end{eqnarray}

Applying Lemma~\ref{lem:fracsum} in the Appendix, the minimum of the above  subject to the constraint $p\in \Delta_r$ is  given by
\begin{equation} \label{eq:discoptprob}
p_i = \frac{\norm{W^{1/2}A^{-1}\overline{S}_i S_i^T  A W^{-1/2}}_F}{\sum_{j=1}^r \norm{W^{1/2}A^{-1}\overline{S}_j S_j^T  A W^{-1/2}}_F}, \quad i=1,2,\dots,r
\end{equation}
Plugging this into~\eqref{eq:pifracsum} gives the result~\eqref{eq:discoptrate}.
 \end{proof}
 
Observe that in general, the optimal probabilities~\eqref{eq:discoptprob} cannot be calculated, since the formula involves the inverse of $A$, which is not known. However, if $A$ is symmetric positive definite, we can choose $W=A^2$, which eliminates this issue. If $A$ is not symmetric positive definite, or if we do not wish to choose $W=A^2$, we can approach the formula \eqref{eq:discoptprob}  as a recipe for a heuristic choice of the probabilities: we can use the iterates $\{X_{k}\}$ as a proxy for $A^{-1}$. With this setup, the resulting method is not guaranteed to converge by the theory developed in this paper. However, in practice one would expect it to work well. We have not done extensive experiments to test this, and leave this to future research. To illustrate, let us consider a concrete simple example. Choose $W =I$ and $S_i =e_i$ (the unit coordinate vector in $\R^n$). We have $\mathbf{S} = [e_1,\dots,e_n] = I$, whence $\overline{S}_i =e_i$ for $i=1,\dots,r$. Plugging into \eqref{eq:discoptprob},   we obtain 
\[p_i = \frac{\norm{X_k e_i e_i^T  A}_F}{\sum_{j=1}^r \norm{X_k e_j e_j ^T A}_F} = \frac{\norm{X_k e_i}_2\norm {e_i^T  A}_2}{\sum_{j=1}^r \norm{X_k e_j}_2\norm{ e_j ^T A}_2}.\]


\subsection{Convenient sampling} \label{sec:discreteconv}

We now ask the following question:  given matrices $S_1, \dots, S_r$ defining a complete  discrete sampling, assign probabilities $p_i$ to  $S_i$ so that the convergence rate $\rho$ becomes {\em easy to interpret}. The following result, first stated in \cite{Gower2015} in the context of solving linear systems, gives a {\em convenient} choice of probabilities resulting in  $\rho$ which depends on a (scaled) condition number of  $A$. 

\begin{proposition} \label{theo:convenconv}
Let $S$ be a complete discrete sampling, where $S =S_i$ with probability  
\begin{equation} \label{eq:convprob}
p_i =\frac{\|W^{1/2}A^T S_i\|_F^2}{\|W^{1/2}A^T\mathbf{S}\|_F^2}.
\end{equation}
Then the convergence rate takes the form
\begin{equation}
\rho = 1- \frac{1}{\kappa_{2,F}^2(W^{1/2}A^T\mathbf{S})}
, \qquad \text{where} \label{eq:rhoconv}  
\end{equation}
\begin{align}\label{eq:kappalower}
\kappa_{2,F}(W^{1/2}A^T\mathbf{S}) &\eqdef \|(W^{1/2}A^T\mathbf{S})^{-1}\|_2 \|W^{1/2}A^T\mathbf{S}\|_F  =
\sqrt{ \frac{\Tr{\mathbf{S}^TAWA^T\mathbf{S}}}{\lambda_{\min}\left(\mathbf{S}^T A WA^T\mathbf{S} \right)} } \geq \sqrt{n}.
\end{align}
\end{proposition}
\begin{proof}
Theorem~5.1 in~\cite{Gower2015} gives \eqref{eq:rhoconv}. The bound in~\eqref{eq:kappalower} follows trivially.
\end{proof}

Following from Remark~\ref{rem:alg2conv}, we can determine a convergence rate for 
Algorithm~\ref{alg:asym-col} based on the Theorem~\ref{theo:convenconv}. 
\begin{remark}\label{rem:convcols}
  Let $S$ be a complete discrete sampling where $S =S_i$ with probability  
\begin{equation} \label{eq:convprob2}
p_i = \frac{\|W^{1/2}A S_i\|_F^2}{ \|W^{1/2}A\mathbf{S}\|_F^2}.
\end{equation}
Then Algorithm~\ref{alg:asym-col}  converges at the rate  
\begin{equation}
\label{eq:rhoconv2}
\rho_2 = 1- \frac{1}{ \kappa_{2,F}^2(W^{1/2}A\mathbf{S})}.\end{equation}
\end{remark}

\subsection{Optimal and adaptive samplings}

Having decided on the probabilities $p_1,\dots,p_r$ associated with the matrices $S_1,\dots,S_r$ in Proposition~\ref{theo:convenconv}, we can now ask the following question. How should we choose the matrices $\{S_i\}$ if we want $\rho $ to be as small as possible?  Since the rate improves as the {\em condition number} $\kappa^2_{2,F}(W^{1/2}A^T\mathbf{S})$ decreases, we should aim for matrices that minimize the condition number. Notice that the lower bound in \eqref{eq:kappalower} is reached for $\mathbf{S} = (W^{1/2}A^T)^{-1} = A^{-T}W^{-1/2}$. While we do not know $A^{-1}$, we can use our best current approximation of it, $X_k$, in its place. This leads to a method which {\em adapts} the probability distribution governing $S$ throughout the iterative process. This observation inspires a very efficient modification of Algorithm~\ref{alg:sym}, which we call AdaRBFGS (Adaptive Randomized BFGS), and describe in Section~\ref{sec:AdaRBFGS}. Notice that, luckily and surprisingly, our twin goals of computing the inverse and optimizing the convergence rate via the above adaptive trick are compatible.  Indeed, we wish to find $A^{-1}$, whose knowledge gives us the  optimal rate. This should be contrasted with the SDP approach mentioned earlier: i) the SDP could potentially be harder than the inversion problem, and ii) having found the optimal probabilities $\{p_i\}$, we are still not guaranteed the optimal rate. Indeed, optimality is relative to the choice of the matrices $S_1,\dots,S_r$, which can be suboptimal.

\begin{remark}[Adaptive sampling]
The convergence rate~\eqref{eq:rhoconv} suggests how one can select a sampling distribution for $S$ that would result in faster practical convergence. We now detail several practical choices for $W$ and indicate how to sample $S$. These suggestions require that the distribution of $S$ depends on the iterate $X_k$, and thus no longer fit into our framework. Nonetheless, we collect these suggestions here in the hope that others will wish to extend these ideas further, and as a  demonstration of the utility of developing convergence rates.  \\
\indent a) If $W =I$, then Algorithm~\ref{alg:asym-row} converges at the rate $\rho = 1- 1/\kappa_{2,F}^2(A^T\mathbf{S})$, and hence $S$ should be chosen so that $\mathbf{S}$ is a preconditioner of $A^T$. For example $\mathbf{S}=X_{k}^T,$ that is, $S$ should be a sampling of the rows of $X_{k}$. \\
\indent b) If $W =I$, then Algorithm~\ref{alg:asym-col} converges at the rate $\rho = 1- 1/\kappa_{2,F}^2(A\mathbf{S})$, and hence $S$ should be chosen so that $\mathbf{S}$ is a preconditioner of $A$. For example $\mathbf{S}=X_{k}$; that is, $S$ should be a sampling of the columns of $X_{k}$. \\
\indent c) If $A $ is symmetric positive definite, we can choose $W=A^{-1}$, in which case Algorithm~\ref{alg:sym} converges at the rate $\rho =1- 1/\kappa_{2,F}^2(A^{1/2}\mathbf{S}).$  This rate suggests that $S$ should be chosen so that $\mathbf{S}$ is an approximation of $A^{-1/2}.$ In Section~\ref{sec:AdaRBFGS} we develop this idea further, and design the AdaRBFGS algorithm.\\
\indent d) If $W=(A^TA)^{-1}$, then Algorithm~\ref{alg:asym-row} can be efficiently implemented with $S=AV$, where $V$ is a complete discrete sampling. Furthermore, $\rho = 1- 1/\kappa_{2,F}^2(A\mathbf{V}),$ where  $\mathbf{V} = [V_1,\ldots, V_r]$.  This rate suggests that $V$ should be chosen so that  $\mathbf{V}$ is a preconditioner of  $A$. For example, we can set $\mathbf{V}=X_{k}$, i.e., $V$ should be a sampling of the rows of $X_{k}$. \\
\indent e) If $W=(AA^T)^{-1}$, then Algorithm~\ref{alg:asym-col} can be efficiently implemented with $S=A^TV$, where $V$ is a complete discrete sampling. From Remark~\ref{rem:convcols}, the convergence rate of the resulting method is given by
$1- 1/ \kappa_{2,F}^2(A^T\mathbf{V}).$
This rate suggests that $V$ should be chosen so that  $\mathbf{V}$ is a  preconditioner of  $A^T$. For example, $\mathbf{V}=X_{k}^T$; that is, $V$ should be a sampling of the columns of $X_{k}$. \\
\indent f) If $A$ is symmetric positive definite, we can choose $W = A^2$, in which case Algorithm~\ref{alg:sym} can be efficiently implemented with  $S =AV.$ Furthermore $\rho = 1- 1/\kappa_{2,F}^2(A\mathbf{V}).$ This rate suggests that $V$ should be chosen so that  $\mathbf{V}$ is a  preconditioner of  $A$. For example $\mathbf{V}=X_{k},$ that is, $V$ should be a sampling of the rows or the columns of $X_{k}$. 
\end{remark}

\section{Randomized Quasi-Newton Updates} \label{sec:discretemethods}

Algorithms~\ref{alg:asym-row}, \ref{alg:asym-col} and \ref{alg:sym} are in fact families of algorithms indexed by the two parameters: i) positive definite matrix $W$ and ii) distribution $\cal D$ (from which we pick random matrices  $S$). This allows us to design a myriad of specific methods by varying these parameters. Here we highlight some of these possibilities, focusing on complete discrete distributions for $S$ so that convergence of the iterates is guaranteed through Theorems~\ref{theo:normEconv} and~\ref{theo:Enormconv}.
We also compute the convergence rate for these special methods for the convenient probability distribution given by~\eqref{eq:convprob}
and~\eqref{eq:convprob2} (Proposition~\ref{theo:convenconv}) so that the convergence rates~\eqref{eq:rhoconv} and~\eqref{eq:rhoconv2}, respectively, depend on a scaled condition number which is easy to interpret. We will also make some connections to existing qN  and Approximate Inverse Preconditioning methods. Table~\ref{tbl:QN} provides a guide through this section.

\begin{table}[!h]
\begin{center}
{
\footnotesize
\begin{tabular}{|c|c|c|c|c|c|}
\hline
$A$                                          & $W$      & $S$ &  Inverse Equation & Randomized Update & Section\\
\hline
any        &  any & inv. & any & One Step & \ref{subsec:1step}\\
any        &  $ I$ & $e_i$ & $AX = I$ & Simultaneous Kaczmarz (SK) & \ref{subsec:09u09u09Kacz}\\
any       &  $I$ & vector & $XA = I$ & Bad Broyden (BB) & \ref{subsec:09u09u09}\\
sym. & $I$ & vector & $AX = I, X=X^T$ & Powell-Symmetric-Broyden (PSB) & \ref{subsec:PSB} \\
any       &  $ I$ & vector & $XA^{-1} = I$ & Good Broyden (GB) & \ref{subsec:09u09u09good}\\
s.p.d.       &  $A$ & vector & $XA^{-1} =I , X=X^T$ & Davidon-Fletcher-Powell (DFP) & \ref{sec:RDFP} \\
s.p.d.       &  $A^{-1}$ & vector & $AX = I, X=X^T$ & Broyden-Fletcher-Goldfarb-Shanno (BFGS) & \ref{sec:RBFGS} \\
any      & $(A^T A)^{-1}$ & vector & $AX=I$ & Column & \ref{sec:RCM}\\
\hline
\end{tabular}
}
\end{center}
\caption{\footnotesize Specific randomized updates for inverting matrices discussed in this section, obtained as special cases of our algorithms. First column: ``sym'' means ``symmetric'' and ``s.p.d.'' means ``symmetric positive definite''; Third column:  ``inv'' means invertible. Block versions of all these updates are obtained by choosing $S$ as a matrix with more than one column.}\label{tbl:QN}
\end{table}
 

\subsection{One Step Update}\label{subsec:1step}

We have the freedom to select $S$ as almost any random matrix that has full column rank. This includes choosing $S$ to be a constant and invertible matrix, such as the identity matrix $I$,  in which case 
$X_{1}$ must be equal to the inverse. Indeed, the sketch-and-project formulations of all our algorithms reveal that. For Algorithm~\ref{alg:asym-row}, for example, the sketched system is $S^T AX = S^T$, which is equivalent to $AX=I$, which has as its unique solution $X=A^{-1}$.  Hence, $X_1 = A^{-1}$, and we have convergence in one iteration/step. Through inspection of the complexity rate, we see that $W^{1/2}\E{Z}W^{1/2} = I$ and  $\rho = \lambda_{\min}(W^{1/2}\E{Z}W^{1/2}) =1$, thus this one step convergence  is predicted in theory by Theorems~\ref{theo:normEconv} and~\ref{theo:Enormconv}.

\subsection{Simultaneous Randomized Kaczmarz Update}\label{subsec:09u09u09Kacz}

Perhaps the most natural choice for the weighting matrix $W$ is the identity $W=I.$ With this choice, Algorithm~\ref{alg:asym-row} is equivalent to applying the randomized  Kaczmarz update simultaneously to the $n$ linear systems encoded in $AX=I$. To see this, 
note that the sketch-and-project viewpoint~\eqref{eq:NF} of Algorithm~\ref{alg:asym-row} is
 \begin{align} \label{eq:NFbadbroyadj}
X_{k+1} = &\arg \min_{X \in \R^{n\times n}} \frac{1}{2}\norm{X - X_{k}}_{F}^2 \quad \mbox{subject to }\quad  S^TAX =S^T,
\end{align} 
which, by~\eqref{eq:Xupdate}, results in the explicit update 
\begin{equation} \label{eq:badbroyadj}
X_{k+1} = X_{k}+ A^TS(S^TAA^TS)^{-1}S^T (I-AX_{k}).
\end{equation}
If $S$ is a random coordinate vector, then~\eqref{eq:NFbadbroyadj} is equivalent to projecting the $j$th column of $X_k$ onto the solution space of $A_{i:}x = \delta_{ij},$ which is exactly an iteration of the randomized Kaczmarz update applied to solving $Ax=e_j.$
In particular, if $S = e_i$ with probability $p_i = \norm{A_{i:}}_2^2/\norm{A}_F^2$ then according to Proposition~\ref{theo:convenconv}, the rate of convergence of update~\eqref{eq:badbroyadj} is given by
\[\E{\norm{X_{k} -A^{-1} }_F^2} \leq \left(1- \frac{1}{\kappa_{2,F}^2(A)}\right)^k \norm{X_{0} -A^{-1} }_F^2,\]
where we used that $\kappa_{2,F}(A) = \kappa_{2,F}(A^T).$ This is exactly the rate of convergence given by Strohmer and Vershynin in~\cite{Strohmer2009} for the randomized Kaczmarz method.

\subsection{Randomized Bad Broyden Update}\label{subsec:09u09u09}

The update~\eqref{eq:badbroyadj} can also be viewed as an adjoint form of the bad Broyden update~\cite{Broyden1965,Griewank2012}. To see this,
if we use  Algorithm~\ref{alg:asym-col} with $W=I$, then the iterative process is
\begin{equation} \label{eq:badbroy}
X_{k+1} = X_{k}+ (I-X_{k}A)S(S^TA^TAS)^{-1}S^TA^T. 
\end{equation}
 This update~\eqref{eq:badbroy} is a randomized block form of the {\em bad Broyden update}~\cite{Broyden1965,Griewank2012}. In the qN setting, $S$ is not random, but rather the previous step direction $S = \delta \in \R^n$. Furthermore, if we rename $\gamma \eqdef AS\in \R^n$, then~\eqref{eq:badbroy} becomes  
\begin{equation}\label{eq:badbroyden}
X_{k+1} = X_{k}+ \frac{\delta-X_k\gamma}{\norm{\gamma}_2^2}\gamma^T,
\end{equation}
which is the standard way of writing the bad Broyden update~\cite{Griewank2012}.
Update~\eqref{eq:badbroyadj} is an adjoint form of bad Broyden in the sense that, if we transpose~\eqref{eq:badbroyadj}, set $S = \delta$ and write $\gamma = A^TS$, we obtain the bad Broyden update, but applied to $X_k^T$  instead.

From the constrain-and-approximate viewpoint~\eqref{eq:RFcols} we give a new interpretation to the bad Broyden update, namely, the update~\eqref{eq:badbroyden} can be written as
\[X_{k+1} = \arg_X \min_{X \in \R^{n\times n}, \; y\in \R^n} \frac{1}{2}\norm{X - A^{-1}}_{F}^2 \quad \mbox{subject to } \quad  X=X_{k} +y \gamma^T. \]
Thus, {\em the bad Broyden update is the best rank-one update (with fixed $\gamma$) approximating the inverse.} We can determine the rate at which our randomized variant of the BB update~\eqref{eq:badbroy} converges by using Remark~\ref{rem:convcols}.
In particular, if $S =S_i$ with probability   $p_i = \left.\norm{A S_i}_F^2\right/\norm{A\mathbf{S}}_F^2$, then~\eqref{eq:BFGSas} converges  with the rate
\[\E{\norm{X_{k} -A^{-1} }_F^2} \leq \left(1- \frac{1}{\kappa_{2,F}^2(A\mathbf{S})}\right)^k \norm{X_{0} -A^{-1} }_F^2. \]


\subsection{Randomized Powell-Symmetric-Broyden Update}\label{subsec:PSB}

If $A$ is symmetric and we use  Algorithm~\ref{alg:sym} with $W=I$, the iterates are given by 
\begin{align}
 X_{k+1} &= X_{k} +AS^T (S^TA^2S)^{-1}SA
(X_{k}AS -S)\left((S^TA^2S)^{-1} S^T A-I \right) \nonumber \\
&-(X_{k}AS -S)(S^TA^2S)^{-1} S^T A, \label{eq:PSB}
 \end{align}
which is a randomized block form of the Powell-Symmetric-Broyden update~\cite{Gower2014c}.
If $S =S_i$ with probability $p_i = \norm{AS_i}_F^2/\norm{A\mathbf{S}}_F^2$, then according to Proposition~\ref{theo:convenconv}, the iterates~\eqref{eq:PSB} and~\eqref{eq:badbroyadj} converge according to
\[\E{\norm{X_{k} -A^{-1} }_F^2} \leq \left(1  - \frac{1}{\kappa^2_{2,F}(A^T\mathbf{S})}\right)^k \norm{X_{0} - A^{-1}}_F^2.\]

\subsection{Randomized Good Broyden Update}
\label{subsec:09u09u09good}

Next we present a method that shares certain properties with Gaussian elimination and can be viewed as a randomized block variant of the good Broyden update~\cite{Broyden1965,Griewank2012}. This method requires the following adaptation of Algorithm~\ref{alg:asym-col}: instead of sketching the inverse equation, consider the update~\eqref{eq:guasssketch} that performs a column sketching of the equation $XA^{-1} = I$ by right multiplying with $Ae_i$, where $e_i$ is the $i$th coordinate vector. Projecting an  iterate $X_k$ onto this sketched equation gives 
\begin{align}\label{eq:guasssketch}
X_{k+1} = &\arg \min_{X \in \R^{n\times n}} \frac{1}{2} \norm{X - X_{k}}_{F}^2 \quad \mbox{subject to }\quad  Xe_i =Ae_i.
\end{align}
 The iterates defined by the above are given by
\begin{equation}\label{eq:gaussinv}X_{k+1}=X_{k} +(A-X_k)e_ie_i^T.
\end{equation}
Given that we are sketching and projecting onto the solution space of $XA^{-1}=I$, the iterates of this method converge to $A$. Therefore the inverse iterates $X^{-1}_k$ converge to $A^{-1}.$ We can efficiently compute the inverse iterates 
by using the Woodbury formula~\cite{Woodbury1950} which gives
\begin{equation}\label{eq:gauss}
X_{k+1}^{-1} = X_k^{-1}-\frac{(X_k^{-1}A-I)e_ie_i^TX_k^{-1}}{e_i^TX_k^{-1}Ae_i}.
\end{equation}
This update~\eqref{eq:gauss} behaves like Gaussian elimination in the sense that, if $i$ is selected in a cyclic fashion, that is $i=k$ on the $k$th iteration, then from~\eqref{eq:gaussinv} it is clear that 
\[ X_{k+1}e_i = Ae_i, \quad \mbox{thus} \quad X_{k+1}^{-1}Ae_i =e_i, \quad \mbox{for }i=1\ldots k.\]
That is, on the $k$th iteration, the first $k$ columns of the matrix $X_{k+1}^{-1}A$ are equal to the first $k$ columns of the identity matrix. Consequently,  $X_n = A$ and $X^{-1}_n = A^{-1}.$ If instead, we select $i$ uniformly at random, then we can adapt Proposition~\ref{theo:convenconv} by swapping each occurrence of $A^T$ for $A^{-1}$ and observing that $S_i = Ae_i$ thus $\mathbf{S}=A$. Consequently the iterates~\eqref{eq:gaussinv} converge to $A$  at a rate of
\[\rho = 1-\kappa_{2,F}^2 \left(A^{-1}A\right)= 1-\frac{1}{n}, \]
and thus the lower bound~\eqref{eq:rholower} is achieved and $X_k$ converges to $A$ according to
\[\E{\norm{X_{k} -A }_F^2} \leq \left(1-\frac{1}{n}\right)^k \norm{X_{0} -A}_F^2. \]
Note that this does not say anything about how fast 
$X^{-1}_k$ converges to $A^{-1}$. Therefore~\eqref{eq:gauss} is not an efficient method for calculating an approximate inverse. If we replace $e_i$  by a \emph{step} direction $\delta_k \in \R^d$, then the update~\eqref{eq:gauss} is known as the {\em good Broyden} update~\cite{Broyden1965,Griewank2012}.

\subsection{Approximate Inverse Preconditioning}\label{sec:AIP} 

When $A$ is symmetric positive definite, we can choose $W =A^{-1}$, and Algorithm~\ref{alg:asym-row} is given by
\begin{equation} \label{eq:BFGSas} X_{k+1} = X_{k} +S(S^TAS)^{-1}S^T(I-AX_{k}). \end{equation}
The constrain-and-approximate viewpoint~\eqref{eq:RF} of this update is
 \[X_{k+1} = \arg_X \min_{X\in \R^{n\times n}, Y\in \R^{n\times q}} \frac{1}{2}\norm{A^{1/2}XA^{1/2} - I}_{F}^2 \quad \mbox{subject to } \quad  X=X_{k} + S Y^T.\]
This viewpoint reveals that the update~\eqref{eq:BFGSas} is akin to the Approximate Inverse Preconditioning (\emph{AIP}) methods~\cite{Benzi1999,Gould1998,Kolotilina1993,Huckle2007}.


 We can determine the rate a which~\eqref{eq:BFGSas}  converges using~\eqref{eq:rhoconv}. In particular, if $S =S_i$ with probability   $p_i = \Tr{S_i^T A S_i}/\Tr{\mathbf{S}^TA\mathbf{S}}$, then~\eqref{eq:BFGSas} converges with rate 
\begin{equation}
\rho \overset{\eqref{eq:rhoconv}}{=} 1- \frac{1}{\kappa_{2,F}^2(A^{1/2}\mathbf{S})} = 1- \frac{\lambda_{\min}(\mathbf{S}^TA\mathbf{S})}{\Tr{\mathbf{S}^TA\mathbf{S}}}, \label{eq:BBonv}  
\end{equation}
and according to
\[\E{\norm{A^{1/2}X_k A^{1/2} - I }_F^2} \leq \left(1- \frac{\lambda_{\min}(\mathbf{S}^TA\mathbf{S})}{\Tr{\mathbf{S}^TA\mathbf{S}}}\right)^k \norm{A^{1/2}X_0A^{1/2} - I}_F^2.\]

%
%
%
%

\subsection{Randomized DFP Update}\label{sec:RDFP}

If $A$ is symmetric positive definite then we can choose $W =A.$ Furthermore, if we adapt the sketch-and-project formulation~\eqref{eq:NF} to sketch the equation $XA^{-1} = I$ by right multiplying by $AS,$ and additionally impose symmetry on the iterates, we arrive at the following update:
\begin{align}\label{eq:DFPsketch}
X_{k+1} = &\arg \min_{X \in \R^{n\times n}} \frac{1}{2} \norm{X - X_{k}}_{F(A^{-1})}^2 \quad \mbox{subject to }\quad  XS =AS, \quad X =X^T.
\end{align}
The solution to the above is given by\footnote{To arrive at this solution, one needs to swap the occurrences of $AS$ for $S$ and plug in  $B=A^{-1}$ in~\eqref{eq:Xupdatesym}. This is because  swapping $AS$ for $S$ in~\eqref{eq:DFPsketch} gives~\eqref{eq:NFsym}.}
\begin{equation}\label{eq:DFPinv}
X_{k+1}= A\Omega A +\left( I- A\Omega 
\right)X_k\left(I -\Omega A\right), \quad \text{where} \quad \Omega = S (S^TAS)^{-1} S^T.
\end{equation}
Using the Woodbury formula~\cite{Woodbury1950}, we find that 
\begin{equation}\label{eq:DFP} X_{k+1}^{-1} = X_{k}^{-1} + AS (S^TAS)^{-1} S^TA  -
X_k^{-1}  S \left( S^T X_k^{-1}  S  \right)^{-1} S^T X_k^{-1}.\end{equation}
The update~\eqref{eq:DFP} is a randomized variant of the Davidon-Fletcher-Powell (DFP) update~\cite{Davidon1959,Fletcher1960}. We can adapt Proposition~\ref{theo:convenconv} to determine the rate at which $X_k$ converges to $A$ by swapping each occurrence of $A^T$ for $A^{-1}$. If we let $S_i= Ae_i$  with probability $ p_i = \left. \lambda_{\min}(A) \right/ \Tr{A},$ for instance,
then the iterates~\eqref{eq:gaussinv} satisfy
\begin{equation}\label{eq:convDFP}
\E{\norm{X_{k} -A }_{F(A^{-1})}^2} \leq \left(1-\frac{\lambda_{\min}(A)}{\Tr{A}}\right)^k \norm{X_{0} - A}_{F(A^{-1})}^2.
\end{equation}
Thus $X_k$ converges to $A$. However, this does not indicate at what rate does $X_{k}^{-1}$ converge to $A^{-1}$. This is in contrast with  randomized BFGS, which produces iterates that converge to $A^{-1}$ at this same rate, as we show in the next section.   This perhaps sheds new light on why BFGS update performs better than the DFP update.

\subsection{Randomized BFGS Update}\label{sec:RBFGS} 

 
If $A$ is symmetric and positive definite, we can  choose $W=A^{-1}$ and apply  Algorithm~\ref{alg:sym} to maintain symmetry of the iterates. The iterates are given by
\begin{equation}\label{eq:qunac}
X_{k+1}  =S(S^TAS)^{-1}S^T+ \left(I-S(S^TAS)^{-1}S^TA\right) X_{k} \left(I -AS(S^TAS)^{-1}S^T \right).
\end{equation}
This is a block variant, see~\cite{Gower2014c}, of the BFGS update~\cite{Broyden1965,Fletcher1960,Goldfarb1970,Shanno1971}.
The constrain-and-approximate viewpoint gives a new interpretation to the block BFGS update. That is, from~\eqref{eq:NFsym}, the iterates~\eqref{eq:qunac} can be equivalently defined by
\[X_{k+1} = \arg_X  \min_{X\in \R^{n\times n}, Y\in \R^{n\times q}} \frac{1}{2}\norm{X -A^{-1}}_{F(A)}^2 \quad \mbox{subject to} \quad
X = X_k + SY^T +YS^T.\]

Thus, the standard and  block BFGS updates produced an improved estimate by doing best approximation to the  inverse  subject to a particular symmetric affine space passing through the current iterate. This is a  new way of interpreting BFGS.

If $p_i = \Tr{S_i^TAS_i}/\Tr{\mathbf{S}A\mathbf{S}^T}$, then according to Proposition~\ref{theo:convenconv}, the update~\eqref{eq:qunac}  converges according to
\begin{equation}\label{eq:convqunac}
\E{\norm{X_k -A^{-1}}_{F(A)}^2} \leq \left(1 - \frac{1}{\kappa^2_{2,F}(A^{1/2}\mathbf{S})}\right)^k\norm{X_0 -A^{-1}}_{F(A)}^2.
\end{equation}

It is remarkable that the update~\eqref{eq:qunac} preserves positive definiteness of the iterates.  Indeed, assume that $X_k$ is positive definite and let $v \in \R^n$ and $P \eqdef S(S^TAS)^{-1}S^T.$ Left and right multiplying~\eqref{eq:qunac} by $v^T$ and $v$, respectively, gives
\[v^TX_{k+1}v = v^TPv+ v^T\left(I-PA\right) X_{k} \left(I -AP\right)v \geq 0.\]
Thus $v^TX_{k+1}v =0$ implies that $Pv=0$ and $ \left(I -AP \right)v=0,$  which when combined gives $v =0.$ This proves that $X_{k+1}$ is positive definite.
Thus the update~\eqref{eq:qunac} is particularly well suited for calculating the inverse of a positive definite matrices.

In Section~\ref{sec:AdaRBFGS}, we detail an update designed to improve the convergence rate in~\eqref{eq:convqunac}. The result is a method that is able to invert large scale positive definite matrices orders of magnitude faster than the state-of-the-art.

\subsection{Randomized Column Update}\label{sec:RCM}
We now describe a method whose deterministic variant does not seem to exist in the literature. For this update, we need to perform a linear transformation of the sampling matrices. For this, let $V$ be a complete discrete sampling where $V= V_i\in \R^{n\times q_i}$ with probability $p_i>0,$ for $i =1,\ldots,r.$ Let $\mathbf{V} =[V_1,\ldots, V_r].$
 Let the sampling matrices be defined as $S_i=AV_i \in \R^{n \times q_i}$ for $i =1,\ldots,r$. As $A$ is nonsingular, and $\mathbf{S} = A \mathbf{V}$, then $S$ is a complete discrete sampling. With these choices and $W^{-1}=A^TA$, the sketch-and-project viewpoint~\eqref{eq:NF} is given by 
\begin{align}
X_{k+1} =\arg \min_{X \in \R^{n\times n}} \frac{1}{2}\norm{X - X_{k}}_{F(A^TA)}^2 \nonumber\quad  \mbox{subject to } \quad V_i^TA^TAX = V_i^TA^T.
\end{align} 
The solution to the above are the iterates of  Algorithm~\ref{alg:asym-row}, which is given by
\begin{equation}\label{eq:RRM}
X_{k+1} = X_{k} +V_i(V_i^T A^T AV_i)^{-1}V_i^T(A^T-A^TAX_{k}).
\end{equation}
 From the constrain-and-approximate viewpoint~\eqref{eq:RF}, this can be written as
 \[X_{k+1} = \arg  \min_{X\in \R^{n\times n}, Y\in \R^{n\times q}}  \frac{1}{2}\norm{A(XA^T-I)}_{F}^2 \quad \mbox{subject to } \quad  X=X_{k} + V_i Y^T.\]
 
With the same parameters $S$ and $W$, the iterates of Algorithm~\ref{alg:sym} are given by
    \begin{align}
X_{k+1}  &=X_{k} + V_i(V_i^TA^2 V_i)^{-1}V_i^T(AX_{k} -I)\left(A^2V_i(V_i^T A^2V_i)^{-1}V_i^T - I \right)\nonumber \\
& -(X_{k}A -I)AV_i(V_i^T A^2 V_i)^{-1}V_i^T. \label{eq:symROM}
\end{align}
If we choose
$p_i = \norm{(AA^T)^{-1/2}AA^TV_i}_F^2 /\norm{(AA^T)^{-1/2}AA^T\mathbf{V}}_F^2 = \norm{A^TV_i}_F^2/\norm{A^T\mathbf{V}}_F^2,$
 then according to Proposition~\ref{theo:convenconv}, the iterates~\eqref{eq:RRM} and~\eqref{eq:symROM} converge linearly in expectation to the inverse according to
\begin{equation}\label{eq:RRMconv}
\E{\norm{A(X_kA^T-I) }_{F}^2} \leq \left(1 - \frac{1}{\kappa^2_{2,F}(A\mathbf{V})}\right)^k \norm{A(X_0A^T-I)}_{F}^2.
\end{equation}
 There also exists an analogous ``row'' variant of~\eqref{eq:RRM}, which arises by using  Algorithm~\ref{alg:asym-col}, but we do not explore it here.

\section{AdaRBFGS: Adaptive Randomized BFGS}\label{sec:AdaRBFGS}

All the updates we have developed thus far use a sketching matrix $S$ that is sampled in an i.i.d.\ fashion from a fixed distribution $\cal D$ at each iteration. In this section we assume that $A$ is symmetric positive definite, and propose AdaRBFGS: a variant of the RBFGS update, discussed in Section~\ref{sec:RBFGS}, which  {\em adaptively} changes the distribution $\cal D$ throughout the iterative process. Due to this change,  Theorems~\ref{theo:normEconv} and~\ref{theo:Enormconv} and Proposition~\ref{theo:convenconv} are no longer applicable. Superior numerical efficiency of this update is verified through extensive numerical experiments in Section~\ref{sec:numerical}.  

\subsection{Motivation}
We now motivate the design of this new update by examining the convergence rate~\eqref{eq:convqunac} of the RBFGS iterates~\eqref{eq:qunac}. Recall that in RBFGS we choose $W=A^{-1}$ and $S=S_i$ with probability 
\begin{equation}\label{eq:ihsih0(0hI}p_i = \frac{\Tr{S_i^TAS_i}}{\Tr{\mathbf{S}A\mathbf{S}^T}}, \qquad i=1,2,\dots,r,\end{equation} 
where $S$ is a complete discrete sampling and ${\bf S} = [S_1,\dots,S_r]$, then
\[\rho = 1 - \frac{1}{\kappa_{2,F}^2(A^{1/2}\mathbf{S})} \overset{\eqref{eq:kappalower}}{=} 1 - \frac{\lambda_{\min}({\bf S}^T A {\bf S}) }{ \Tr{{\bf S}^T A {\bf S}}}.\]
Consider now the question of choosing the  matrix $\bf S$ in such a way that $\rho$ is as small as possible. Note that the  optimal choice is any $\bf S$ such that ${\bf S}^T A {\bf S}= I$. Indeed, then $\rho = 1-1/n$, and the lower bound~\eqref{eq:kappalower} is attained. For instance, the choice ${\bf S} = A^{-1/2}$ would be optimal. Hence, in each iteration we would choose $S$ to be a  random column (or random column submatrix) of $A^{-1/2}$. Clearly, this is not a feasible choice, as we do not know the inverse of $A$. In fact, it is $A^{-1}$  which we are trying to find! However, this  leads to the following  observation: {\em the goals of finding the inverse of $A$ and of designing an optimal distribution $\cal D$ are in synchrony.} 

\subsection{The algorithm}

While we do not know $A^{-1/2}$, we can use the iterates $\{X_k\}$ themselves to construct a good {\em adaptive} sampling. Indeed,  the iterates contain information about the inverse and hence we can use them to design a better sampling $S$. In order to do so, it will be useful to maintain a factored form of the iterates,  \begin{equation}\label{eq:X_k_98987H}X_k  = L_kL_k^T,\end{equation}
 where $L_k \in \R^{n\times n}$ is invertible. With this in place, let us choose $S$  to be a random column submatrix of $L_k$. In particular, let $C_1,C_2,\dots,C_r$  be nonempty subsets of $[n]=\{1,2,\dots,n\}$ forming a partition of $[n]$, and at iteration $k$ choose \begin{equation}\label{eq:siug98sUJi}S = L_k I_{:C_i} \eqdef S_i, \end{equation} with probability $p_i$ given by \eqref{eq:ihsih0(0hI} for $i=1,2,\dots,r$. 
For simplicity, assume that $C_1=\{1,\dots,c_1\}$, $C_2 = \{c_1+1,\dots,c_2\}$ and so on, so that, by the definition of $\bf S$, 
\begin{equation}\label{eq:S_9898y8}{\bf S} = [S_1, \dots, S_r] = L_k.\end{equation}
Note that now both $\bf S$ and $p_i$ depend on $k$.  The method described above satisfies the following recurrence.

\begin{theorem}\label{lem:9hs9hIKKK} After one step of the AdaRBFGS method we have
\begin{equation}\label{eq:98h98hJJJ}\E{\norm{X_{k+1} -A^{-1}}_{F(A)}^2\, | \, X_k} \leq \left(1-\frac{\lambda_{\min }(AX_k )}{\Tr{AX_k}}\right) \norm{X_k -A^{-1}}_{F(A)}^2. \end{equation}
\end{theorem}
\begin{proof}
Using the same arguments as those in the proof of Theorem~\ref{theo:Enormconv}, we obtain
\begin{equation}\label{eq:Ec}
\E{\norm{X_{k+1} -A^{-1}}_{F(A)}^2\, | \, X_k} \leq \left(1-\rho_k\right) \norm{X_k -A^{-1}}_{F(A)}^2,
\end{equation}
where $\rho_k\eqdef\lambda_{\min}\left(A^{-1/2}\E{Z \;|\; X_k}A^{-1/2}\right)$ and
   \begin{equation}\label{eq:Zxxx}
   Z \overset{\eqref{eq:Zxxx}}{=}AS_i(S_i^T A S_i)^{-1}S_i^T A.
\end{equation}
So, we only need to show that
$\rho_k \geq \lambda_{\min }(AX_k ) / \Tr{AX_k}$. Since $S$ is a complete discrete sampling, Proposition~\ref{prop:Ediscrete} applied to our setting says that
\begin{equation}\label{eq:hYhGfR6}\E{Z \;|\; X_k} =  A \mathbf{S} D^2 \mathbf{S}^T  A, \qquad \text{where} \end{equation}
 \begin{equation} \label{eq:Dxxx}
D~\eqdef~\mbox{Diag}\left( \sqrt{p_1}(S_1^T A  S_1)^{-1/2}, \ldots, \sqrt{p_r}(S_r^T A  S_r)^{-1/2}\right).\end{equation}
We now have
\begin{eqnarray*}
\rho_k &\overset{\eqref{eq:hYhGfR6}+\eqref{eq:S_9898y8}}{\geq}& \lambda_{\min}\left(A^{1/2}L_k L_k^T A^{1/2}\right) \lambda_{\min}(D^2) 
\quad \overset{\eqref{eq:X_k_98987H}}{=} \quad   \frac{\lambda_{\min }(AX_k )}{\lambda_{\max}(D^{-2})} \\
& \overset{\eqref{eq:Dxxx}}{=} & \frac{\lambda_{\min }(AX_k )}{\max_{i} \lambda_{\max}(S_i^T A S_i) / p_i}
\quad \geq \quad \frac{\lambda_{\min }(AX_k )}{\max_{i} \Tr{S_i^T A S_i} / p_i}
\overset{\eqref{eq:ihsih0(0hI} + \eqref{eq:S_9898y8}}{=} 
\frac{\lambda_{\min }(AX_k )}{\Tr{A X_k}}.
\end{eqnarray*}
In the second equality we have used the fact that the largest eigenvalue of  a block diagonal matrix is equal to the maximum of the largest eigenvalues of the blocks.
\end{proof}

If $X_k$ converges to $A^{-1}$, then necessarily the one-step rate of AdaRBFGS proved in Theorem~\ref{lem:9hs9hIKKK} asymptotically reaches the lower bound
\[\rho_k \eqdef  1 - \frac{\lambda_{\min}(AX_k)}{\Tr{A X_k}} \to 1-\frac{1}{n}.\]

In other words, as long as this method works, the convergence rate gradually improves, and  becomes asymptotically optimal and independent of the condition number. We leave a deeper analysis of this and other adaptive variants of the methods developed in this paper to future work.

\subsection{Implementation}

To implement the AdaRBFGS update, we need to maintain the iterates $X_k$ in the factored form \eqref{eq:X_k_98987H}. Fortunately,  a factored form of the update~\eqref{eq:qunac} was introduced in~\cite{Gratton2011}, which we shall now describe and adapt to our objective. Assuming that $X_k$ is symmetric positive definite such that $X_k = L_k L_k^T$, we shall describe how to obtain a corresponding factorization of $X_{k+1}$. Letting $G_k =  (S^TL_k^{-T}L_k^{-1}S)^{1/2}$ and $R_k= (S^TAS)^{-1/2}$, it can be verified through direct inspection~\cite{Gratton2011} that $X_{k+1} = L_{k+1} L_{k+1}^T$, where
\begin{equation}\label{eq:Chol}L_{k+1} = L_k +S R_k \left(G_k^{-1}S^T L_k^{-T}-R_k ^T S^TA L_k\right).
\end{equation}
If we instead of \eqref{eq:siug98sUJi} consider the more general update $S = L_k \tilde{S}$, where $\tilde{S}$ is chosen in an i.i.d.\ fashion from some fixed distribution $\cal \tilde{D}$, then 
\begin{equation}
 L_{k+1}  = L_k +L_k \tilde{S} R_k \left(({\tilde{S}}^T \tilde{S})^{-1/2}\tilde{S}^T -R_k^T \tilde{S}^T L_k^T A L_k\right). \label{eq:Cholimprov}
\end{equation}
The above can now be implemented efficiently, see Algorithm~\ref{alg:AdaRBFGS}. 

\begin{algorithm}[!h]
\begin{algorithmic}[1]
\State \textbf{input:} symmetric positive definite matrix $A$
\State \textbf{parameter:} ${\cal \tilde{D}}$ = distribution over random matrices with $n$ rows 
\State \textbf{initialize:} pick invertible $L_0 \in \R^{n\times n}$

\For {$k = 0, 1, 2, \dots$}
	\State Sample an independent copy $\tilde{S} \sim {\cal \tilde{D}}$
	\State Compute $S = L_k \tilde{S}$ 
	\Comment $S$ is sampled adaptively, as it depends on $k$
	\State Compute $R_k = (\tilde{S}^T A \tilde{S})^{-1/2}$ 
	\State $L_{k+1} = L_k +S R_k \left((\tilde{S}^T \tilde{S})^{-1/2}\tilde{S}^T -R_k^T S^T A L_k\right)$
	\Comment Update the factor
\EndFor
\State \textbf{output:} $X_k = L_k L_k^T$
\end{algorithmic}

\caption{ Adaptive Randomized BFGS  (AdaRBFGS)}
\label{alg:AdaRBFGS}
\end{algorithm}

In Section~\ref{sec:numerical} we test two variants based on~\eqref{eq:Cholimprov}. The first is  the \emph{AdaRBFGS\_gauss} update, in which the entries of $\tilde{S}$ are standard Gaussian. The second is \emph{AdaRBFGS\_cols}, where $\tilde{S} =I_{:C_i}$, as described above, and  $|C_i| =q$ for all $i$ for some $q$.

\section{Numerical Experiments} \label{sec:numerical}

Given demand for approximate inverses of positive definite matrices in preconditioning and in variable metric optimization methods, we restrict our tests to inverting positive definite matrices. The code for this section can downloaded from the authors' webpages, together with scripts that allow for the tests we describe here to be easily reproduced.

We test four iterative methods for inverting matrices. This rules out the all-or-nothing direct methods such as Gaussian elimination of LU based methods. For our tests we use two variants of  Algorithm~\ref{alg:AdaRBFGS}: AdaRBFGS\_gauss, where $\tilde{S}\in\R^{n\times q}$ is a normal Gaussian matrix, and  
AdaRBFGS\_cols, where $\tilde{S}$ consists of a collection of $q$ distinct coordinate vectors in $\R^n$, selected uniformly at random. 
At each iteration the AdaRBFGS methods compute the inverse of a small matrix $S^TAS$ of dimension $q\times q$.
To invert this matrix we use MATLAB's inbuilt \texttt{inv} function, which uses $LU$ decomposition or Gaussian elimination, depending on the input. Either way, \texttt{inv} costs $O(q^3).$ For simplicity, we selected $q=\sqrt{n}$ in all our tests. 


We compare our method to two well established and competitive methods, the \emph{Newton-Schulz method}~\cite{Schulz1933}
and the global self-conditioned Minimal Residual (\emph{MR}) method~\cite{Chow1998}.  The Newton-Schulz method arises 
from applying the Newton-Raphson method to solve the equation $X^{-1}=A$, which gives
\begin{equation}\label{eq:Newton-Schulz}
X_{k+1} = 2X_{k} -X_{k} A X_{k}.
\end{equation}
The MR method was designed to calculate  approximate inverses, and it does so by minimizing the norm of the residual along the preconditioned residual direction:
\begin{equation} \label{eq:MRdef}
\norm{I-AX_{k+1}}_F^2 = \min_{\alpha \in \R} \left\{\norm{I-AX}_F^2 \quad \mbox{subject to} \quad X = X_k +\alpha X_k(I-AX_k)\right\}.
\end{equation}
See~\cite[chapter 10.5]{Saad2003} for a didactic introduction to MR methods. Letting  $R_k = I-AX_k$,
the resulting iterates of the MR method are given by
\begin{equation} \label{eq:MRiter}
 X_{k+1} =X_k +\frac{\Tr{R_k^TAX_kR_k}}{\Tr{(AX_kR_k)^TAX_kR_k}}X_kR_k.
\end{equation}


We perform two sets of tests. On the first set, we choose a different starting matrix for each method which is optimized, in some sense, for that method. We then compare the empirical convergence of each method, including the time taken to calculate $X_0$. In particular, 
  the  Newton-Schulz is only guaranteed to converge for an initial matrix $X_0$ such that $\rho(I-X_0A)< 1$. Indeed, the Newton-Schulz method did not converge in most of our experiments when $X_0$ was not carefully chosen according to this criteria.
  To remedy this, we choose $X_0 = 0.99 \cdot A^T/\rho^2(A)$ for the Newton-Schulz method, so that $\rho(I-X_0A)<1$ is satisfied. To compute $\rho(A)$ we used the inbuilt MATLAB function \texttt{normest} which is coded in C++. For MR we followed the suggestion in~\cite{Saad2003} and used the projected identity for the initial matrix $X_0 = (\Tr{A}/\Tr{AA^T}) \cdot I.$ For our AdaRBFGS methods we simply used $X_0 =I$, as this worked well in practice.
  
  In the second set of tests, which we relegate to the Appendix,  we compare the empirical convergence of the methods starting from the same matrix, namely the identity matrix $X_0 = I$.

We run each method until the relative error $\norm{I-AX_{k}}_{F}/\norm{I-AX_0}_{F}$ is below $10^{-2}.$  All experiments were performed  and run in MATLAB R2014b. To appraise the performance of each method we plot the relative error against time taken and against the number of floating point operations (\emph{flops}).


\subsection{Experiment 1: synthetic matrices}
First we compare the four methods on synthetic matrices generated using the matrix function {\tt rand}. To appraise the difference in performance of the methods as the dimension of the problem grows, we tested for $n=1000$, $2000$ and $5000.$ As the dimension grows, only the two variants of the AdaRBFGS method are able to reach the $10^{-2}$ desired tolerance in a reasonable amount time and number of flops (see Figure~\ref{fig:pdsynth}).

\begin{figure}
    \centering
    \begin{subfigure}[t]{0.46\textwidth}
        \centering
\includegraphics[width =  \textwidth, trim= 50 310 60 310, clip ]{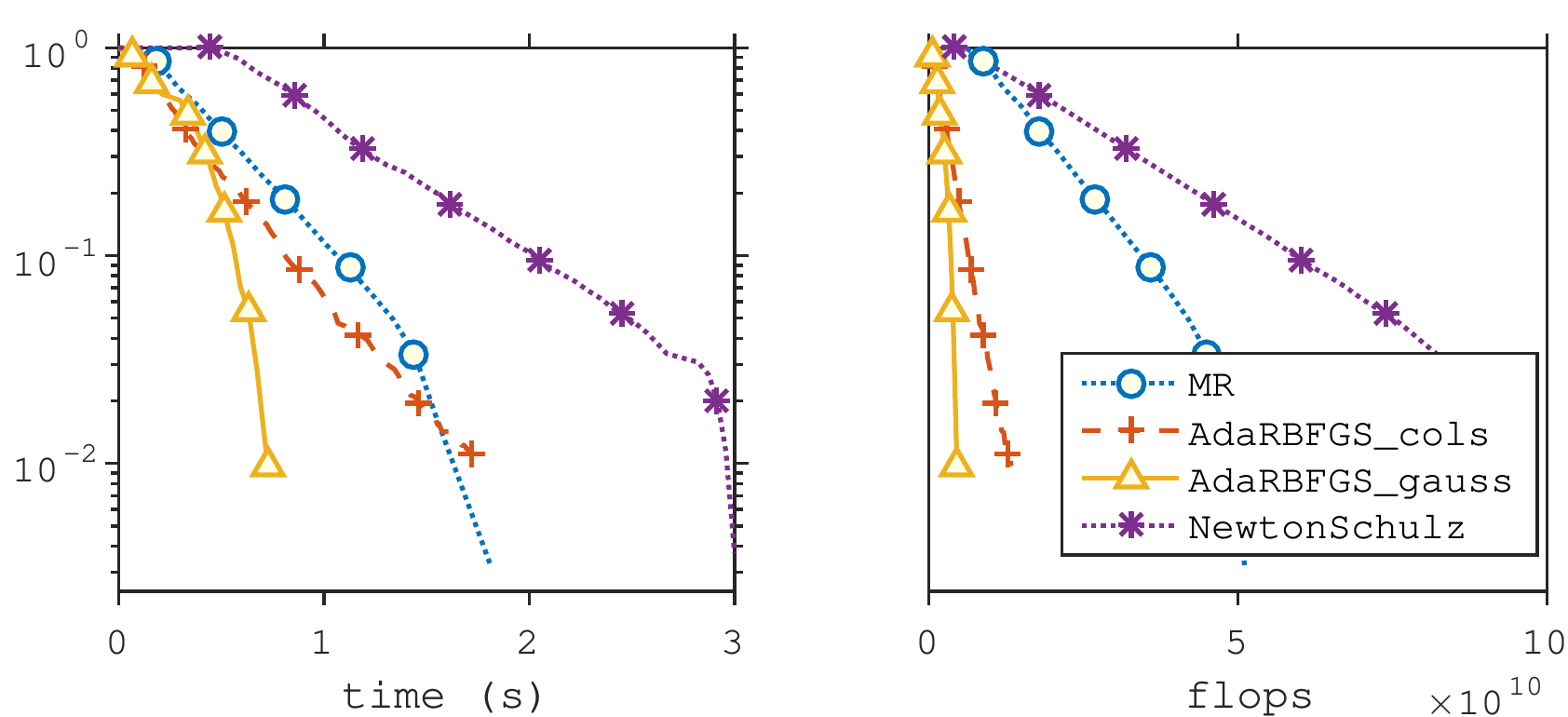}
        \caption{\texttt{rand} with $n=10^4$ }\label{fig:rand}
    \end{subfigure}%
     \hspace{0.05\textwidth}
         \begin{subfigure}[t]{0.46\textwidth}
        \centering
\includegraphics[width =  \textwidth, trim= 50 310 60 310, clip ]{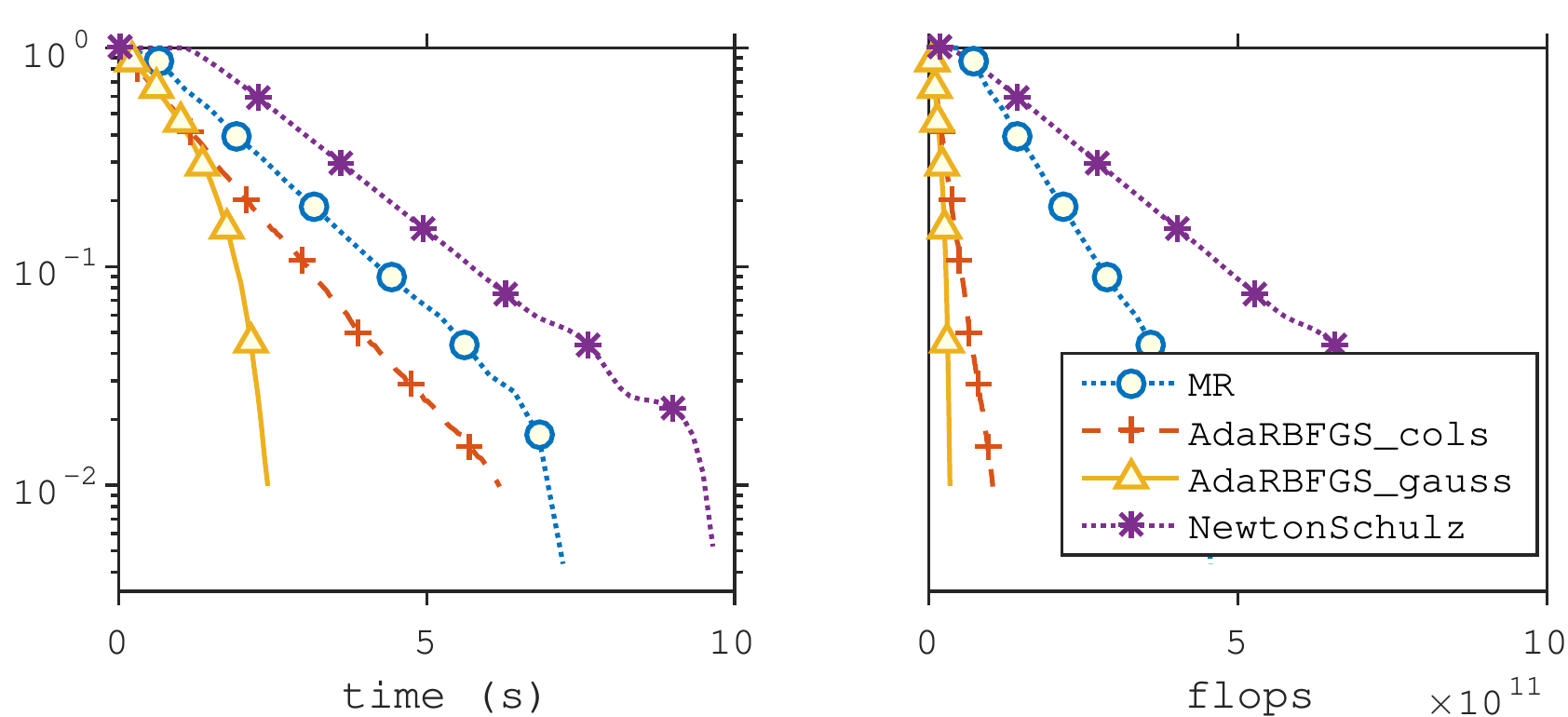}
        \caption{\texttt{rand} with $n=2 \cdot 10^4$ }\label{fig:rand2}
    \end{subfigure}%
     \hspace{0.01\textwidth}
         \begin{subfigure}[t]{0.46\textwidth}
        \centering
\includegraphics[width =  \textwidth, trim= 50 310 60 310, clip ]{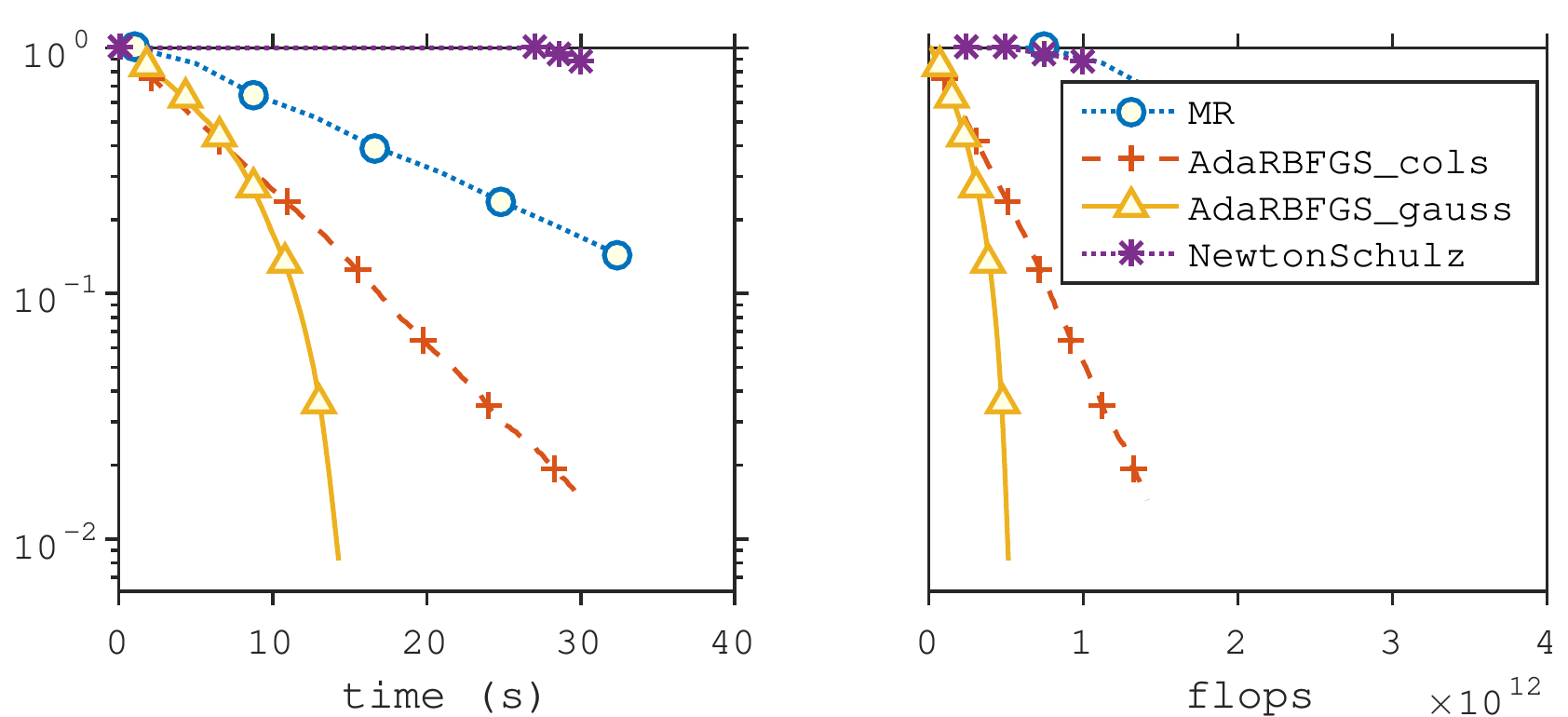}
        \caption{\texttt{rand} with $n=5 \cdot 10^4$ }\label{fig:rand3}
    \end{subfigure}%
    \caption{\footnotesize  Synthetic problems with a uniform random matrix $A~=~\bar{A}^T\bar{A}$, where $\bar{A}=$\texttt{rand}$(n).$ 
}\label{fig:pdsynth}
\end{figure}


\subsection{Experiment 2: LIBSVM matrices} 

Next we invert the Hessian matrix $\nabla^2 f(x)$ of four ridge-regression problems of the form
 \begin{equation}\label{eq:ridgeMatrix}
\min_{x\in \R^n}f(x)\eqdef \frac{1}{2}\norm{Ax-b}_2^2 + \frac{\lambda}{2} \norm{x}_2^2,\quad  \quad\nabla^2 f(x) = A^TA+\lambda I,
\end{equation}
using data from LIBSVM~\cite{Chang2011}, see Figure~\ref{fig:LIBSVM}. We use $\lambda =1$ as the regularization parameter. On the two problems of smaller dimension, \texttt{aloi} and \texttt{protein}, the four methods have a similar performance, and encounter the inverse in a few seconds. On the two larger problems, \texttt{gisette-scale} and \texttt{real-sim}, the two variants of AdaRBFGS significantly outperform the MR and the Newton-Schulz method.
\begin{figure}
    \centering
    \begin{subfigure}[t]{0.46\textwidth}
        \centering
\includegraphics[width =  \textwidth, trim= 50 310 60 310, clip ]{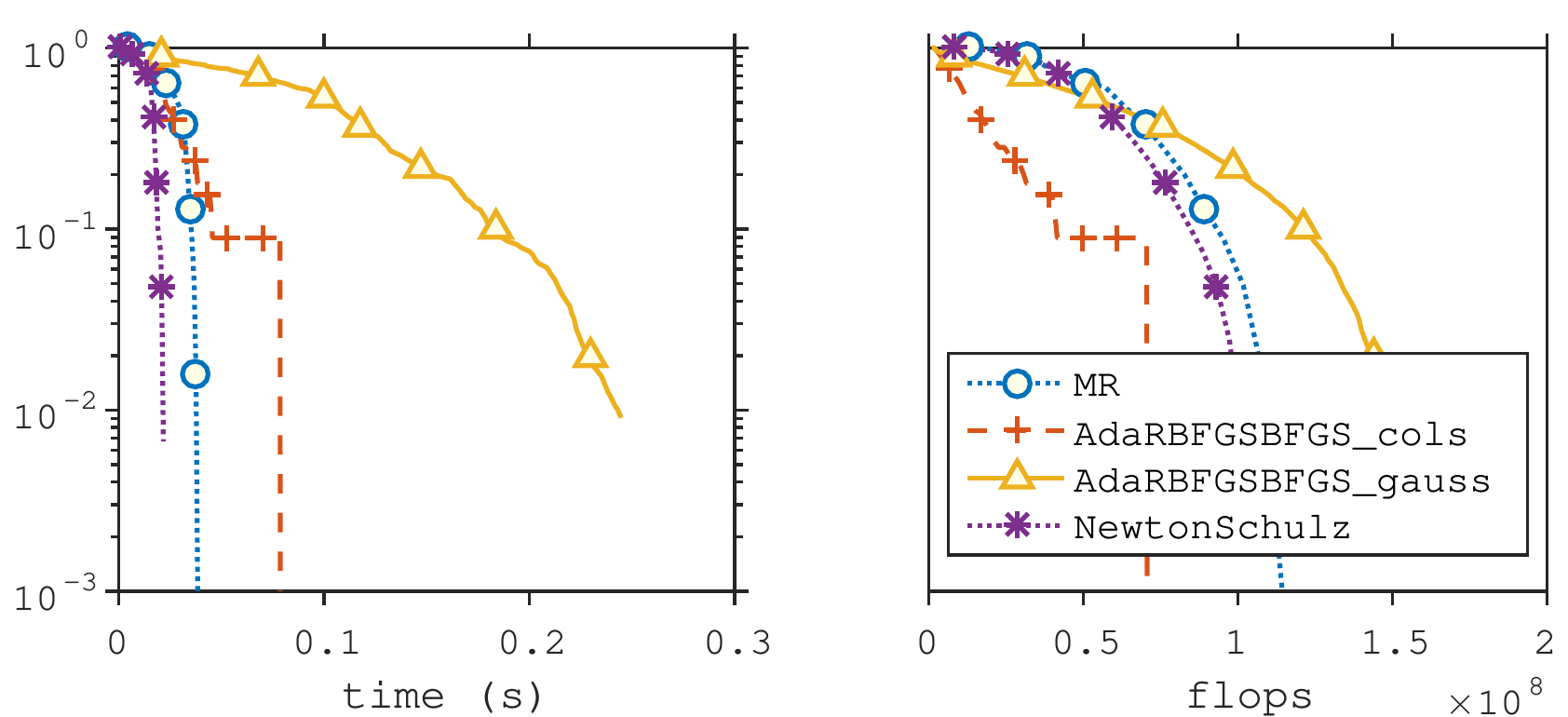}
        \caption{\texttt{aloi}}
    \end{subfigure}%
  \hspace{0.05\textwidth}
    \begin{subfigure}[t]{0.46\textwidth}
        \centering
\includegraphics[width =  \textwidth, trim= 50 310 60 310, clip ]{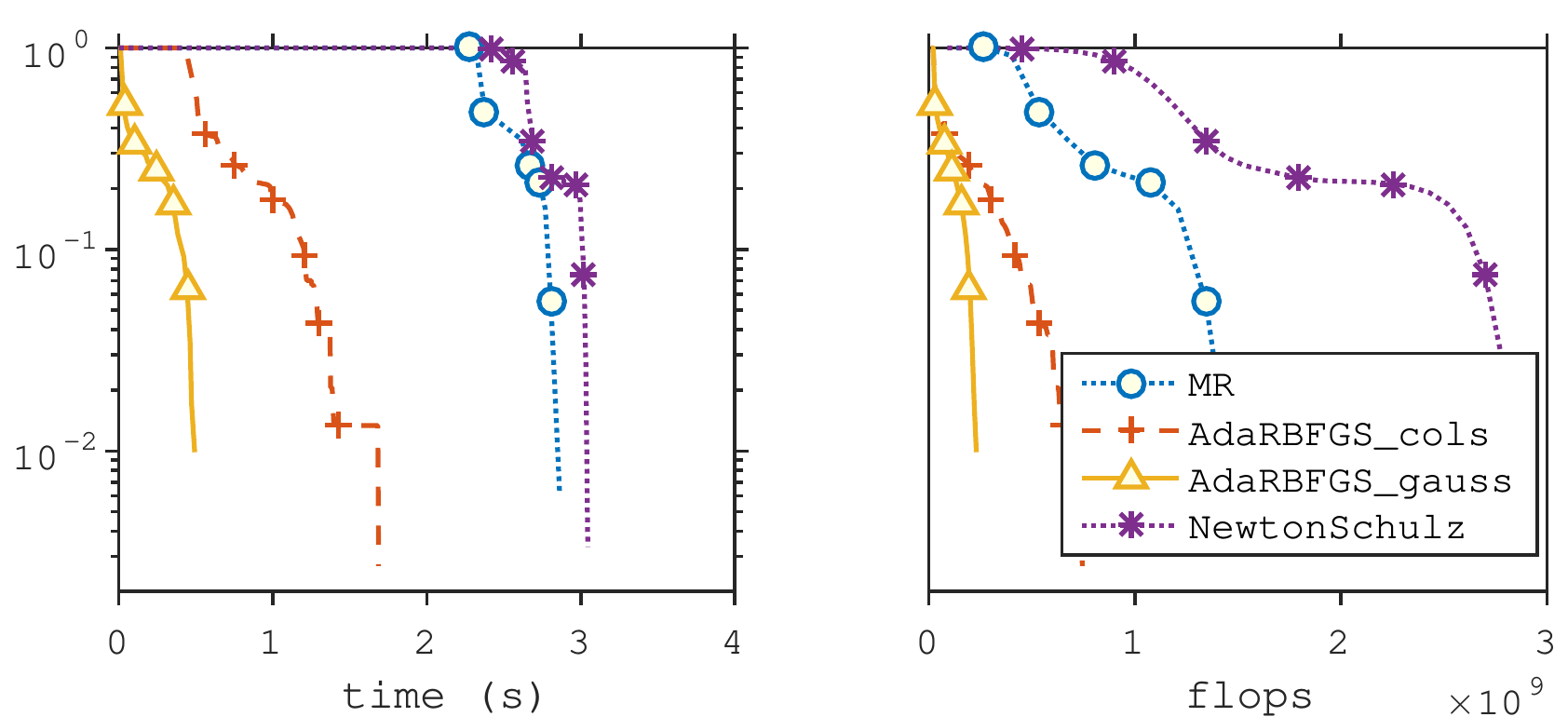}
        \caption{\texttt{protein}}
    \end{subfigure}
        \begin{subfigure}[t]{0.46\textwidth}
        \centering
\includegraphics[width =  \textwidth, trim=50 310 60 310, clip ]{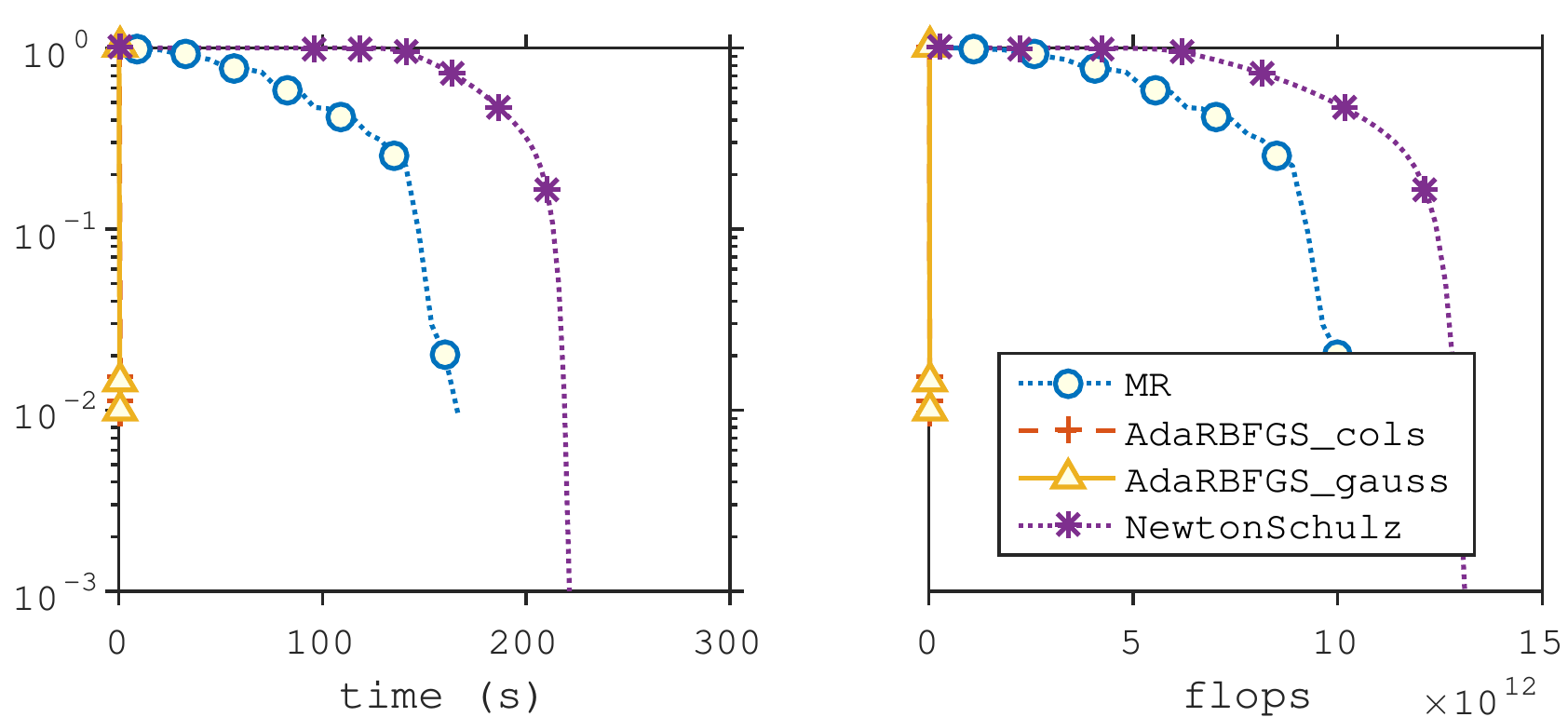}
        \caption{\texttt{gisette\_scale}}
    \end{subfigure}%
  \hspace{0.05\textwidth}
    \begin{subfigure}[t]{0.46\textwidth}
        \centering
\includegraphics[width =  \textwidth, trim= 50 310 60 310, clip ]{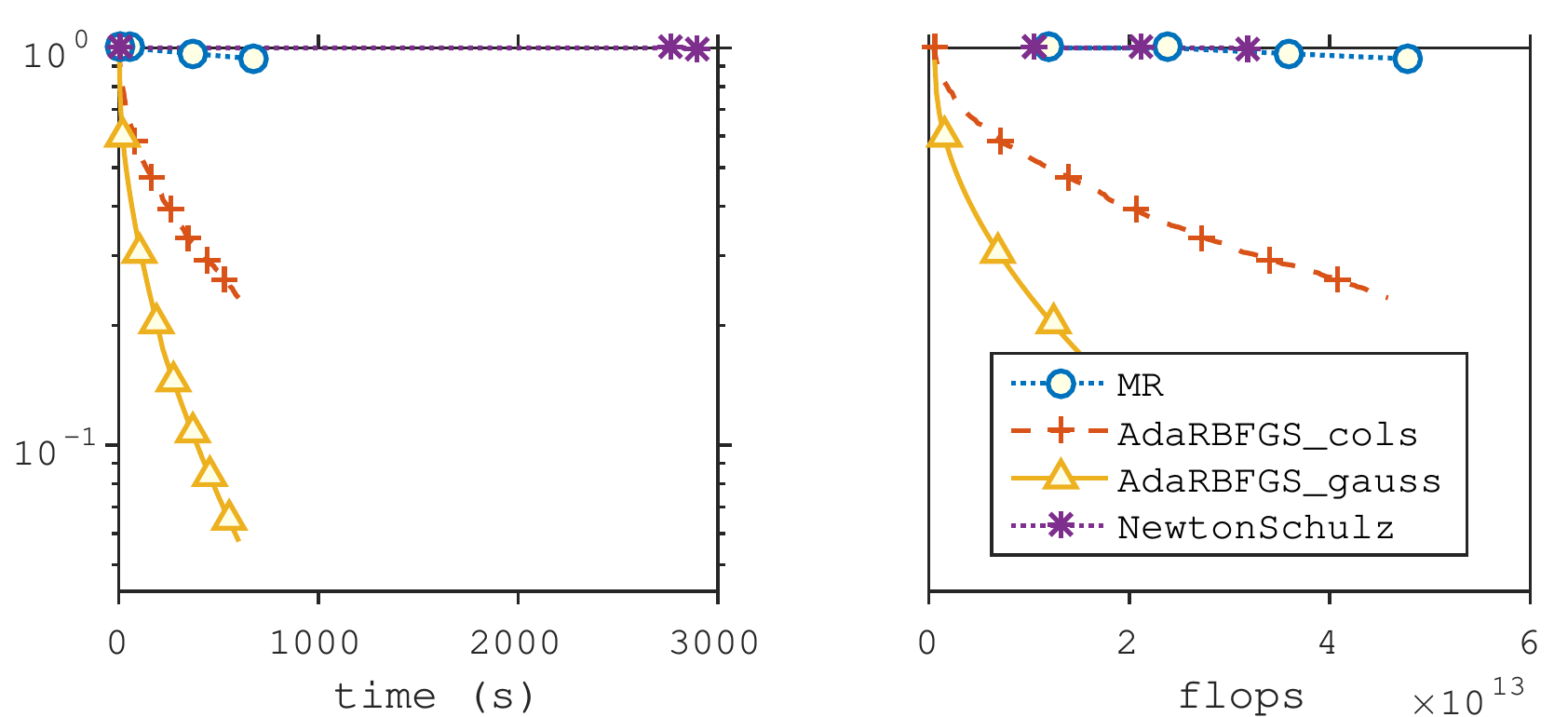}
        \caption{\texttt{real\_sim}}
    \end{subfigure}
    \caption{\footnotesize The performance of Newton-Schulz, MR, AdaRBFGS\_gauss and AdaRBFGS\_cols methods on the Hessian matrix of four LIBSVM test problems: (a) \texttt{aloi}: $(m;n)=(108,000;128)$ (b) \texttt{protein}: $(m; n)=(17,766; 357)$ (c)  \texttt{gisette\_scale}: $(m;n)=(6000; 5000)$ (d) \texttt{real-sim}: $(m;n)=(72,309; 20,958)$.} \label{fig:LIBSVM}
\end{figure}

\subsection{Experiment 3: UF sparse matrices}

For our final batch of tests, we invert several sparse matrices from the Florida sparse matrix collection~\cite{Davis:2011}. We have selected six problems from six different applications, so that the set of matrices display a varied sparsity pattern and  structure, see Figure~\ref{fig:UF}.

On the matrix \texttt{Bates/Chem97ZtZ} of moderate size, the four methods perform well, with the Newton-Schulz method converging first in time and AdaRBFGS\_cols first in flops. On the matrices of larger dimensions, the two variants of AdaRBFGS converge often orders of magnitude faster.  The significant difference between the performance of the methods on large scale problems can be, in part, explained by their iteration cost. The iterates of the Newton-Schulz and MR method compute $n\times n$ matrix-matrix products. While the cost of an iteration of the AdaRBFGS methods is dominated by the cost of a $n\times n$ matrix by $n\times q$ matrix product. As a result, and because we set $q = \sqrt{n},$ this is difference of $n^3$ to $n^{2+1/2}$ in iteration cost, which clearly shows on the larger dimensional instances. On the other hand, both the Newton-Schulz and MR method are quadratically locally convergent, thus when the iterates are close to the solution, these methods enjoy a notable speed-up.

\begin{figure}
    \centering
\begin{subfigure}[t]{0.46\textwidth}
        \centering   
\includegraphics[width =  \textwidth, trim=50 310 60 310, clip ]{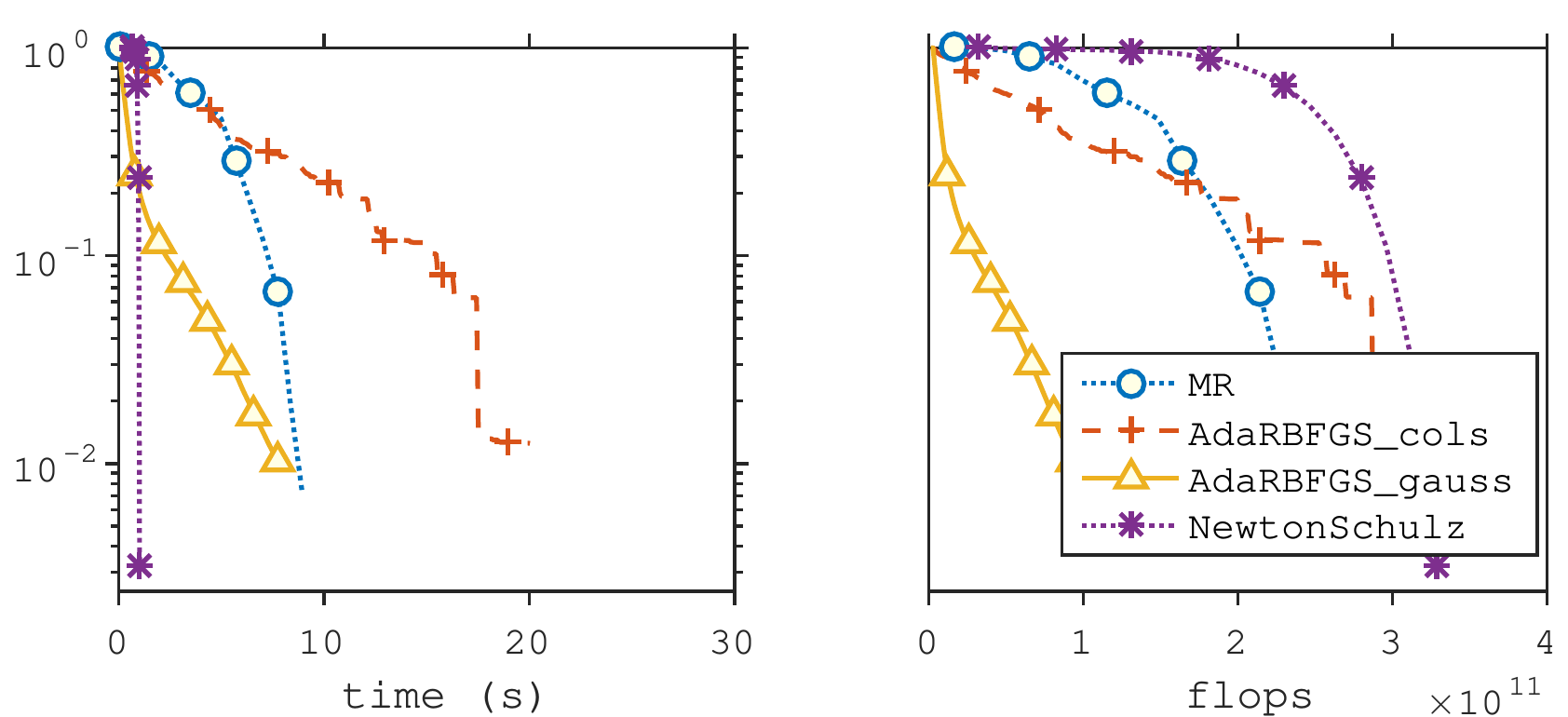}
        \caption{\texttt{Bates/Chem97ZtZ}}
\end{subfigure} \hspace{0.01\textwidth} 
\begin{subfigure}[t]{0.46\textwidth}
        \centering
\includegraphics[width =  \textwidth, trim= 50 310 60 310, clip ]{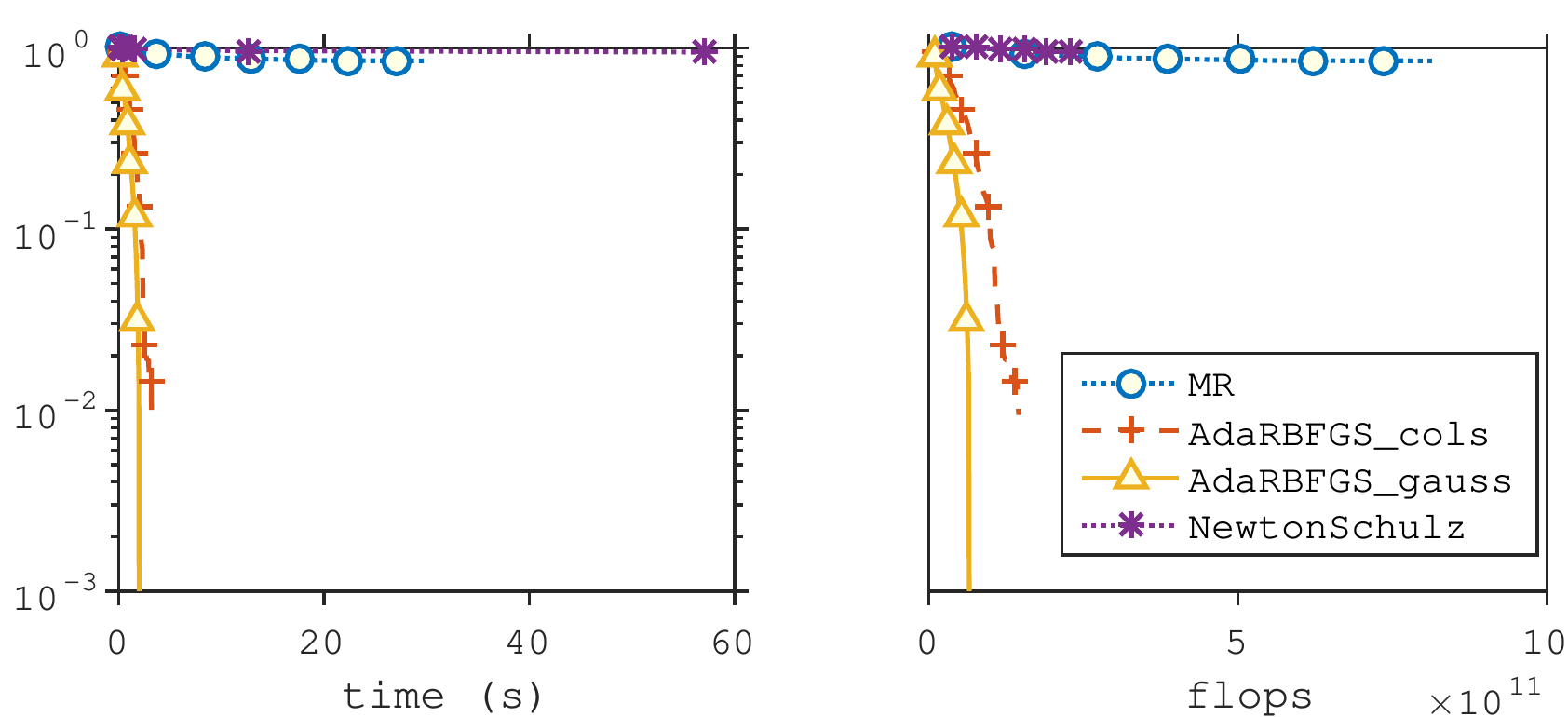}
        \caption{\texttt{FIDAP/ex9}}
\end{subfigure} \hspace{0.01\textwidth}  %
\begin{subfigure}[t]{0.46\textwidth}
        \centering
\includegraphics[width =  \textwidth, trim=50 310 60 310, clip ]{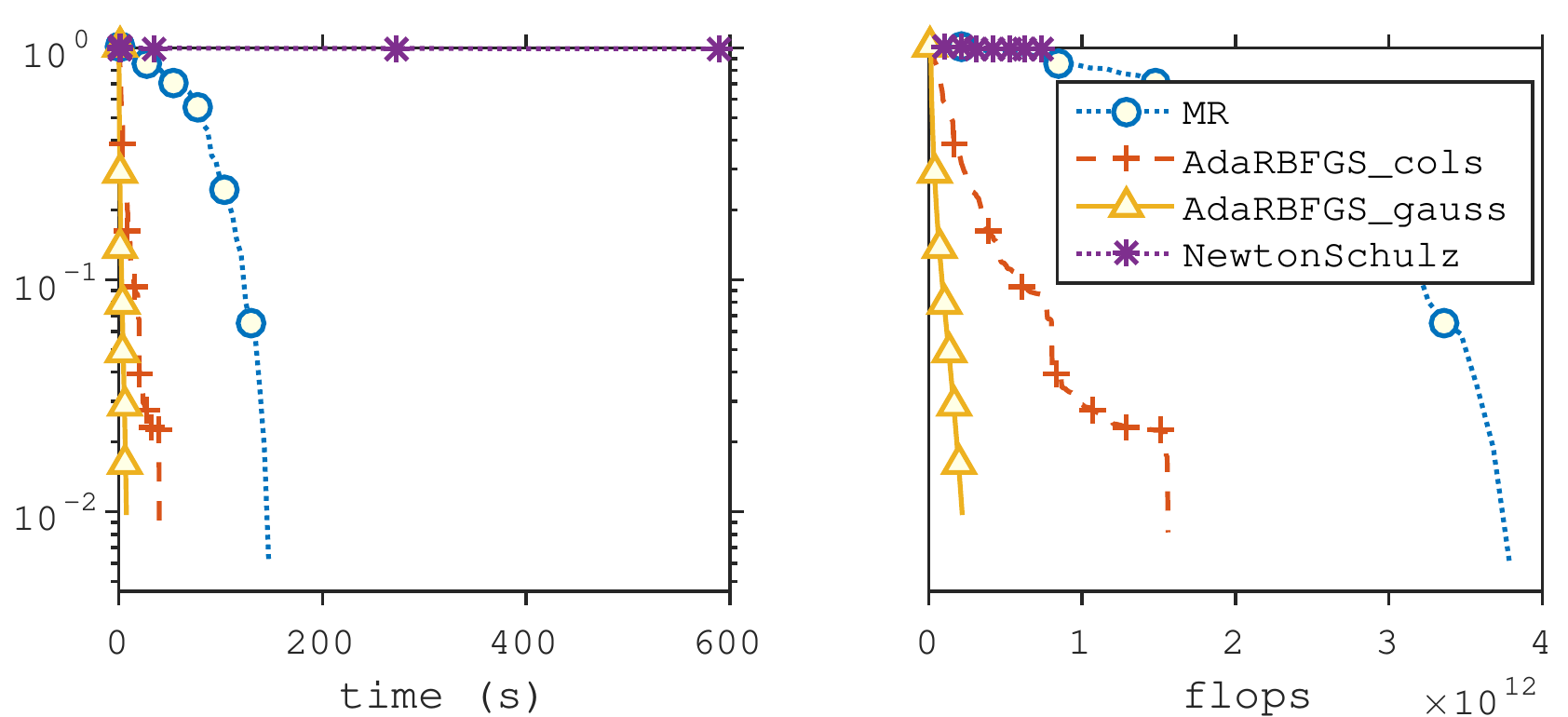}
        \caption{\texttt{Nasa/nasa4704}}
\end{subfigure}  \hspace{0.01\textwidth} 
\begin{subfigure}[t]{0.46\textwidth}
        \centering
\includegraphics[width =  \textwidth, trim=50 310 60 310, clip ]{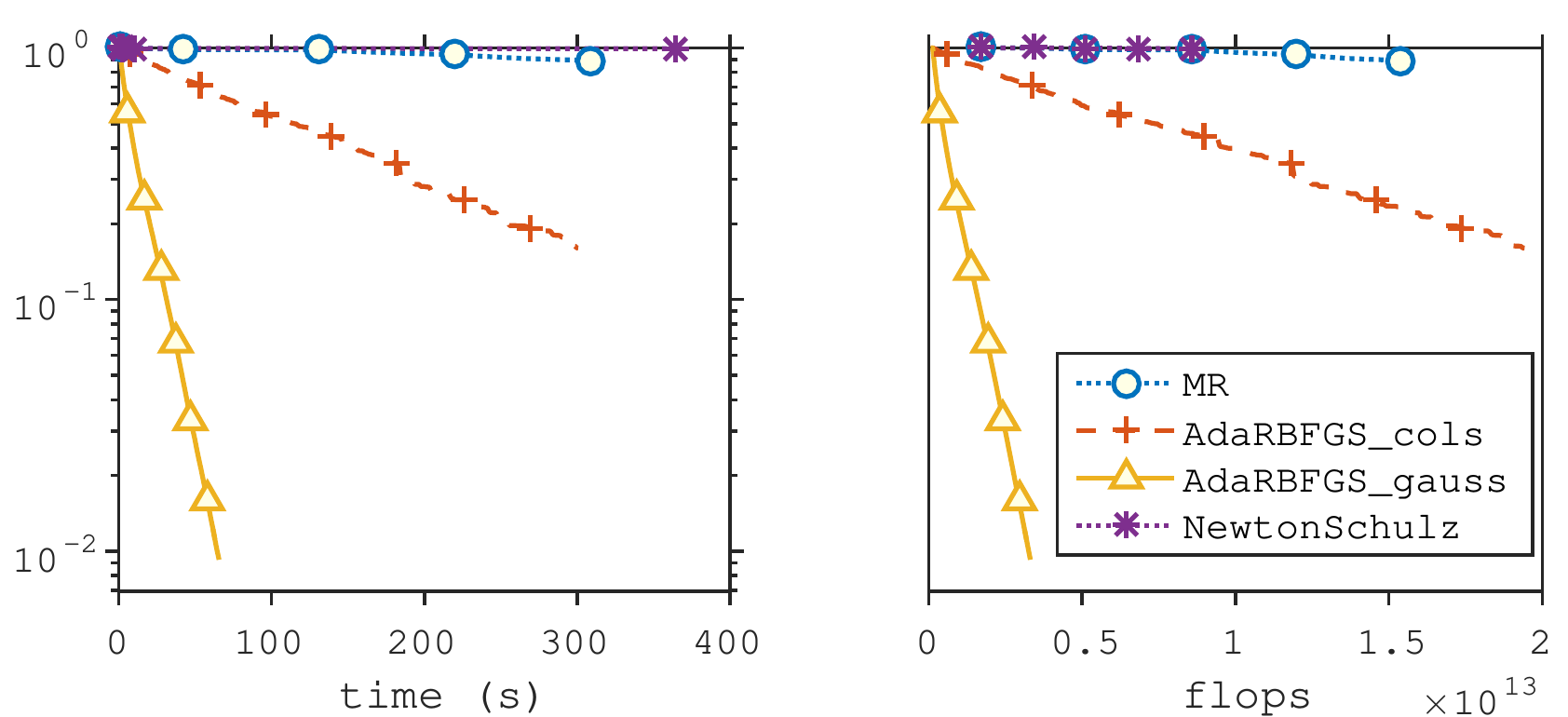}
        \caption{\texttt{HB/bcsstk18}}
\end{subfigure}\hspace{0.01\textwidth}  
\begin{subfigure}[t]{0.46\textwidth}
        \centering
\includegraphics[width =  \textwidth, trim=50 310 60 310, clip ]{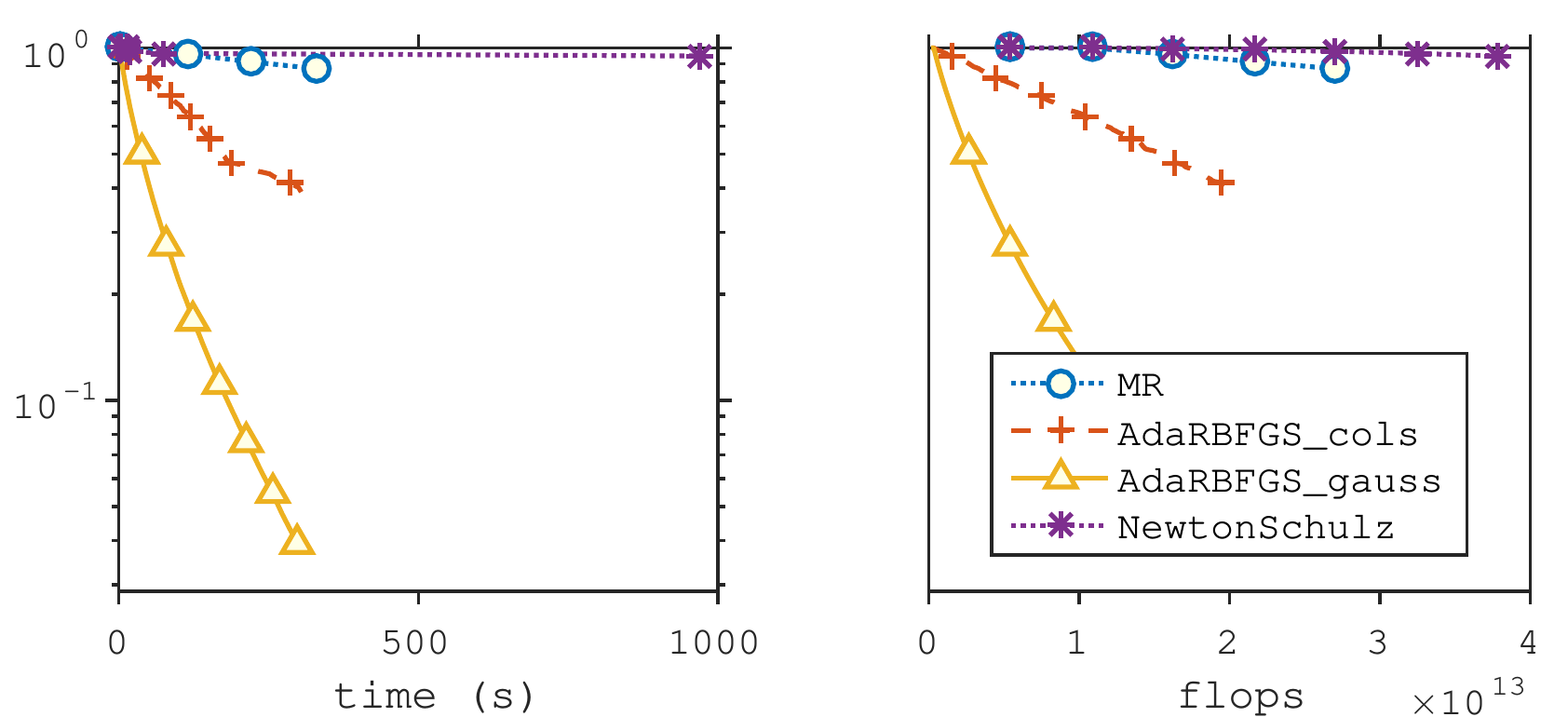}
        \caption{\texttt{Pothen/bodyy4}}
\end{subfigure}  \hspace{0.01\textwidth} 
\begin{subfigure}[t]{0.46\textwidth}
        \centering
\includegraphics[width =  \textwidth, trim=50 310 60 310, clip ]{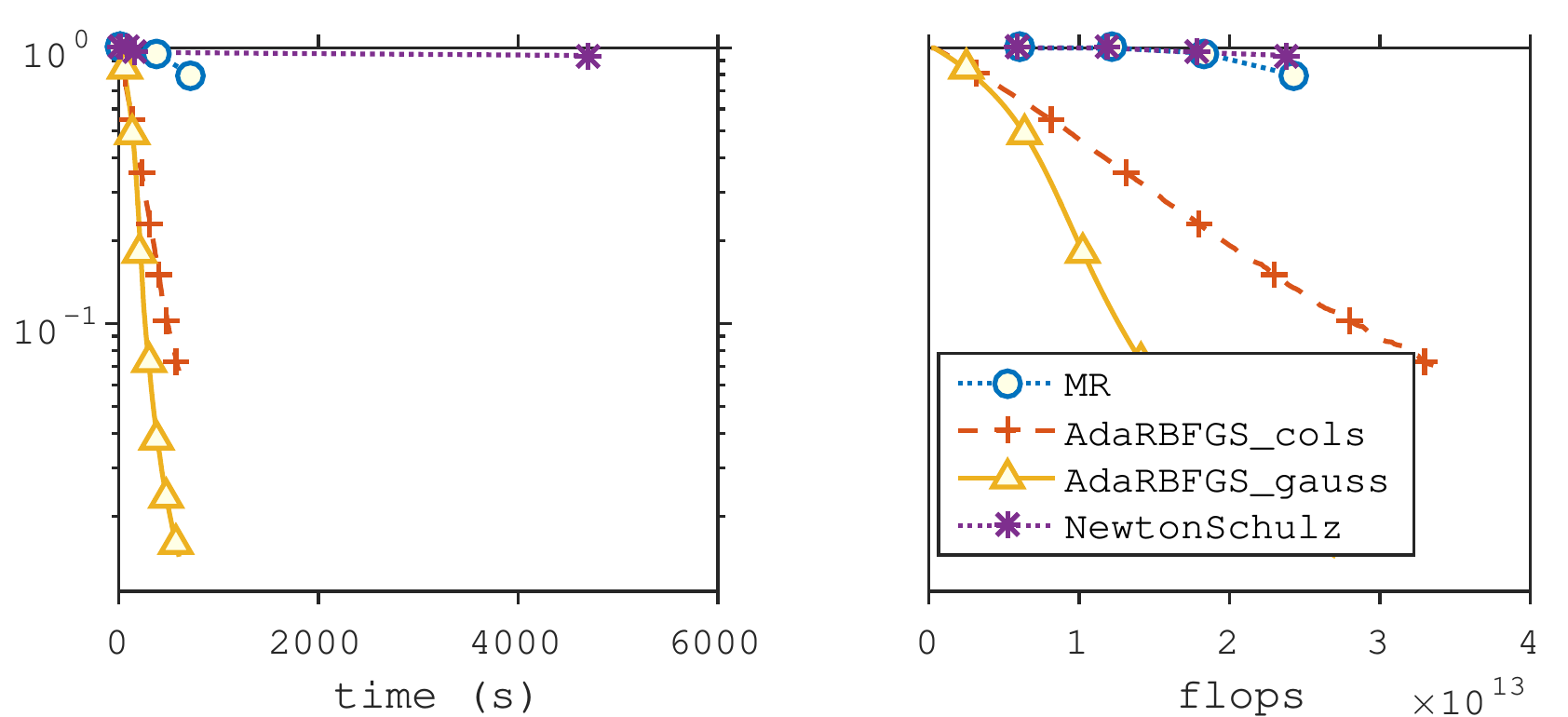}
        \caption{\texttt{ND/nd6k}}
\end{subfigure}\hspace{0.01\textwidth} 
\begin{subfigure}[t]{0.46\textwidth}
        \centering
\includegraphics[width =  \textwidth, trim=50 310 60 310, clip ]{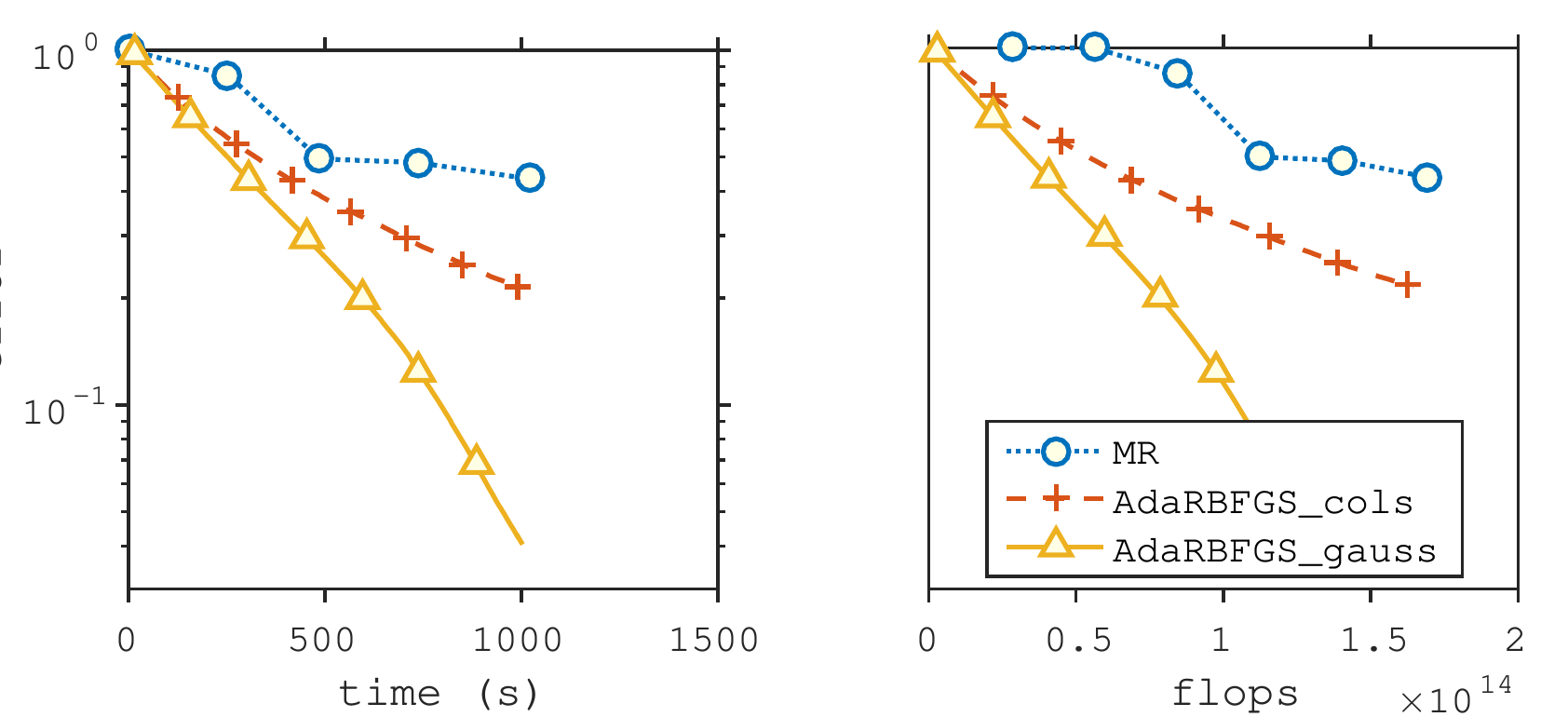}
        \caption{\texttt{GHS\_psdef/wathen100}}
\end{subfigure}%
  \caption{\footnotesize  The performance of Newton-Schulz, MR, AdaRBFGS\_gauss and AdaRBFGS\_cols on 
(a) \texttt{Bates-Chem97ZtZ}: $n= 2,541$,   
  (b) \texttt{FIDAP/ex9}: $n = 3,363 $, (c) \texttt{Nasa/nasa4704}: $n= 4,704$, 
  (d) \texttt{HB/bcsstk18}: $n=11,948$,   (e) \texttt{Pothen/bodyy4}: $n =17,546$ (f) \texttt{ND/nd6k}: $n=18,000$ (g) \texttt{GHS\_psdef/wathen100}: $n = 30, 401$.  } \label{fig:UF}
  \end{figure}

\section{Conclusion}

%
%
%
%
%

We developed a family of stochastic methods for iteratively inverting matrices, with  specialized variants for nonsymmetric, symmetric and positive definite matrices. The methods have  dual viewpoints: a sketch-and-project viewpoint (which is an extension of the least-change formulation of the qN methods),   and a constrain-and-approximate viewpoint (which is related to the approximate inverse preconditioning (API) methods). The equivalence between these  viewpoints reveals a deep connection between the qN and the API methods, which were previously considered to be unrelated. 
   
     Under mild conditions, we establish  convergence of the expected norm of the error, and the norm of the expected error. Our convergence theorems are general enough to accommodate discrete samplings and continuous  samplings, though we only explore discrete samplings here in more detail.  For discrete samplings, we determine a probability distribution for which the convergence rates are equal to a scaled condition number, and thus are easily interpretable. Furthermore, for discrete sampling, we determining a practical optimized sampling distribution, that is obtained by minimizing an upper bound on the convergence rate. We develop new randomized block variants of the qN updates, including the BFGS update, complete with convergence rates, and provide new insights into these methods using our dual viewpoint. 
           
            For positive definite matrices, we develop an Adaptive Randomized BFGS method (AdaRBFGS), which in large-scale numerical experiments can be orders of magnitude faster (in time and flops) than the self-conditioned minimal residual method and the Newton-Schulz method. In particular, only the AdaRBFGS methods are able to approximately invert  
the $20,958 \times 20,958$ ridge regression matrix based on the \texttt{real-sim} data set in reasonable time and flops. 

   This paper opens up many avenues for future work, such as developing methods that use continuous random sampling and implementing a limited memory approach akin to the LBFGS~\cite{Nocedal1980} method, which could maintain an operator that serves as an approximation to the inverse. The same method can also be used to calculating the pseudo-inverse.  As recently shown in~\cite{Gower2015c}, an analogous method can be applied to linear systems, converging with virtually no assumptions on the system matrix. 



{ 
\printbibliography
}

\section{Appendix: Optimizing the Convergence Rate}
  
\begin{lemma}\label{lem:fracsum} Let $a_1,\dots,a_r$ be positive real numbers. Then
\[\left[ \frac{\sqrt{a_1}}{\sum_{i=1}^r \sqrt{a_i}},\ldots, \frac{\sqrt{a_n}}{\sum_{i=1}^r \sqrt{a_i}} \right] =\arg\min_{p\in \Delta_r} \sum_{i=1}^r \frac{a_i}{p_i}.\]
\end{lemma}
\begin{proof}
Incorporating the constraint $\sum_{i=1}^r p_i =1$ into the Lagrangian we have 
\[\min_{p\geq 0} \sum_{i=1}^r \frac{a_i}{p_i} + \mu\sum_{i=1}^r (p_i-1),\]
where $\mu \in \R.$ Differentiating in $p_i$ and setting to zero, then isolating $p_i$ gives
\begin{equation} \label{eq:muaipi}
 p_i = \sqrt{\frac{a_i}{\mu}}, \quad \mbox{for }i=1,\ldots r.\end{equation}
Summing over $i$ gives
\[ 1 = \sum_{i=1}^r \sqrt{\frac{a_i}{\mu}} \quad \Rightarrow \quad \mu = \left(\sum_{i=1}^r \sqrt{a_i}\right)^2.\]
Inserting this back into~\eqref{eq:muaipi} gives $p_i = \left.\sqrt{a_i}\right/\sum_{i=1}^r \sqrt{a_i}.$
\end{proof}

\section{Appendix: Numerical Experiments with the Same Starting Matrix}\label{app:numerics}
We now investigate the empirical convergence of the methods MR, AdaRBFGS\_cols and AdaRGFBS\_gauss when initiated with the same starting matrix $X_0 = I,$ see Figures~\ref{fig:LIBSVM2} and~\ref{fig:UF2}. We did not include the Newton-Schultz method in these figures because it diverged on all experiments when initiated from $X_0 =I.$
Again we observe that, as the dimension grows, only the two variants of the AdaRBFGS are capable of inverting the matrix to the desired $10^{-2}$ precision in a reasonable amount of time. Furthermore, the AdaRBFGS\_gauss variant had the overall best best performance.

\begin{figure}
    \centering
    \begin{subfigure}[t]{0.46\textwidth}
        \centering
\includegraphics[width =  \textwidth, trim= 30 270 40 280, clip ]{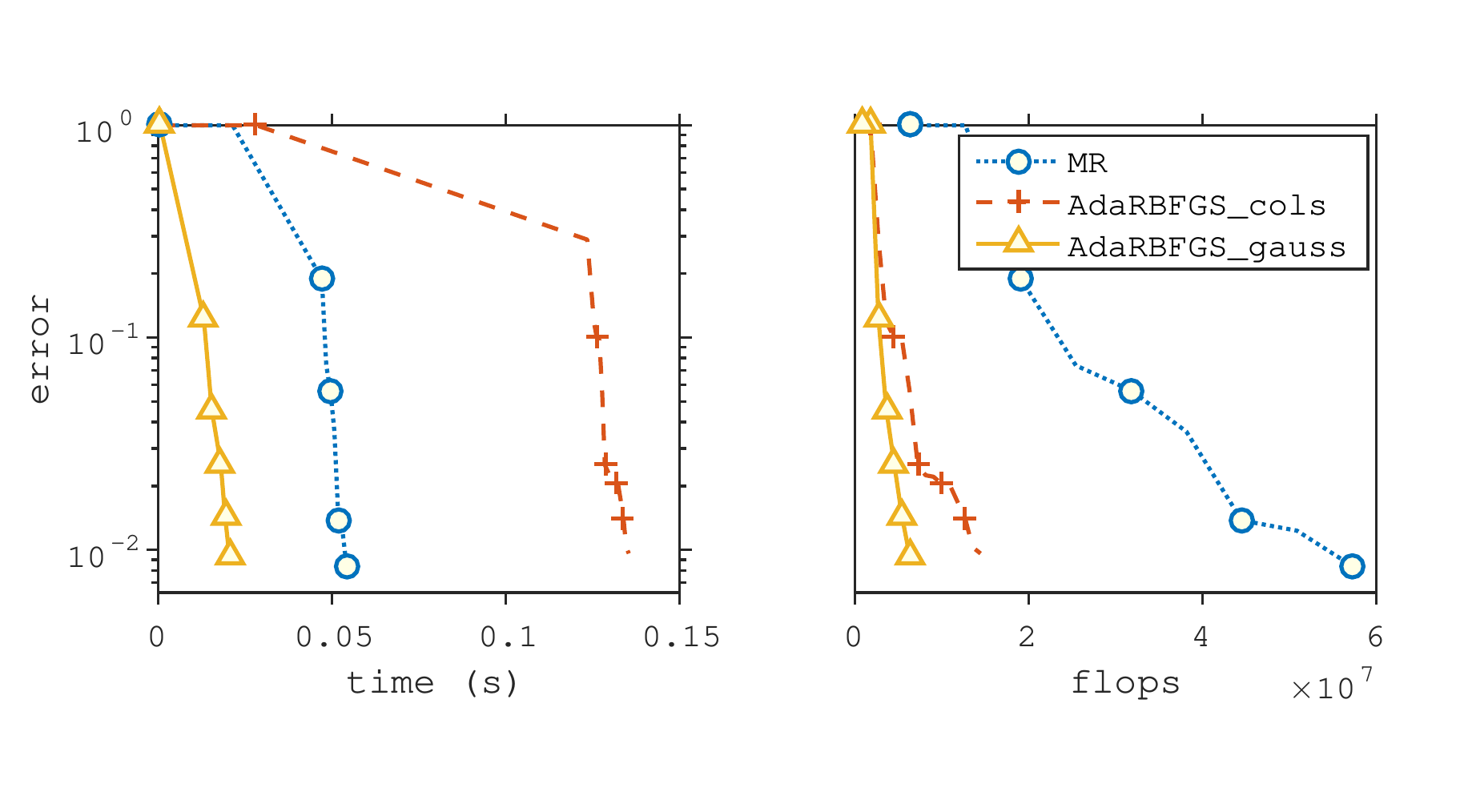}
        \caption{\texttt{aloi}}
    \end{subfigure}%
  \hspace{0.05\textwidth}
    \begin{subfigure}[t]{0.46\textwidth}
        \centering
\includegraphics[width =  \textwidth, trim=30 270 40 280, clip ]{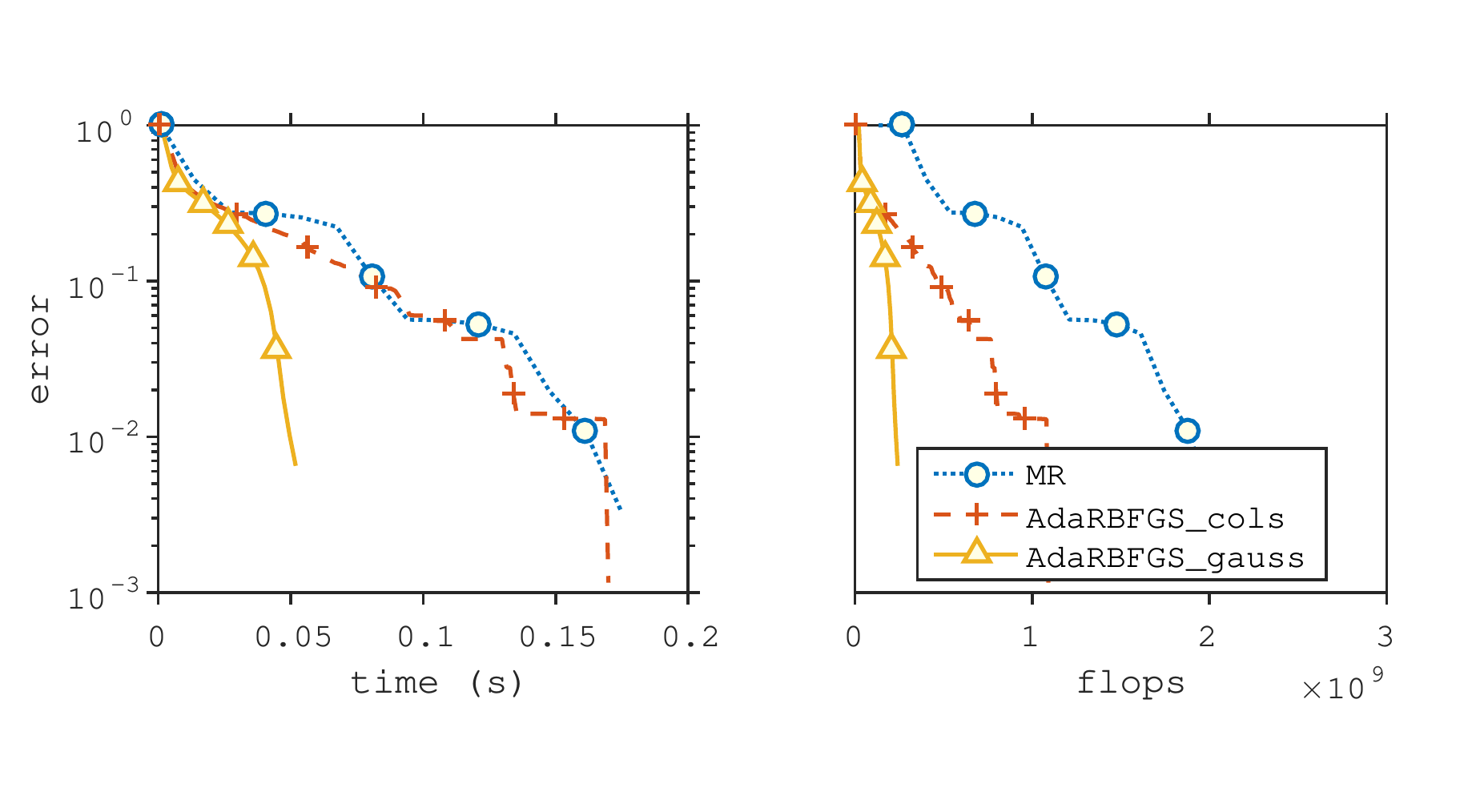}
        \caption{\texttt{protein}}
    \end{subfigure}
        \begin{subfigure}[t]{0.46\textwidth}
        \centering
\includegraphics[width =  \textwidth, trim=30 270 40 280, clip ]{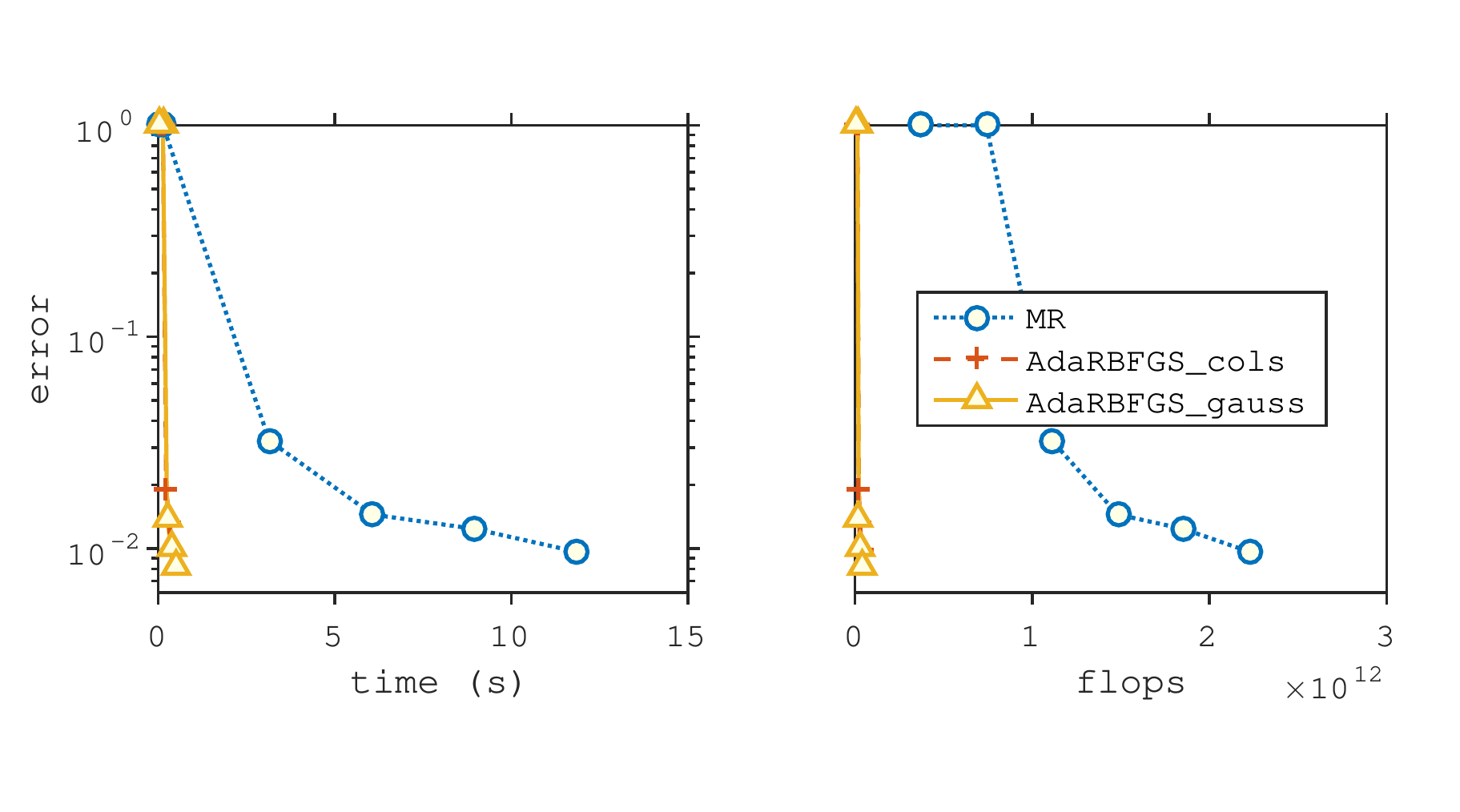}
        \caption{\texttt{gisette\_scale}}
    \end{subfigure}%
  \hspace{0.05\textwidth}
    \begin{subfigure}[t]{0.46\textwidth}
        \centering
\includegraphics[width =  \textwidth, trim= 30 270 40 280, clip ]{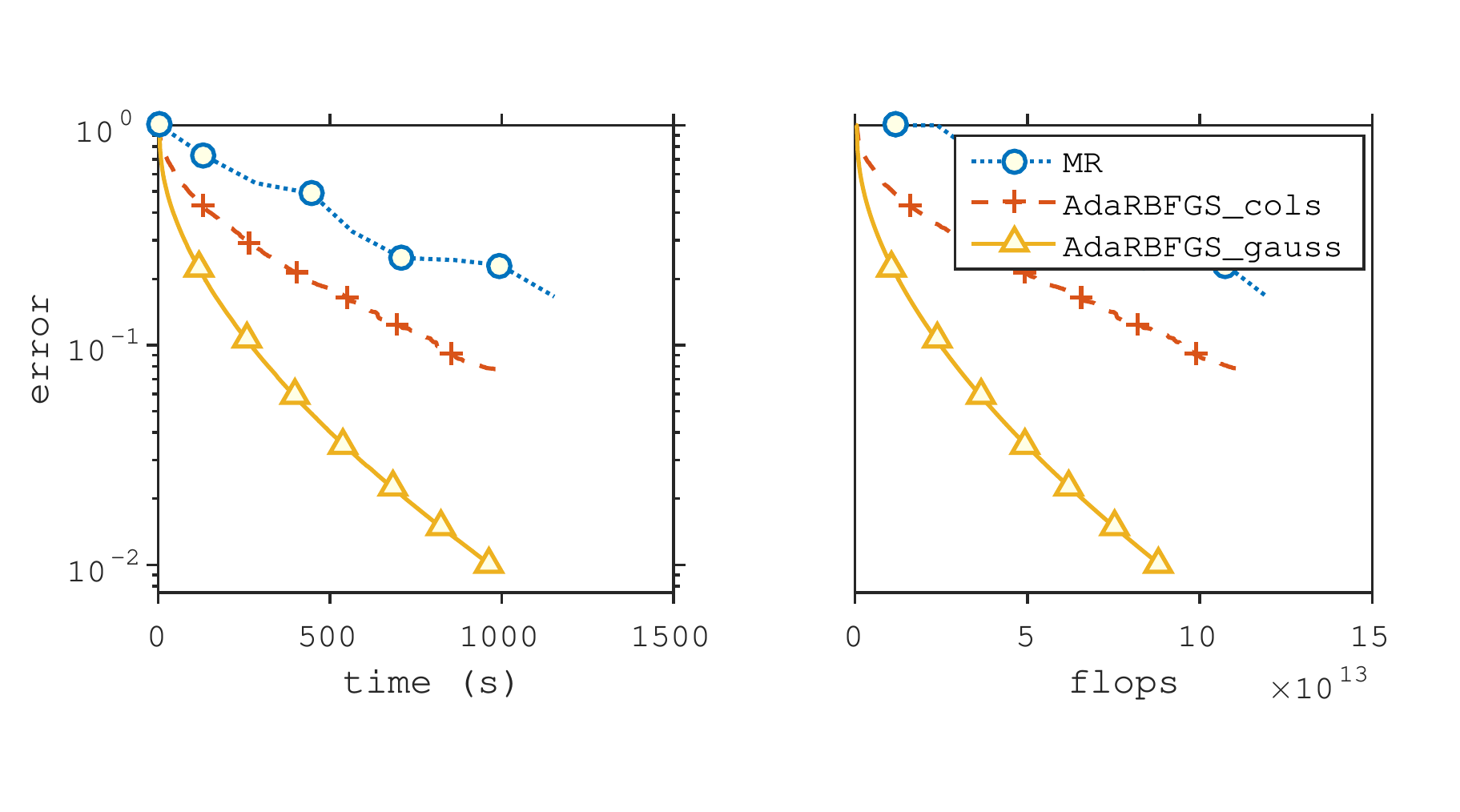}
        \caption{\texttt{real\_sim}}
    \end{subfigure}
    \caption{The performance of Newton-Schulz, MR, AdaRBFGS\_gauss and AdaRBFGS\_cols methods on the Hessian matrix of four LIBSVM test problems: (a) \texttt{aloi}: $(m;n)=(108,000;128)$ (b) \texttt{protein}: $(m; n)=(17,766; 357)$ (c)  \texttt{gisette\_scale}: $(m;n)=(6000; 5000)$ (d) \texttt{real-sim}: $(m;n)=(72,309; 20,958)$. The starting matrix $X_0=I$ was used for all methods.} \label{fig:LIBSVM2}
\end{figure}

\begin{figure}
    \centering
\begin{subfigure}[t]{0.5\textwidth}
        \centering
\includegraphics[width =  \textwidth, trim=0 270 0 260, clip ]{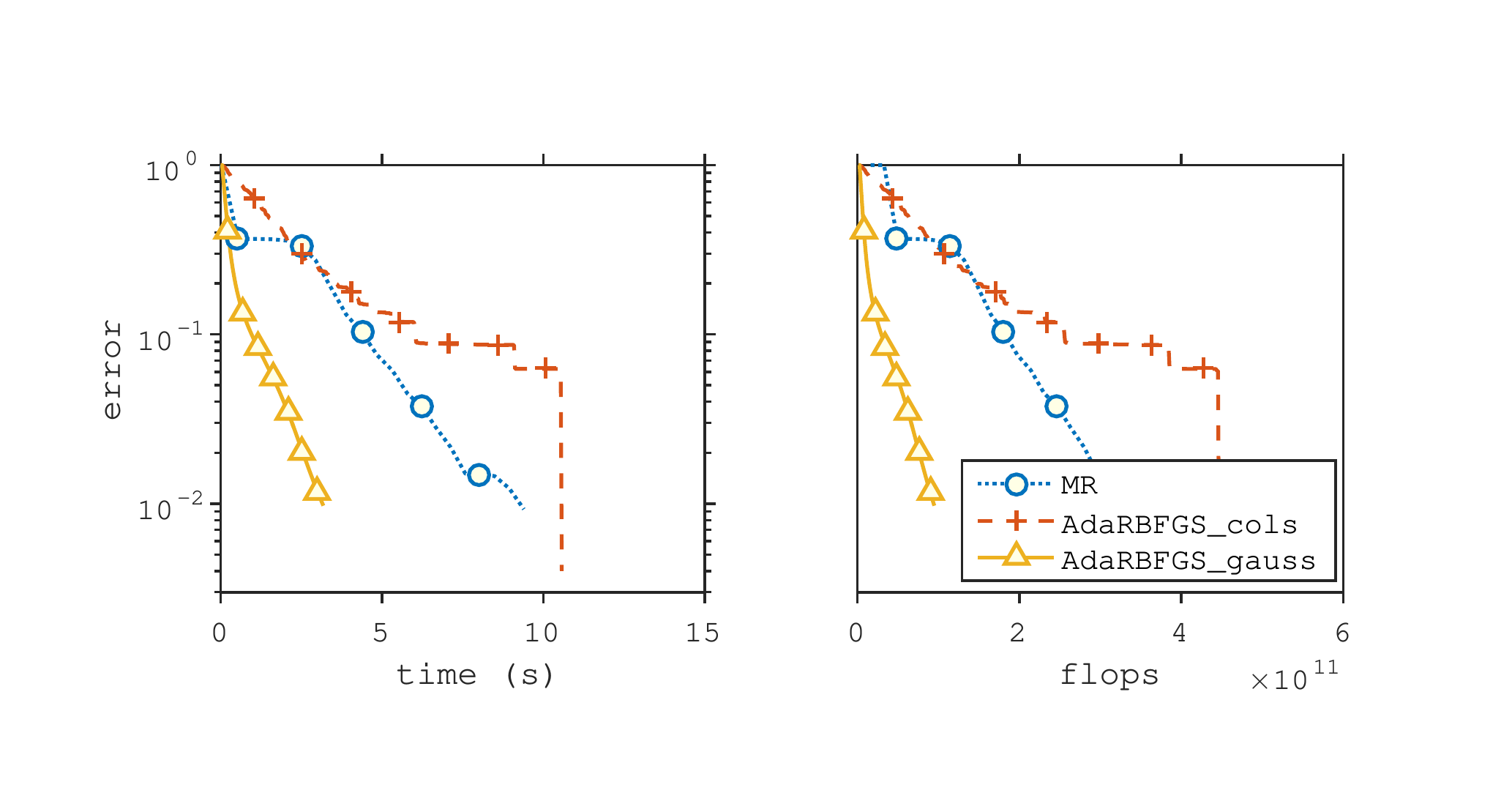}
        \caption{\texttt{Bates/Chem97ZtZ}}
\end{subfigure} 
\begin{subfigure}[t]{0.5\textwidth}
        \centering
\includegraphics[width =  \textwidth, trim= 0 270 0 260, clip ]{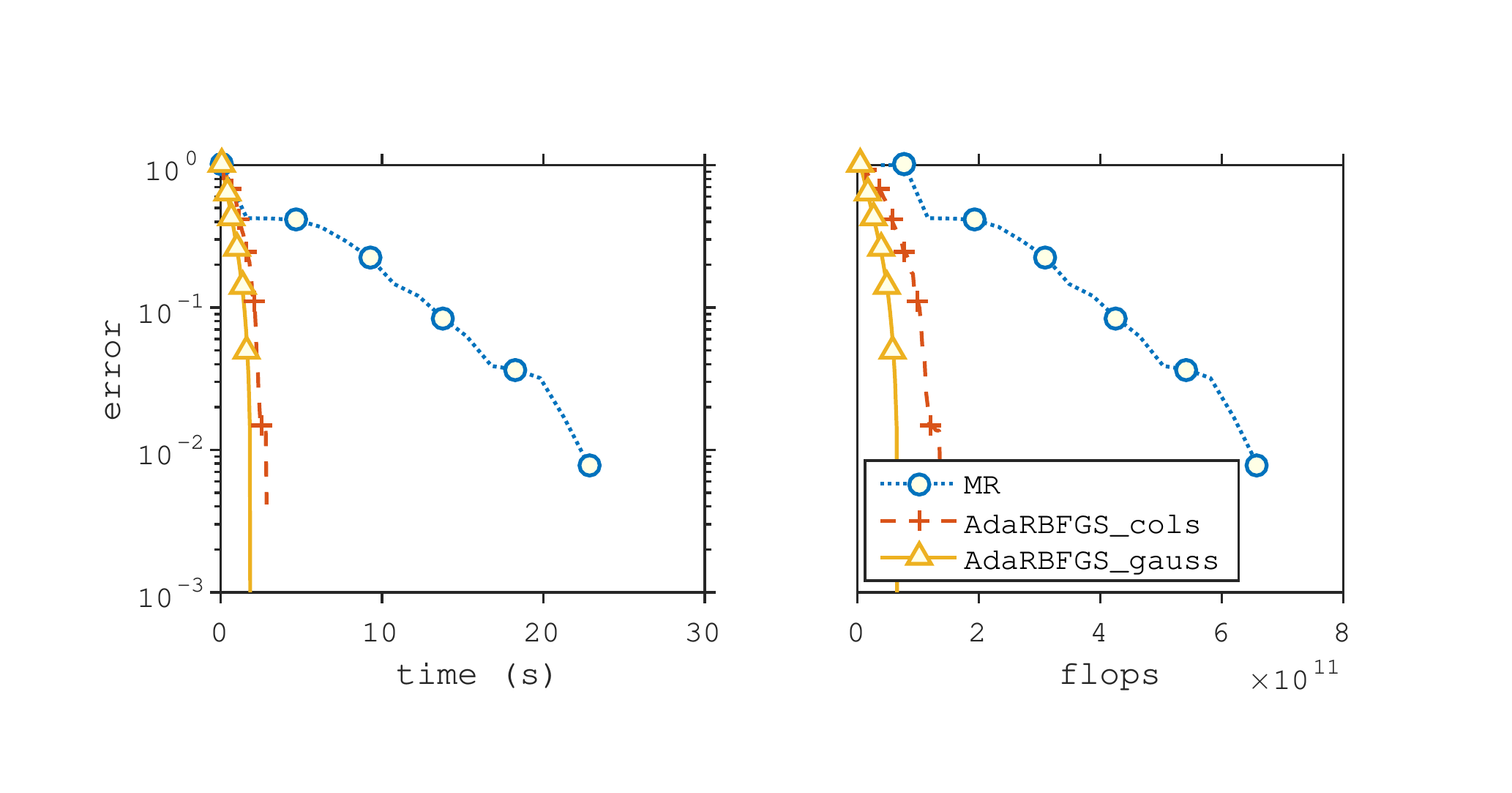}
        \caption{\texttt{FIDAP/ex9}}
\end{subfigure}%
\begin{subfigure}[t]{0.5\textwidth}
        \centering
\includegraphics[width =  \textwidth, trim=0 270 0 260, clip ]{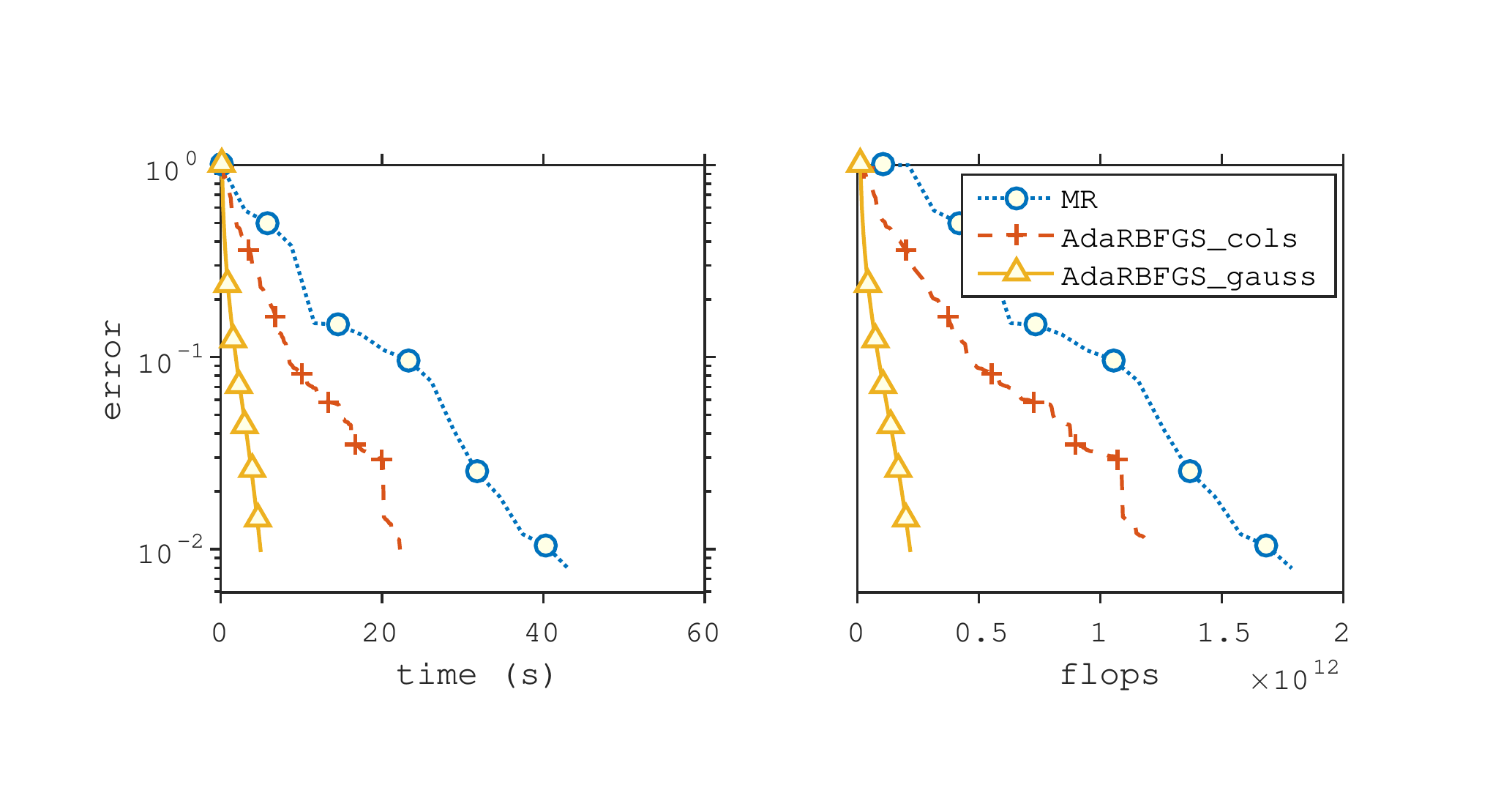}
        \caption{\texttt{Nasa/nasa4704}}
\end{subfigure} 
\begin{subfigure}[t]{0.5\textwidth}
        \centering
\includegraphics[width =  \textwidth, trim=0 270 0 260, clip ]{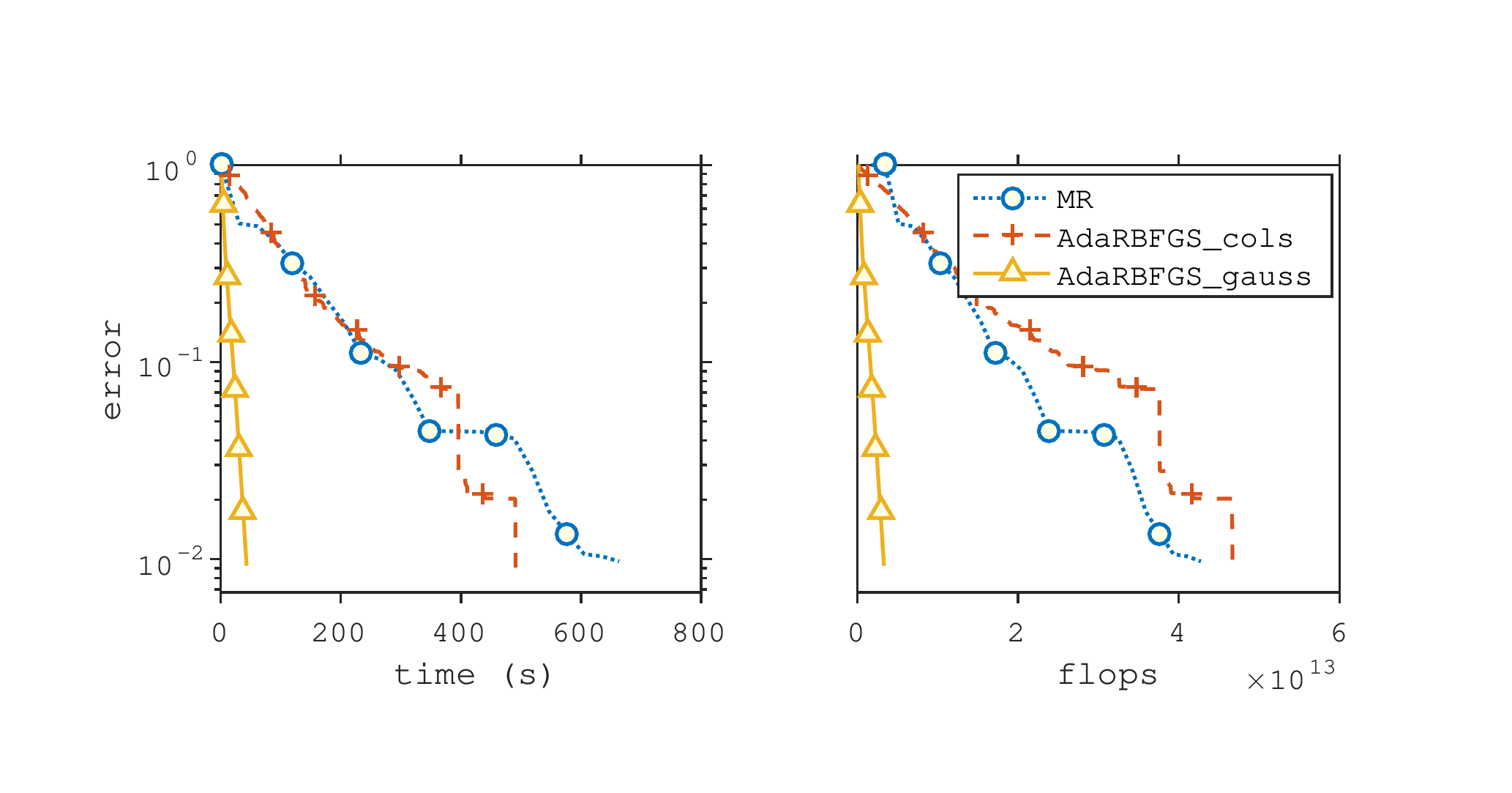}
        \caption{\texttt{HB/bcsstk18}}
\end{subfigure}%
\begin{subfigure}[t]{0.5\textwidth}
        \centering
\includegraphics[width =  \textwidth, trim=0 270 0 260, clip ]{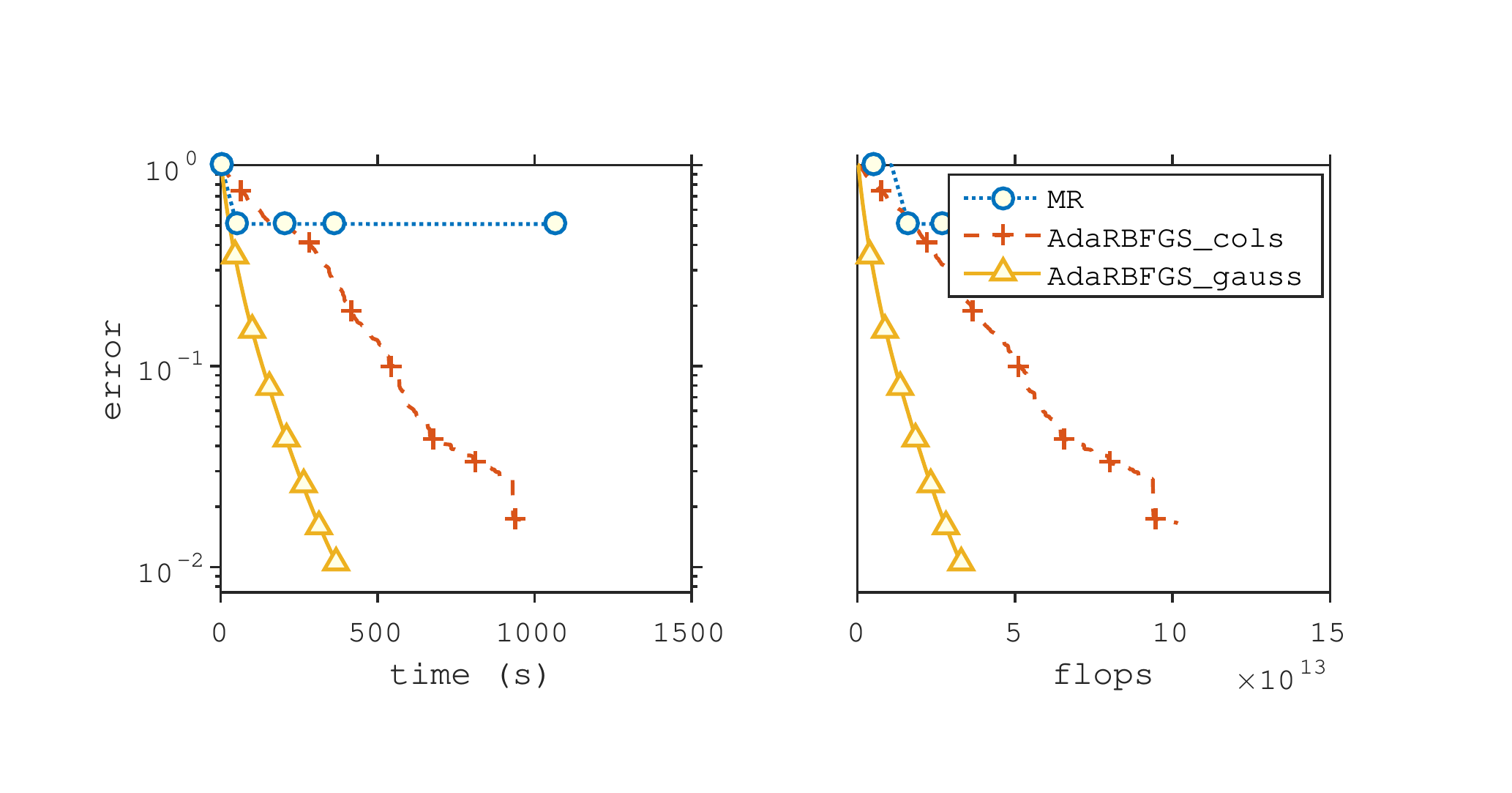}
        \caption{\texttt{Pothen/bodyy4}}
\end{subfigure} 
\begin{subfigure}[t]{0.5\textwidth}
        \centering
\includegraphics[width =  \textwidth, trim=0 270 0 260, clip ]{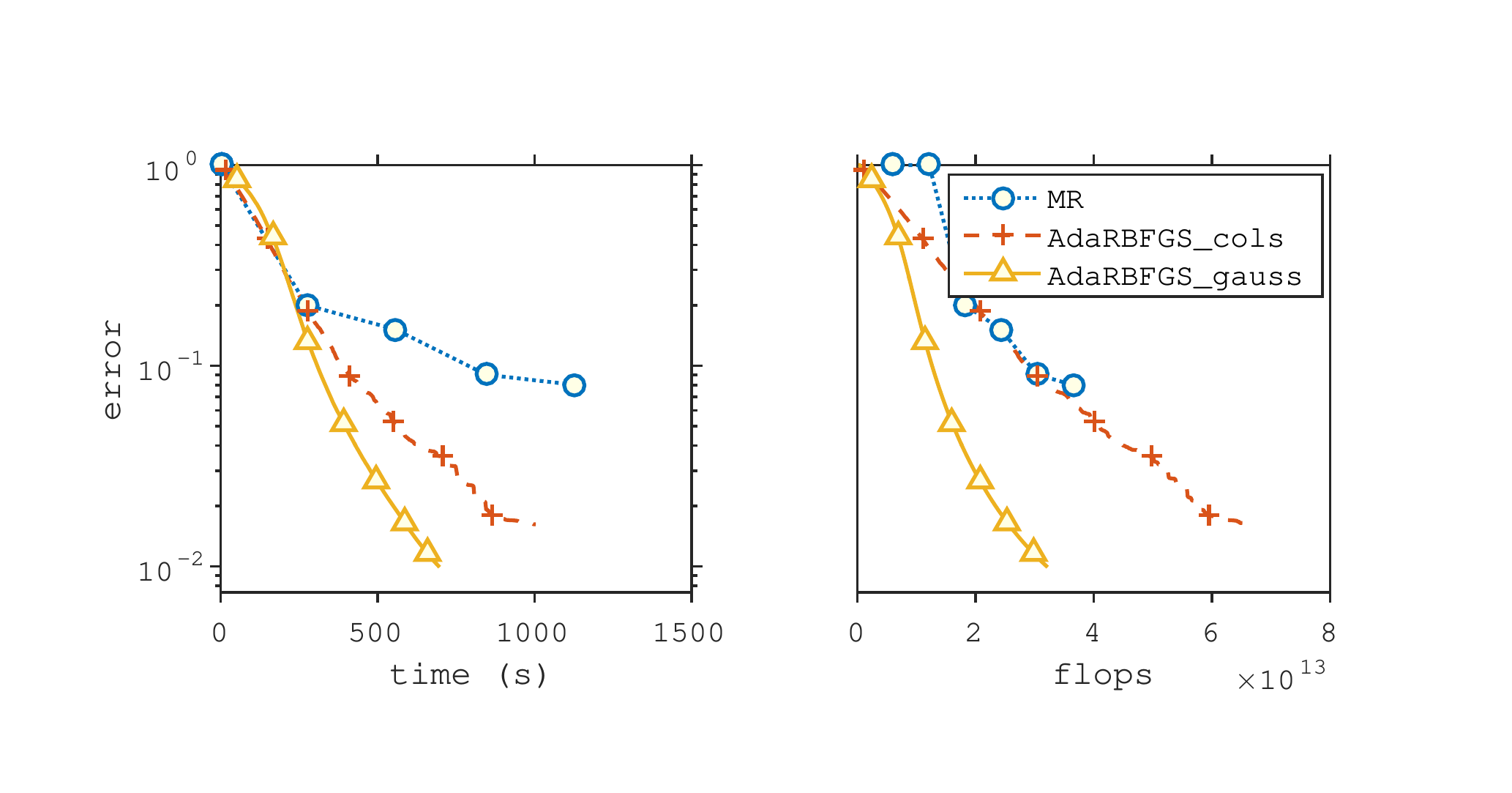}
        \caption{\texttt{ND/nd6k}}
\end{subfigure}%
\begin{subfigure}[t]{0.5\textwidth}
        \centering
\includegraphics[width =  \textwidth, trim=0 270 0 260, clip ]{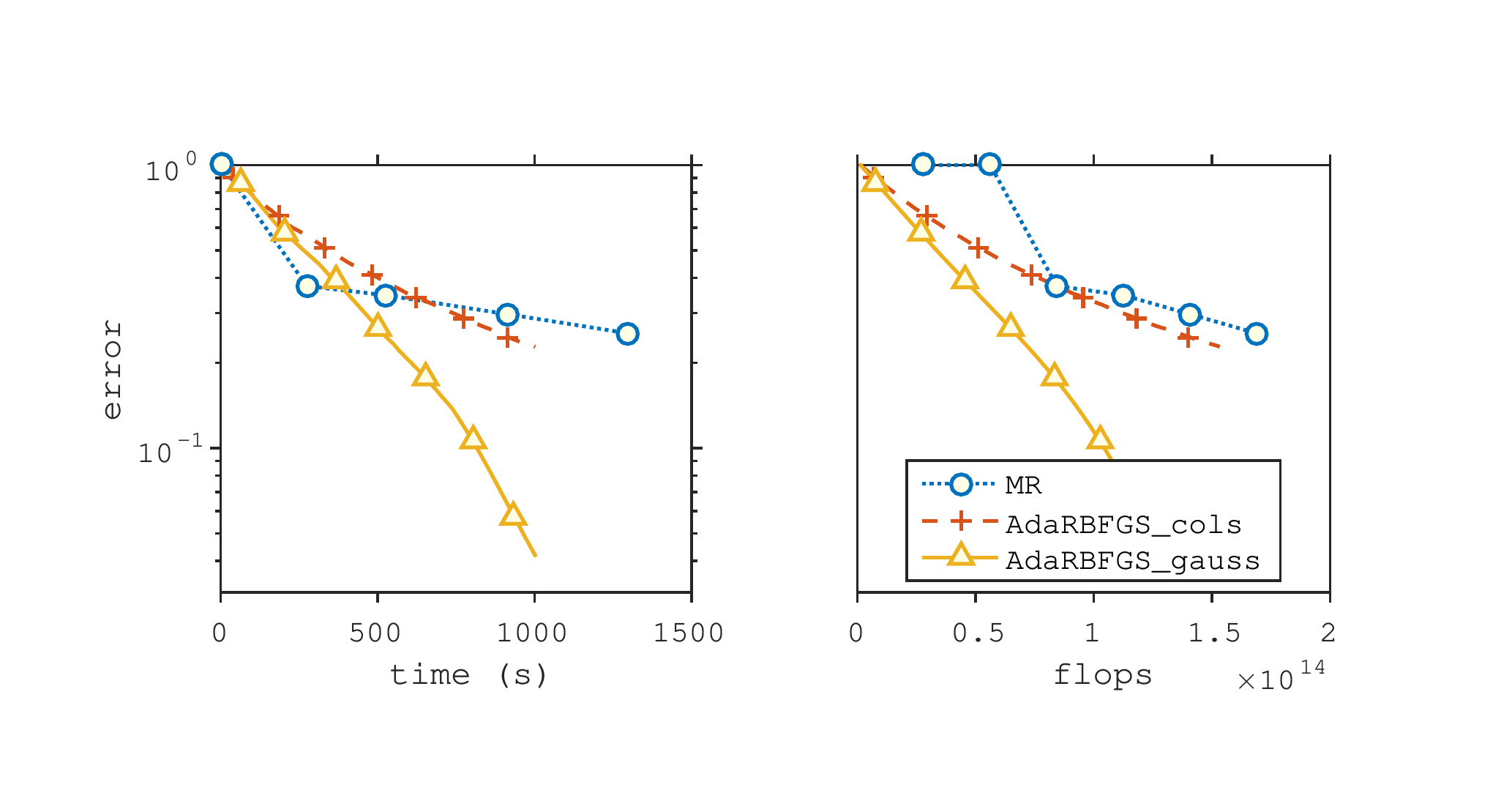}
        \caption{\texttt{GHS\_psdef/wathen100}}
\end{subfigure}%
  \caption{The performance of Newton-Schulz, MR, AdaRBFGS\_gauss and AdaRBFGS\_cols on 
(a) \texttt{Bates-Chem97ZtZ}: $n= 2\,541$,   
  (b) \texttt{FIDAP/ex9}: $n = 3,\,363 $, (c) \texttt{Nasa/nasa4704}: $n= 4\,,704$, 
  (d) \texttt{HB/bcsstk18}: $n=11,\,948$,   (e) \texttt{Pothen/bodyy4}: $n =17,\,546$ (f) \texttt{ND/nd6k}: $n=18,\,000$ (g) \texttt{GHS\_psdef/wathen100}: $n = 30, \,401$. The starting matrix $X_0=I$ was used for all methods. } \label{fig:UF2}
  \end{figure}

\end{document}